\documentclass[a4 paper,11pt]{amsart}

\usepackage{amsfonts}
\usepackage{amsthm}
\usepackage{amssymb}
\usepackage{amsmath}
\usepackage[utf8]{inputenc}
\usepackage[english]{babel}
\usepackage[T1]{fontenc}
\usepackage{xcolor}
\usepackage{enumitem}
\usepackage{manfnt} 
\usepackage{hyperref}
\usepackage{graphicx}
\usepackage{quiver}
\usepackage{mathrsfs}
\usepackage{dsfont}
\usepackage{stmaryrd}

\title{Aubert duality and co-tempered Langlands data}
\author{Nans BONNEL}
\address{Institut Mathématiques de Jussieu-Paris Rive Gauche\\ Sorbonne Université\\ Campus Pierre et Marie Curie\\
4, place Jussieu, 75252 Paris Cedex 05}
\email{bonnel@imj-prg.fr}

\def\X{{\rm X}}

\newcommand{\s}{\mathfrak{s}}
\newcommand{\m}{\mathfrak{m}}

\newcommand{\e}{\varepsilon}
\newcommand{\Cusp}{\mathscr{C}}

\newcommand{\Data}{\mathrm{Data}}

\newcommand{\la}{\lambda}

\newcommand{\GL}{\mathrm{GL}}

\newcommand{\SO}{\mathrm{SO}}
\newcommand{\Sp}{\mathrm{Sp}}

\newcommand{\Z}{\mathbb{Z}}
\newcommand{\R}{\mathbb{R}}

\newcommand{\Seg}{\mathrm{Seg}}

\newcommand{\Irr}{\mathrm{Irr}}

\newif\ifdraft\drafttrue
\ifdraft
\newcommand{\comment}[1]{\marginpar{\textcolor{red}{\dbend}}\textcolor{red}{#1}}
\newcommand{\commentlatex}[1]{\marginpar{\textcolor{brown}{\dbend}}\textcolor{brown}{#1}}

\else
\newcommand{\comment}[1]{}
\newcommand{\commentlatex}[1]{}
\fi

\theoremstyle{plain}
\newtheorem{deff}{Definition}[subsection]

\theoremstyle{remark}
\newtheorem{rem}[deff]{Remark}

\theoremstyle{remark}
\newtheorem{ex}[deff]{Example}

\theoremstyle{plain}
\newtheorem{prop}[deff]{Proposition}

\theoremstyle{plain}
\newtheorem{lem}[deff]{Lemma}

\theoremstyle{plain}

\newtheorem{theorem}[deff]{Theorem}

\theoremstyle{plain}
\newtheorem{coro}[deff]{Corollary}

\newcommand{\AD}{\mathrm{AD}}

\newcommand{\ADp}{\mathrm {AD}_\rho}
\newcommand{\ATemp}{\mathrm{CTemp}}
\newcommand{\Temp}{\mathrm{Temp}}
\newcommand{\Mult}{\mathrm{Mult}}
\newcommand{\DMult}{\mathrm{DMult}}
\newcommand{\Symm}{\mathrm{Symm}}
\newcommand{\DSymm}{\mathrm{DSymm}}
\newcommand{\SSymm}{\mathfrak S \mathrm{ymm}}
\newcommand{\trans}{\mathrm{trans}}
\newcommand{\T}{\mathrm{T}}
\newcommand{\CT}{\mathrm{CT}}

\newcommand{\Rep}{\mathrm{Rep}}

\newcommand{\USymm}{\mathfrak{S}ymm}
\newcommand{\USeg}{\mathfrak{S}eg}

\usepackage{csquotes}
\usepackage[style=alphabetic, maxnames=99, maxalphanames=6, minalphanames=6]{biblatex}

\everymath{\displaystyle}

\begin{document}

\begin{abstract}Let $F$ be a non-archimedean local field of characteristic 0, and let $G$ be either $\Sp_{2n}(F)$ or $\SO_{2n+1}(F)$. We introduce a new algorithm to compute the Aubert dual at the level of Langlands data. This algorithm acts as the dual to the recent Lanard-Mínguez algorithm. It fundamentally differs in two ways: it follows a bottom-up approach rather than a top-down one, and its internal computations strictly preserve the temperedness of the representations. Consequently, this approach naturally yields a new constructive characterization of co-tempered representations.

By operating exclusively within the realm of tempered data, this algorithm enables inductive proofs of new properties for co-tempered representations. In particular, we provide a precise description of their tempered components and establish an explicit duality formula for a large class of tempered representations. 
\end{abstract}

\maketitle

\tableofcontents
\section{Introduction}

In \cite{A}, Aubert defined an involution on the set of irreducible representations for any connected reductive group over a non-archimedean local field $F$ of characteristic 0. This involution is particularly important for computing the wavefront sets of representations (see \cite{CMBO2, HLLS}). Furthermore, it allows us to define a crucial class of representations: we call \textit{co-tempered} the representations that are the Aubert duals of tempered representations, which play an essential role in Arthur's classification (see \cite{ART, ATIKMS}) of classical groups. 

While tempered representations possess a very natural and clear characterization in terms of Langlands data, co-tempered representations lack such a direct description. Since the Aubert involution is intrinsically defined on the representation-theoretic side, a major challenge is to translate it into a combinatorial operation at the level of Langlands data in order to explicitly understand the co-tempered spectrum.

For $\GL_n(F)$, M{\oe}glin and Waldspurger introduced in \cite{MW86} a combinatorial algorithm to compute the dual of a representation using its Langlands parameter. Through this algorithm, one easily recovers a remarkably simple and direct structural description of co-tempered representations that already follows from Zelevinsky's work (see \cite[9.3, 9.5, 8.4]{zel}): their multisegments consist exclusively of segments of length 1.

Later, Atobe and Mínguez gave in \cite{AM} an algorithm to compute the dual for classical groups. Finally, Lanard and Mínguez in \cite{LM} introduced a symmetrical version of Langlands data for classical groups on which they formulated a powerful descendant algorithm analog to the one from M{\oe}glin and Waldspurger. However, the direct combinatorial characterization of co-tempered representations fails for classical groups, and even the Lanard-Mínguez algorithm does not yield a straightforward structural description of them.

The core of this article is to compute the Aubert duals of some natural operators that build Langlands data segment by segment for the classical groups $\Sp_{2n}(F)$ and $\SO_{2n+1}(F)$. By defining operators that add a segment, together with a sign if it is centered, we can generate all Langlands data from the empty datum. Constructing the exact duals of these operators provides an ascendant algorithm for the Aubert involution. This new algorithm relates in a very specific way to the Lanard-Mínguez algorithms; it can be viewed as an ascendant dual of these descendant algorithms. This provides a new, simpler, constructive description of co-tempered data.

A major advantage of this algorithm is that it preserves the temperedness along the computations, thus enabling inductive reasoning on the set of tempered data. By diving deep into the combinatorics of this new algorithm, we are able to prove new results about co-tempered data, such as an explicit description of their tempered part and even a direct formula for the Aubert duality on a broad class of tempered data.

Let us be more specific now.

\subsection{The case of $\GL_n$}
Let $F$ be a non-archimedean local field of characteristic 0.

In the case of $\GL_n(F)$, according to Zelevinsky (\cite[Theorem 6.1]{zel}), the irreducible representations are parameterized by multisegments. If $\pi$ is an irreducible representation and $\m$ its corresponding multisegment, we write $\m^t$ for the multisegment associated with the Aubert dual $\hat \pi$ of $\pi$.

On the set of segments we have two variations of the lexicographic order $>$ and $>'$ that are symmetric to each other.  We say that $[x,y]> [x',y']$ if $x>x'$ or if $x=x'$ and $y<y'$ and we say that $[x,y]>' [x',y']$ if $y>y'$ or if $y=y'$ and $x<x'$.

The key idea for our new algorithm is to define, for any segment $s=[b,e]$, an operator $\T_s$ that adds one copy of $s$ to a multisegment containing only segments smaller than $s$ with respect to the order $>'$.

This simple operator allows us to reformulate the M{\oe}glin-Waldspurger algorithm in the following way: from $\m$, the algorithm extracts a segment $s$ and a multisegment $\m^\#$ such that every segment in $\m^\#$ is smaller than $s$ and $\m^t=\T_s((\m^\#)^t)$.

We define an other operator $\CT_s$ and we show in Proposition \ref{prop duality GLn} that it is the dual of the additive operator $\T_s$: by that we mean that $\CT_s=\T_s((\cdot)^t)^t$.

Since any multisegment can be generated by sequentially applying these additive operators, the operators $\CT_s$ naturally yields an ascendant algorithm to compute the Aubert dual of a datum. Indeed, if $\m=\T_{s_n}\circ\ldots\circ\T_{s_1}(\varnothing)$ then $\m^t=\CT_{s_n}\circ\ldots\circ\CT_{s_1}(\varnothing)$.

To define $\CT_s(\m)$, we first extract an increasing sequence $\Theta=(\Theta_1,\ldots,\Theta_l)$ of segments from $\m$. This sequence is a minimal increasing sequence for the order $>$ such that the first segment $\Theta_1$ ends in $b-1$ and the endings of the segments $\Theta_i$ are increasing by one (so that $\Theta_i$ ends in $b-2+i$). We then construct $\CT_s(\m)$ by extending the segments of $\Theta$ by one on the right and, where necessary, appending additional segments of length one.

This operator is designed such that when we apply the M{\oe}glin-Waldspurger algorithm to $\CT_s(\m)$, the segments that we have to shorten are exactly the ones we have extended or added, effectively making $\CT_s$ a section of the map $\m \mapsto \m^\#$. This implies indeed that $\CT_s$ is the dual of $\T_s$ and this is the way we prove it.

In the $\GL_n(F)$ case, it is already well established that co-tempered multisegments are the ones containing exclusively segments of length one (subject to a symmetry condition). In the remainder of the paper, our goal is to adapt this construction to classical groups, for which no such direct characterization of co-tempered data currently exists.

\subsection{Langlands data for classical groups}
Let $G_n$ denote either the symplectic group $\Sp_{2n}(F)$ or the split special orthogonal group $\SO_{2n+1}(F)$. We write $\Irr(G)$ for the disjoint union of the sets of irreducible representations of $G_n$ over all $n \ge 0$ (keeping the series of groups either strictly orthogonal or strictly symplectic).

According to the Langlands subrepresentation theorem, any representation $\pi \in \Irr(G)$ is uniquely determined by a tempered representation $\pi_{\mathrm{temp}}$ and a multisegment $\mathfrak{n}$ in which every segment has a strictly negative center. Furthermore, Arthur's classification \cite{ART} parameterizes such tempered representations by a Langlands parameter $\phi$ and a character $\eta$ of the component group. Consequently, to any $\pi \in \Irr(G)$, we can associate a unique triplet $(\mathfrak{n}; \phi, \eta)$, where $\eta$ is subject to a certain positivity condition. Such a triplet is called a \textit{Langlands datum}, and we denote the set of all such data by $\Data(G)$.

\subsection{Symmetrical Langlands  data and the Lanard-Mínguez algorithm}
For classical groups, an initial algorithm to compute the Langlands data of the dual $\hat{\pi}$ was developed in \cite{AM}. To refine this approach and establish a closer analogue to the M{\oe}glin-Waldspurger algorithm, Lanard and Mínguez \cite{LM} introduced the notion of \textit{symmetrical Langlands data} (with signs). A symmetrical Langlands datum is a multisegment where every segment appears with the same multiplicity as its reflection across $0$, equipped with a sign function on the centered segments.

By symmetrizing $\mathfrak{n}$ and extracting centered signed segments from $(\phi, \eta)$, one obtains a natural bijection between the standard data $\Data(G)$ and the set of symmetrical data, denoted $\Symm^\e(G)$. Thus, the tempered data correspond to the symmetrical data containing only centered segments. Through this identification, Lanard and Mínguez formulated a new descendant algorithm in the same spirit as the M{\oe}glin-Waldspurger for $\GL_n(F)$. This result serves as the starting point for our work.

\subsection{Decomposition on supercuspidal lines}
Any segment within a symmetrical datum is supported on a supercuspidal line $\Z_\rho$, where $\rho$ is a supercuspidal representation of some $\GL_m(F)$. This allows us to cleanly decompose the space of symmetrical data as a direct sum 
$$\Symm^{\e}(G) \simeq \bigoplus_{\rho} \Symm_{\rho}^{\e}(G)$$

Crucially, the Aubert involution preserves the supercuspidal support; thus, it acts independently on each factor (see, for instance, \cite[Section 4]{LM}). We can therefore restrict our study to the duality map on a single supercuspidal line, denoted $\ADp: \Symm^{\e}_\rho(G) \to \Symm^{\e}_\rho(G)$. The behavior of this map depends heavily on the \textit{parity} of $\rho$, which can be \textit{good}, \textit{bad}, or \textit{ugly} (see Definition \ref{def:good/bad/ugly}). Let us now fix such a $\rho$. There exists a unique $x \geq 0$ such that $\rho_u:=\rho | \cdot|^x$ is unitary. We call $\rho_u$ the unitarization of $\rho$ and $\rho^\vee$ its contragredient.

\subsection{The bad and ugly parity cases}

In the bad and ugly parity cases, symmetrical data do not involve signs, rendering the combinatorial situation remarkably similar to that of $\GL_n(F)$.

Specifically, the ugly parity case follows directly from that of $\GL_n(F)$. While the bad parity case cannot be deduced directly, our construction for $\GL_n(F)$ adapts to it effortlessly. Both cases yield a characterization of co-tempered data (see Propositions \ref{thm ugly} and \ref{thm bad}) which are highly analogous to the $\GL_n(F)$ setting (Lanard-Mínguez algorithm was sufficient for these characterizations).

We note that Hazeltine and Lo \cite{HL} related the Lanard-Mínguez algorithm for bad parity to the Pyasetskii involution, providing a geometric interpretation of this process.
 
The true combinatorial challenge arises in the good parity case.

\subsection{The good parity case and the dual operators}

The good parity case is significantly more challenging, as signs play a key role in the computation of the Aubert dual. However, we follow the exact same strategy as for $\GL_n(F)$, substituting the set of multisegments with $\Symm_\rho^\e(G)$ and the M{\oe}glin-Waldspurger algorithm with that of Lanard-Mínguez. The added complexity is therefore primarily technical. For the remainder of the paper, we assume $\rho$ is of good parity, restrict our study to $\Symm_\rho^\e(G)$, and drop $\rho$ from the subscripts.

Let $\s$ denote either a segment $s$ with a strictly positive center, or a signed centered segment $(s,\eta)$. The order $<'$ defined on the segments extend to the set of such $\s$ by forgetting the eventual sign.

This allows us to define the bounded subset of symmetrical data $\Symm^{\le' \s}$ (consisting of data whose positive and centered segments are all bounded by $\s$ with respect to $\le'$), as well as its dual space $\DSymm^{\le' \s} = \AD(\Symm^{\le' \s})$.

On the set $\Symm^{\le' \s}$, we define the adding operator $\T_{\s}$, which maps $(\m, \e)$ to $(\m, \e) + \tilde{\s}$, where $\tilde{\s} = s+s^\vee$ if $\s = s$ has a strictly positive center, and $\tilde{\s} = \s$ otherwise. Clearly, any symmetrical Langlands datum $(\m,\e)$ can be uniquely generated from the empty datum via a sequence of these operators:
$$(\m,\e) = \T_{\s_n}\circ\ldots\circ \T_{\s_1}(\varnothing)$$

This observation naturally leads to the following strategy: if we can find an explicit formula for the conjugated operator $\AD\circ \T_\s\circ \AD: \DSymm^{\le' \s} \to \DSymm^{\le' \s}$, we obtain a complete algorithm to compute the dual $\AD(\m,\e)$ for any $(\m,\e) \in \Symm^\e(G)$. Indeed, we have:
$$\AD(\m,\e)=(\AD\circ \T_{\s_n}\circ \AD)\circ \ldots\circ(\AD\circ \T_{\s_1}\circ \AD)(\varnothing)$$

In the Subsection \ref{subsec def CT} we define, for any such $\s$, an operator $\CT_{\s}: \DSymm^{\le' \s} \to \DSymm^{\le' \s}$. The main result is then to prove that it is the dual of $\T_\s$:

\begin{theorem}[Theorem \ref{main theorem}]\label{main theorem intro}
Let $\s \in \Seg^{>0} \cup (\Seg^{=0}\times\{1,-1\})$. We have the following correspondence of operators under Aubert duality:
$$\CT_{\s}\stackrel{\AD}{\longleftrightarrow} \T_{\s}$$
By this, we mean that the following square is commutative:
\[\begin{tikzcd}[sep=small]
	{\Symm^{\le' \s}} && {\DSymm^{\le' \s}} \\
	\\
	{\Symm^{\le' \s}} && {\DSymm^{\le' \s}}
	\arrow["{\AD}", tail reversed, from=1-1, to=1-3]
	\arrow["{\T_{\s}}"', from=1-1, to=3-1]
	\arrow["{\CT_{\s}}", from=1-3, to=3-3]
	\arrow["{\AD}"', tail reversed, from=3-1, to=3-3]
\end{tikzcd}\]
\end{theorem}

Restricting this process exclusively to operators $\T_\s$ where $\s$ is a signed centered segment generates the entire tempered spectrum. By duality, this provides a constructive description for the co-tempered spectrum:

\begin{coro}[Theorem \ref{coro pairing temp}]\label{coro pairing temp intro}
If $\rho$ is of good parity, then for any non-empty co-tempered symmetrical Langlands datum $(\m,\e)$, there exists a unique tuple of centered segments $(s_i,\eta_i)_{1\leq i \leq r}$ satisfying $s_1 \subseteq \dots \subseteq s_r$ (where $s_i=s_{i+1}$ implies $\eta_i=\eta_{i+1}$), such that:
$$(\m,\e)=\CT_{s_r,\eta_r}\circ\dots\circ \CT_{s_1,\eta_1}(\varnothing,+)$$
Moreover, in this case, we have $\AD(\m,\e)=\sum_{i=1}^r(s_i,\eta_i)$.
\end{coro}

\subsection{Strategy of proof and link with the Lanard-Mínguez algorithm}

We now outline the underlying strategy for constructing $\CT_{\s}$. Given a symmetrical datum $(\m,\e)$, the Lanard-Mínguez algorithm isolates a term $(\m_1,\e_1)$ and a residual datum $(\m^\#,\e^\#)$ such that $\AD(\m,\e)=(\m_1,\e_1)+\AD(\m^\#,\e^\#)$.

As shown in \cite{LM}, the extracted datum $(\m_1,\e_1)$ always takes the form of $\tilde{\s}$ for some maximal element $\s$ with respect to $\le'$. In terms of our operators, this relation translates exactly to $\AD(\m,\e)=\T_{\s}(\AD(\m^\#,\e^\#))$.

This fundamental observation implies that the map $(\m,\e) \mapsto (\m^\#,\e^\#)$ admits sections parameterized by elements $\s$. Our objective is to construct $\CT_{\s}$ precisely as the section associated with $\s$. The proof strategy then emerges naturally: applying our operator $\CT_\s$ followed by the Lanard-Mínguez algorithm must bring us exactly back to the starting point.

This mechanism is precisely what we mean when characterizing our approach as the ascendant dual of the Lanard-Mínguez algorithm. Below, we summarize the key symmetries and differences between the two algorithms (note that this proof strategy and comparison also hold for $\GL_n(F)$ and the bad and ugly parities):

\begin{itemize}
    \item \textbf{Combinatorial action:} Our definition of $\CT_\s$ relies on an \textit{increasing} sequence of segments $(\Theta_i)_{1\le i\le l}$ that must be extended, whereas the $\#$ operation in the Lanard-Mínguez algorithm involves a \textit{decreasing} sequence of segments $(\Delta_i)_{1\le i\le k}$ that must be shortened.
    \item \textbf{Algorithmic flow:} Our algorithm is \textit{ascendant}, meaning we successively compute the duals of larger and larger data to build the final result. In contrast, the Lanard-Mínguez algorithm is \textit{descendant}, recursively computing the duals of smaller and smaller data.
    \item \textbf{Preservation of structure:} Our algorithm preserves \textit{temperedness}: computing the dual of a tempered datum only requires computing the duals of smaller tempered data. Conversely, the Lanard-Mínguez algorithm preserves \textit{co-temperedness}: computing the dual of a co-tempered datum only involves computing the duals of smaller co-tempered data.
    \item \textbf{Complexity:} For a given datum $(\m,\e)$, our algorithm requires $n$ steps, where $n$ is the number of positive and centered segments in $(\m,\e)$. The Lanard-Mínguez algorithm requires $m$ steps, where $m$ is the number of such segments in its dual $\AD(\m,\e)$. Since a tempered datum consists of fewer segments than its co-tempered dual, our algorithm is significantly faster for tempered data, while the Lanard-Mínguez approach is faster for co-tempered data. Executing one step for both algorithm is comparable.
\end{itemize}
\subsection{Applications}

Finally, by diving deeper into the combinatorics of our algorithm (more precisely, its restriction to tempered data, which is simpler), we uncover new properties about co-tempered representations and their pairing with tempered representations. First, we provide a complete characterization of the set $\mathcal{T}$, which consists of all possible tempered components of co-tempered data (Theorem \ref{thm im phi=T}).

Next, we introduce a reduction process $(\m,\e)\mapsto(\m^{red},\e^{red})$ for tempered symmetrical data. This reduction simplifies the computation of the dual by restricting it to multisegments where no three consecutive segments share the same sign. Building on this, we define the class $\Temp^d$ of \textit{decreasing} tempered data. Obtaining a direct, closed-form characterization of co-tempered data from our inductive algorithm is highly non-trivial, as it requires tracking the evolution of the sequence $\Theta$ step by step. The class $\Temp^d$ overcomes this combinatorial obstacle and although its representation-theoretic meaning is not yet known, $\Temp^d$ provides an ideal setting where this tracking is fully resolved, yielding an explicit formula for the Aubert dual (Theorem \ref{thm explicit tempd}). Furthermore, this class is remarkably rich: it allows us to directly construct a large family of co-tempered data whose tempered components exhaust the entire set $\mathcal{T}$.

\subsection{Wavefront sets}

The wavefront set of a representation is a particularly important invariant, and its computation often relies on the Aubert duality (for an upper bound see \cite{HLLS}, \cite{AtoCiu}, or for an equality see \cite{CMBO2}). Thus, our algorithm also provides a new way to compute the wavefront set (or at least an upper bound) in many cases.

One of our motivations for this algorithm was to prove a formula for the geometric wavefront set of tempered unipotent representations of $\SO_{2n+1}(F)$. Waldspurger showed in \cite{Waldtemp} that the wavefront set of these representations is a singleton, and for some of them, he computed it explicitly in terms of partitions.

In \cite[Subsection 11.2]{HLLS}, Hazeltine, Liu, Lo, and Shahidi conjectured that the wavefront set of such a representation $\pi$ is parameterized by the partition $d_{BV}(\underline p(\phi_{\hat \pi}))$, where $d_{BV}$ is the Barbasch-Vogan duality map on partitions and $\underline p(\phi_{\hat \pi})$ is the partition of lengths of the $L$-parameter attached to the Aubert dual of $\pi$.

Since we have $\phi_{\hat \pi}=\AD(\phi_\pi)$, we can try to use our algorithm to prove by induction that this quantity equals the one computed by Waldspurger; the preservation of temperedness being crucial for such an induction. 

Independently, Antor and Okada recently proved in \cite{AntorOkada} a compatibility result between the Iwahori-Matsumoto involution and the Aubert duality. The previous conjecture (in its general case) then follows from \cite[Theorem 5.7]{AntorOkada} combined with the computations in \cite[Section 7]{La}, which show that the wavefront set of such a $\pi$ is $d_{BV}(\underline p(\phi_{\mathrm{IM}( \pi)}))$, where $\mathrm{IM}(\pi)$ is the Iwahori-Matsumoto dual of $\pi$.

\subsection*{Structure of the paper} 
In Section \ref{section preliminaries}, we set up the necessary background and notation. Section \ref{section GLn} is devoted to the cases of $\GL_n$ and bad/ugly parity. The remainder of the paper focuses exclusively on the good parity case. In Section \ref{section good parity}, we explicitly construct the dual operators $\CT_{\s}$. The proof of our main result, Theorem \ref{main theorem}, is carried out in Section \ref{section proof}. Finally, in Section \ref{section implications on anti-temp data}, we give the simplified version of our algorithm when restricted to tempered data and Theorem \ref{coro pairing temp intro}, then we delve into the combinatorics of $\CT_{\s}$ and derive the aforementioned applications of our algorithm.

\subsection*{Acknowledgments} 
The author is deeply grateful to Alberto Mínguez for suggesting this problem during a 9-month internship at the University of Vienna, and thanks the university for providing excellent working conditions. Special thanks are also due to Thomas Lanard for generously sharing his time to discuss intricate details and for his careful reading of early drafts of this manuscript. Finally the author would also like to thank Emile Okada for a recent discussion on his recent joint paper with Jonas Antor \cite{AntorOkada} and its applications to the wavefront set of tempered unipotent representations of $\SO_{2n+1}(F)$.

\section{Preliminaries}\label{section preliminaries}

In order to state our main results, we first set up the necessary notation. We follow closely the formalism developed in \cite[Sections 2--4]{LM}. As our exposition will be brief, we refer the reader to their article for a comprehensive introduction to these objects.

Throughout this article, $F$ denotes a non-archimedean local field of characteristic 0 with absolute value $|\cdot|$. Let $\textbf{G}$ be a connected reductive group defined over $F$, and let $G=\textbf{G}(F)$ be its group of $F$-points. Let $\Rep(G)$ be the category of smooth complex representations of $G$ of finite length, and let $\Irr(G)$ denote the set of its irreducible objects. For any $\pi \in \Rep(G)$, we denote its contragredient by $\pi^\vee$. Finally, for $\pi \in \Irr(G)$, let $\hat{\pi}$ denote the Aubert dual of $\pi$.

\subsection{Supercuspidal representations of $\GL_n(F)$}\label{subsection supercusp of GL}

We start by defining $$\Irr^{\GL}:= \bigcup_{n \geq 0} \Irr(\GL_n(F))$$ and let $\Cusp^{\GL} \subset \Irr^{\GL}$ denote the subset of supercuspidal representations.

The equivalence relation on $\Cusp^\GL$ generated by $\rho \sim \rho|\cdot|$ partitions $\Cusp^\GL$ into equivalence classes called \emph{supercuspidal lines}. The supercuspidal line containing $\rho \in \Cusp^\GL$ is given by
$$ \Z_\rho:= \{\rho |\cdot|^n \mid n \in \Z\}. $$
We denote the set of all such equivalence classes by $\Cusp^\GL/{\sim}$.

We say that two supercuspidal representations $\rho,\rho'\in \Cusp^\GL$ are \emph{line equivalent}, denoted by $\rho \sim' \rho'$, if $\rho \sim \rho'$ or $\rho \sim {\rho'}^\vee$. Equivalently, $\rho \sim' \rho'$ if $\Z_\rho \cup \Z_{\rho^\vee} = \Z_{\rho'} \cup \Z_{{\rho'}^\vee}$.

\subsection{Segments}

A \emph{segment} $\Delta$ is a finite, non-empty subset of $\Cusp^\GL$ of the form
$$ \Delta = \{\rho |\cdot|^x, \rho |\cdot|^{x+1}, \dots, \rho |\cdot|^y\}, $$
where $\rho \in \Cusp^\GL$ and $x, y \in \R$ with $y - x \in \mathbb N$. We denote such a segment by $[x,y]_\rho$. Note that $[x,y]_\rho = [x',y']_{\rho'}$ if and only if $\rho|\cdot|^x = \rho'|\cdot|^{x'}$ and $\rho|\cdot|^y = \rho'|\cdot|^{y'}$; hence, one can assume that $\rho$ is unitary whenever necessary.

We denote by $\Seg$ the set of all segments. Let $\rho \in \Cusp^\GL$ be unitary and let $\Delta=[x,y]_\rho \in \Seg$. We define:
\begin{itemize}
    \item the \emph{beginning} $b(\Delta):= x$ and the \emph{end} $e(\Delta):= y$;
    \item the \emph{length} $\ell(\Delta):= y - x + 1$;
    \item the \emph{degree} $\deg(\Delta):= \deg(\rho)\ell(\Delta)$;
     \item the \emph{center} $c(\Delta):= (x+y)/2 \in \R$;
    \item the \emph{determinant} $\det(\Delta):= \det(\rho)^{\ell(\Delta)} |\cdot|^{\ell(\Delta)c(\Delta)}$.
\end{itemize}

We denote by $\Seg^{>0}$ (resp. $\Seg^{0}$, $\Seg^{<0}$) the subset of $\Seg$ consisting of segments $\Delta$ such that $c(\Delta)>0$ (resp. $c(\Delta)=0$, $c(\Delta)<0$). We say that $\Delta$ is \emph{centered} if $c(\Delta)=0$.

We define the following operations on a segment $\Delta = [x, y]_\rho$:
\begin{align*}
\Delta^- &= [x, y - 1]_\rho, & {}^-\Delta &= [x + 1, y]_\rho, \\
\Delta^+ &= [x, y + 1]_\rho, & {}^+\Delta &= [x - 1, y]_\rho, \\
\Delta^\vee &= [-y, -x]_{\rho^\vee}.
\end{align*}
By convention we define $\Delta^-$ and ${}^-\Delta$ to be the empty set $\varnothing$ if $\ell(\Delta) > 1$.

\begin{deff}\label{def orders on seg}
We define two orders on $\Seg$:
\begin{itemize}
    \item We say that $[x,y]> [x',y']$ if $x>x'$ or if $x=x'$ and $y<y'$.
    \item We say that $[x,y]>' [x',y']$ if $y>y'$ or if $y=y'$ and $x<x'$.
\end{itemize}
\end{deff}
These two orders are symmetric with respect to each other.

\subsection{Multisegments}\label{subsect multiseg}

Given a set $\X$, let $\mathbb{N}(\X)$ denote the commutative semigroup of maps from $\X$ to $\mathbb{N}$ with finite support. A \emph{multisegment} is a multiset of segments, i.e., an element of $\Mult:= \mathbb{N}(\Seg)$. We view a multisegment $\m$ as a finite sum $\m = \Delta_1 + \dots + \Delta_N$ with $\Delta_i \in \Seg$. We denote the multiplicity of a segment $\Delta$ in $\m$ by $m_{\m}(\Delta)$. The definitions of contragredient, length, and degree are extended from segments to multisegments by linearity. For $\m = \Delta_1 + \dots + \Delta_N \in \Mult$, we define its \emph{support} as the multiset $\Delta_1 \sqcup \dots \sqcup \Delta_N \in \mathbb{N}(\Cusp^{\GL})$ and its \emph{determinant} as 
$ \det(\m):=\det(\Delta_1)\times \dots \times \det(\Delta_N) $.

We say that $\m$ is \emph{positive} (resp. \emph{negative}) if $c(\Delta_i) > 0$ (resp. $c(\Delta_i) < 0$) for all $i=1, \dots, N$. We denote by $\Mult^{\clubsuit}$ the subset of $\Mult$ consisting of multisegments whose segments all belong to $\Seg^{\clubsuit}$, where $\clubsuit \in \{>0, 0, <0\}$. The natural projection $\Mult \to \Mult^{>0} \times \Mult^0 \times \Mult^{<0}$ is denoted by $\m \mapsto (\m^{>0}, \m^0, \m^{<0})$.

For $\rho \in \Cusp^\GL$, let $\Mult_\rho$ be the submonoid of multisegments supported in $\Z_\rho$, and let $\Irr^\GL_\rho$ be the set of irreducible representations with supercuspidal support in $\Z_\rho$. We set $\Mult_\rho^{\clubsuit}:= \Mult_\rho \cap \Mult^{\clubsuit}$. There is a natural map $\Mult \to \Mult_\rho$, denoted $\m \mapsto \m_\rho$, which sends $\m$ to the sum of its segments supported in $\Z_\rho$.

According to Zelevinsky \cite{zel}, there is a bijection between $\Irr^{\GL}$ and $\Mult$. To classify irreducible representations of $\Sp_{2n}(F)$ and $\SO_{2n+1}(F)$, we must introduce signs. A \emph{signed centered segment} is a pair $(\Delta, \varepsilon)$ where $\Delta \in \Seg^0$ and $\varepsilon \in \{-1, 1\}$. Formally, we define $\Mult^\varepsilon$ as the set of maps 
$$ \Seg^{<0} \cup \Seg^{>0} \cup (\Seg^0 \times \{-1, 1\}) \longrightarrow \mathbb{N} $$
with finite support. For $\mathfrak{s} \in \Mult^\varepsilon$, the \emph{underlying multisegment} $\m \in \Mult$ is obtained by forgetting the signs. We usually write $\mathfrak{s} = (\m, \varepsilon)$, where for each $\Delta \in \m^0$, $\varepsilon(\Delta) \in \{-1, 1\}$ is the sign assigned to $\Delta$.

\subsection{Zelevinsky classification }Let $\Irr^\GL=\bigcup_{n\ge 0}\Irr(\GL_n(F))$.

The Zelevinsky classification (see \cite[Theorem 6.1]{zel}) provides a bijection $L:\Mult\to\Irr^\GL$ that sends the set of multisegments containing only centered segments on the set of tempered representation.

In \cite{MW86}, M{\oe}glin and Waldspurger defined recursively an involutive map $\m\mapsto \m^t$. Thus they proved the following result:

\begin{theorem}[\cite{MW86}]
    For any $\m\in\Mult$, we have:
    $$L(\m^t)=\widehat{L(\m)}$$
\end{theorem}

\subsection{Langlands data for classical groups}

Throughout this paper, $G_n$ denotes either the split special orthogonal group $\SO_{2n+1}(F)$ or the symplectic group $\Sp_{2n}(F)$ of rank $n$. We fix this choice of type throughout the article. We define
$$ \Irr^G:= \bigcup_{n \geq 0} \Irr(G_n), $$
where the union is taken over groups of the fixed type.

A Langlands datum for $G$ is a triple $(\mathfrak{n}; \phi, \eta)$ where $\phi$ is a Langlands parameter for $G$, $\eta$ is a character of the component group $\mathcal{S}_\phi$, and $\mathfrak{n} \in \Mult^{<0}$. We denote by $\Data(G)$ the set of all such data, and by $\Data^+(G)$ the subset of triples $(\mathfrak{n}; \phi, \eta)$ satisfying $\eta(z_\phi)=1$, where $z_\phi$ is the central element of the enhanced component group $\mathcal{A}_\phi$. We refer the reader to \cite[Sections 4.3--4.4]{LM} for a detailed exposition.

Using the Langlands subrepresentation theorem and Arthur's classification of tempered representations for classical groups, we have the following result:

\begin{theorem}[Arthur-Langlands]\label{thm L data}
To any datum $(\mathfrak{n}; \phi, \eta) \in \Data^+(G)$, one can associate an irreducible representation $L(\mathfrak{n}; \phi, \eta) \in \Irr^G$. This association induces a bijection between $\Data^+(G)$ and $\Irr^G$. Finally, this representation $L(\mathfrak n;\phi,\eta)$ is tempered if and only if $\mathfrak{n} = 0$.
\end{theorem}

\subsection{Symmetrical Langlands data of classical groups}

Let $\Symm \subseteq \Mult$ be the set of symmetrical multisegments, defined by $\Symm:= \{ \m \in \Mult \mid \m^{\vee} = \m \}$. We define $\Symm^\e \subseteq \Mult^\e$ as the subset of signed multisegments whose underlying multisegment belongs to $\Symm$. Let $\Symm^\e(G)$ be the subset of $\Symm^\e$ consisting of elements $(\m, \e)$ satisfying the following conditions for every pair of centered segments $\Delta, \Delta' \in \m$:
\begin{enumerate}
    \item If $\Delta = \Delta'$, then $\e(\Delta) = \e(\Delta')$.
    \item If $\Delta$ is supported in $\Z_\rho$ with $\rho$ bad or ugly, then $\e(\Delta) = 1$.
    \item If $\Delta$ is supported in $\Z_\rho$ with $\rho$ bad, then the multiplicity of $\Delta$ in $\m$ is even.
\end{enumerate}
Finally, we define $\Symm^{\e,+}(G)$ as the subset of $\Symm^\e(G)$ consisting of elements $(\m, \e)$ such that $\det(\m) = 1$ and, if $\Delta_1, \dots, \Delta_m$ are the centered segments of $\m$ (counted with multiplicity), then $\e(\Delta_1)\times...\times\e(\Delta_m)=1$.

In \cite[Section 4.6]{LM}, Lanard and Mínguez introduced a transfer map to associate a symmetrical Langlands datum with a regular Langlands datum. For $(\mathfrak{n}; \phi, \eta) \in \Data(G)$, we define $\trans(\mathfrak{n}; \phi, \eta) = (\m, \e)$ where
\[ \m = \sum_{s \in \mathfrak{n}} (s + s^\vee) + \sum_{\rho \boxtimes S_a \in \phi} \left[ \frac{-a+1}{2}, \frac{a-1}{2} \right]_\rho, \]
and for any $\rho \boxtimes S_a \in \phi$, the sign is given by
\[ \e\left( \left[ \frac{-a+1}{2}, \frac{a-1}{2} \right]_\rho \right) = \eta(\rho \boxtimes S_a). \]

The map $\trans$ is a bijection from $\Data(G)$ to the subset of $\Symm^\e(G)$ consisting of elements with determinant 1, and it maps $\Data^+(G)$ onto $\Symm^{\e,+}(G)$. Consequently, Theorem \ref{thm L data} induces a bijection between $\Irr^G$ and $\Symm^{\e,+}(G)$. Since $L(\mathfrak{n}; \phi, \eta)$ is tempered if and only if $\mathfrak{n} = 0$, tempered representations in $\Irr^G$ correspond to elements of $\Symm^{\e,+}(G)$ containing only centered segments. For any $\pi = L(\mathfrak{n}; \phi, \eta)$, we denote $\phi^{sym}_\pi:= \trans(\mathfrak{n}; \phi, \eta)$.

\subsection{Good, bad and ugly parity}
 A supercuspidal representation is unitary if and only if its central character is unitary. Therefore, given any supercuspidal representation $\rho$, there exists a unique $x \geq0$ such that $\rho_u:=\rho | \cdot|^x$ is unitary. We call $\rho_u$ the unitarization of $\rho$.

\begin{deff}\label{def:good/bad/ugly}
Let $\rho \in \Cusp^{\GL}$. We write $\rho = \rho_u |\cdot|^x$ with $\rho_u$ unitary and $x \in \R$.
\begin{enumerate}
 \item We say that $\rho$ is \emph{ugly} if $\rho_u$ is not self-dual or $x \notin (1/2)\Z$ (that is, $\mathbb{Z}_\rho \neq \mathbb{Z}_{\rho^\vee}$).
 \item We say that $\rho$ is \emph{good} if $\rho_u$ is self-dual and:
 \begin{itemize}
  \item If $\rho_u$ is of the same type as $G$, then $x \in \Z$.
  \item If $\rho_u$ is of the opposite type as $G$, then $x \in (1/2)\Z \setminus \Z$.
 \end{itemize}
 \item We say that $\rho$ is \emph{bad} if $\rho_u$ is self-dual and:
 \begin{itemize}
  \item If $\rho_u$ is of the same type as $G$, then $x \in (1/2)\Z \setminus \Z$.
  \item If $\rho_u$ is of the opposite type as $G$, then $x \in \Z$.
 \end{itemize}
\end{enumerate}

\end{deff}

\subsection{Lanard-Mínguez algorithm}

For $\rho \in \Cusp^{\GL}$, let $\Symm_{\rho}^{\e}(G)$ be the subset of $\Symm^{\e}(G)$ consisting of elements whose underlying multisegment belongs to $\Mult_\rho$ (if $\rho$ is good or bad) or to $\Mult_\rho \times \Mult_{\rho^{\vee}}$ (if $\rho$ is ugly). This yields a natural decomposition
\[
 \Symm^{\e}(G) \simeq \bigoplus_{\rho \in \Cusp^\GL/{\sim'}} \Symm_{\rho}^{\e}(G).
\]

In \cite{LM}, Lanard and Mínguez defined a map $\ADp: \Symm_\rho^\e(G) \to \Symm_\rho^\e(G)$ recursively. They then defined $\AD:= \oplus_{\rho} \ADp$ and the induced map on Langlands data $\mathbf{AD}:= \trans^{-1} \circ \AD \circ \trans$. Their main result identifies this combinatorial map with the Aubert involution:

\begin{theorem}[\cite{LM}, Theorem 5.4.1]\label{thm lanard Mínguez}
 Let $\pi = L(\mathfrak{n}; \phi, \eta) \in \Irr^G$. Then its Aubert dual is given by
 \[ \widehat{\pi} = L(\mathbf{AD}(\mathfrak{n}; \phi, \eta)). \]
\end{theorem}

Let $\Temp \subset \Symm^\e(G)$ denote the subset of \emph{tempered symmetrical data}, i.e., those whose underlying multisegments contain only centered segments. These correspond via the map $\trans$ to the Langlands data of tempered representations. We define the set of \emph{co-tempered symmetrical data} as $\ATemp:= \AD(\Temp)$. For any $\rho \in \Cusp^\GL$, we let $\Temp_\rho$ and $\ATemp_\rho$ be the respective projections of these sets onto $\Symm^\e_\rho(G)$.

By convention, we treat the empty set $\varnothing$ as a centered segment of length $0$ and multiplicity $1$ in any multisegment. We define the \emph{empty symmetrical Langlands datum} as the pair $(\varnothing, +1)$, denoted by $(\varnothing, +)$ or just $\varnothing$. This datum is tempered and satisfies $\ADp(\varnothing, +) = (\varnothing, +)$, making it also co-tempered.

\section{The cases of $\GL_n$ and the bad and ugly parities}\label{section GLn}

In the cases of bad and ugly parity, signs do not play any role in the Langlands data, and as we can see in the Lanard-Mínguez algorithm, the Aubert duality acts in a very similar way to that of $\GL_n$.

\subsection{The case of $\GL_n$}

First, we recall the M{\oe}glin-Waldpsurger algorithm. Take $\m\in\Mult$. To begin, let $e_{max}$ be the maximum of the ends $e(s)$ of the segments $s\in\m$. Now let $\Delta_1$ be the largest element of $\m$, for the order $>$, such that $e(\Delta_1)=e_{max}$. Then we define recursively the sequence $\Delta_1\geq \cdots\geq \Delta_k$. Suppose $\Delta_i$ is already constructed for an index $i\geq 0$; then we define $\Delta_{i+1}$ as the largest segment in $\m$ such that $e(\Delta_{i+1})=e(\Delta_i)-1$ and $\Delta_{i+1}\leq \Delta_i$.

If no such segment exists in $\m$, then we set $k=i$ and we stop. 

From this sequence, we define:
$$\m_1=[e(\Delta_k),e(\Delta_1)]_{\rho_u}
$$
and
$$\m^\#=\m+\sum_{i=1}^l(\Delta_i^--\Delta_i)$$

That finally allows us to define $\ADp$ recursively by the following formula:
$$\m^t=\m_1+(\m^\#)^t$$

We now define $\Mult^{\le' s}$ as the set of multisegments $\m \in \Mult$ such that every segment $\Delta \in \m$ satisfies $\Delta\le' s$.

For $\m \in \Mult^{\le' s}$, we define the operator $\T_s: \Mult^{\le' s} \to \Mult^{\le' s}$ by:
    $$\T_s(\m) = \m + s$$

Now let $\DMult^{\le's}$ be the dual $(\Mult^{\le's})^t$ of $\Mult^{\le's}$. This definition is not intrinsic because we need to compute the dual of a multisegment to determine whether it belongs to $\DMult^{\le's}$.

However, take $\m\in\Mult$. By applying the M{\oe}glin-Waldspurger algorithm, we obtain a segment $\m_1$ and a multisegment $\m^\#$ such that $\m^t=\m_1+(\m^\#)^t$, the segment $\m_1$ being the largest one in $\m^t$ for the order $<'$. Let us denote this segment $\m_1$ by $s(\m)$.

Then, we do not need to compute $\m^t$ to determine whether it belongs to $\DMult^{\le's}$: it suffices to compute $s(\m)$ and to compare it to $s$:
$$\DMult^{\le's}=\{\m\in\Mult~|~s(\m)\le's\}$$

We now fix a segment $s$ and we define an operator $\CT_s$ on $\DMult^{\le's}$. To do this, we first define an increasing sequence of segments of $\m$.

Define $\Theta^s_1$ as the smallest segment in $\m$ for the order $<$ that ends at $b(s)-1$, and as $\varnothing$ if there are no segments in $\m$ satisfying this condition.

Now we define recursively the sequence $\Theta^s_1\le\ldots\le\Theta^s_l$ with $\Theta^s_i$ being the smallest segment in $\m$ such that $e(\Theta^s_i)=e(\Theta^s_{i-1})+1$ and $\Theta^s_i>\Theta^s_{i-1}$.
If no such segment exists in $\m$, then we set $l=i$ and we stop.
Of course, if $\m=\varnothing$, we have $\Theta_1^s=\varnothing$ (regardless of what $s$ is).
Even if $\Theta^s_i$ depends on the segment $s$, since this choice has been fixed, we will denote it by $\Theta_i$ in what follows.

Finally we define:
$$\CT_s(\m)=\m+\sum_{i=1}^l(\Theta_i^+-\Theta_i)+\sum[j,j]$$
where the last sum is indexed by the set $\{j\in e(s)-\mathbb N~|~j> e(\Theta_l)+1\}$.

Before the proofs we provide an example:
\tikzset{
    segment/.style={line width=1.5pt, black},
    start pt/.style={circle, fill=blue, inner sep=1.5pt},
    end pt/.style={circle, fill=red, inner sep=1.5pt},
    point halo/.style={circle, fill=black, inner sep=2pt},
    theta link/.style={line width=3pt, green!50, opacity=0.7},
    point theta/.style={circle, fill=green!50, opacity=0.7,inner sep=2pt}
}

\begin{ex}
We reproduce a detailed manual calculation to further illustrate the algorithm. Let $s = [0, 2]$. We apply $\CT_\s$ to the multisegment $\mathfrak{m} = [-2,1] + [-2,-1]+ [-1,0] + [0,0] \in\DMult^{\le'[0,2]}$.
\begin{enumerate}[noitemsep,topsep=0pt]
    \item \textbf{Initial sequence $\Theta^\s$:} We compute $\Theta^{[0,2]}(\mathfrak{m},\e)$. To do that we first order the segments of $\m$ increasingly for the order $<$:
    
 
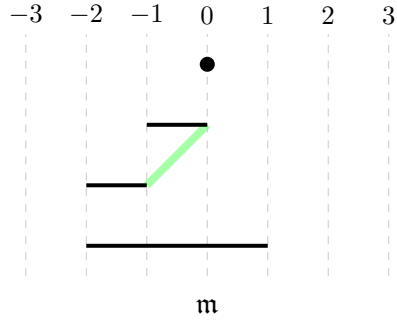
\hspace{2.1cm}
\begin{tikzpicture}[scale=0.8]
    \foreach \x in {-3,...,3} {
        \draw[dashed, gray!40] (\x, -2.0) -- (\x, 2.0);
        \node[above] at (\x, 2.0) {\small $\x$};
    }
    
    \draw[theta link] (-1, -0.5) -- (0, 0.5);
    
    \draw[segment] (-2, -1.5) -- (1, -1.5);
    \fill[start pt] (-2, -1.5) circle; \fill[end pt] (1, -1.5) circle;
    
     \draw[segment] (-2, -0.5) -- (-1, -0.5);
    \fill[start pt] (-2, -0.5) circle; \fill[end pt] (-1, -0.5) circle;
    
      \draw[segment] (-1, 0.5) -- (0, 0.5);
    \fill[start pt] (-5, -2.0) circle; \fill[end pt] (-4, -2.0) circle;

    \node[point halo] at (0, 1.5) {}; \node[right] at (0.1, 0.0) {};

    \node at (0, -2.5) {$\mathfrak{m}$};
 
\end{tikzpicture}

    The only segment ending at $-1$ is $\Theta_1  = [-2,-1]$.  The smallest segment ending at $0$ and bigger than $[-2,-1]$ is $[-1,0]$. the sequence ends here since $[-2,1]$ is smaller than $\Theta_2$.
    
    \item \textbf{Result $\CT_s(\m)$:} By extending the segments $\Theta_1$ and $\Theta_2$ and by adding the point $[2,2]$ we get $\CT_s(\m)=[-2,1]+[-2,0]+[-1,1]+[0,0]+[1,1]$.
    
    
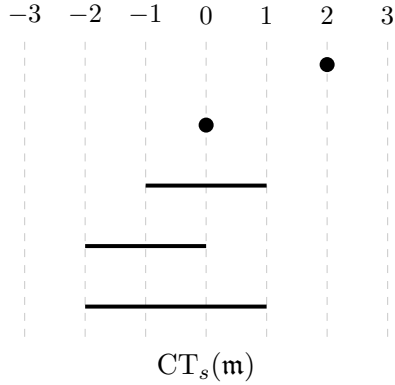
\begin{center}      
    \begin{tikzpicture}[scale=0.8]
    \foreach \x in {-3,...,3} {
        \draw[dashed, gray!40] (\x, -2.5) -- (\x,2.5);
        \node[above] at (\x, 2.5) {\small $\x$};
    }

    \draw[segment] (-2, -2) -- (1, -2);
    \fill[start pt] (-2, -1.5) circle; \fill[end pt] (1, -1.5) circle;
    
     \draw[segment] (-2, -1) -- (0, -1);
    \fill[start pt] (-2, -0.5) circle; \fill[end pt] (0, -0.5) circle;
    
      \draw[segment] (-1, 0) -- (1, 0);
    \fill[start pt] (-1, 0.5) circle; \fill[end pt] (1, 0.5) circle;

    \node[point halo] at (0, 1) {}; \node[right] at (0.1, 0.0) {};
    
    \node[point halo] at (2, 2) {}; \node[right] at (0.1, 0.0) {};

    \node at (0, -3) {$\CT_s(\m)$};
 
    \end{tikzpicture}
    \end{center}

\end{enumerate}
We can now try to apply the M{\oe}glin-Waldpsurger algorithm to this new multisegment and we verify that $\CT_s(\m)_1=s$ and $\CT_s(\m)^\#=\m$ so $\CT_s(\m)^t=s+\m=\T_s(\m)$. This is the key result and we will prove it in the general case with the same strategy: applying M{\oe}glin-Waldspurger algorithm to $\CT_s(\m)$.
       
\end{ex}

Write $s=[e,b]\in \Seg$ and fix $\m\in\DMult^{\le's}$. We start with a lemma.

\begin{lem}\label{lemm bound ending GLn}
    We necessarily have $e(\Theta_l)\le e-1$, and if $e(\Theta_l)= e-1$, there are no segments in $\m$ that end at $e$ and that begin strictly after $b(\Theta_l)$.
\end{lem}

\begin{proof}
    Since $\m\in\DMult^{\le 's}$, we have $s(\m)\le' \s$. If the endpoint $e'=e(s(\m))$ of $s(\m)$ is strictly smaller than $e=e(s)$, the result is clear since $e'$ is by definition the largest endpoint among the segments of $\m$. Now assume $e'=e$. By definition, we then have $b'=b(s(\m))$ greater than $b=b(s)$. We now prove that having $e(\Theta_l)=e$ contradicts the fact that the sequence $\Delta_1,\ldots,\Delta_k$ stops at $\Delta_k$ with $e(\Delta_k)=b'$. 

    Assume that $e(\Theta_l)=e$; then by definition of $\Delta_1$, we have $\Theta_l\le\Delta_1$. Now consider $\Delta_i$ and assume that $\Theta_r\le\Delta_i$ with $r$ such that $\Theta_r$ ends at $j:=e(\Delta_i)$. If $j\ge b-1$, we have $r\ge 2$ and then $\Theta_{r-1}\in\m$. Moreover, $\Theta_{r-1}\le\Theta_r\le \Delta_i$ and $e(\Theta_{r-1})=e(\Theta_r)-1=e(\Delta_i)-1$. This implies that $k\ge i+1$. 

    By induction, this implies that $e(\Delta_k)\le b-1<b$, which is impossible because $e(\Delta_k)=b'\ge b$. This implies that $e(\Theta_l)\le e-1$. 

    Now if $e(\Theta_l)=e-1$, there are no segments that end at $e$ and that begin strictly after $b(\Theta_l)$
 because such a segment would imply the existence of $\Theta_{l+1}$ ending at $e$.    
    \end{proof}

Thanks to this lemma, we can now prove the key result concerning $\CT_s$, which is that $\CT_s$ is dual to $\T_s$.

\begin{prop}\label{prop duality GLn}
    For any $\m\in\DMult^{\le's}$, we have $ (\CT_s(\m))^t=\T_s(\m^t)$.
\end{prop}

\begin{proof}
Take $\m\in\DMult^{\le's}$.
To prove this result, we apply the M{\oe}glin-Waldspurger algorithm to $\bar \m:=\CT_s(\m)$ and we verify that $\bar \m^\#=\m$ and $\bar \m_1=s$, so that we obtain $\bar \m^t=s+\m^t$.
Then, since $s$ is larger than any segment in $\m$ by hypothesis, we get:
$$(\CT_s(\m))^t=\bar \m^t=\T_s(\m^t)$$

To do this, we need to compute the sequence $\bar \Delta_1,\ldots,\bar \Delta_k$ of $\bar \m$. First, note that all the segments in $\m$ have to end before $e$ since $\m\in\DMult^{\le's}$. Then we deduce that the largest endpoint $\bar e_{\max}$ of $\bar \m$ has to be $e$. Indeed, according to Lemma \ref{lemm bound ending GLn}, either $\Theta^s_l$ ends at $e-1$ and then $\Theta_l^+$ ends at $e$, or $\Theta_l$ ends strictly before $e-1$ and then $[e,e]\in\bar \m$ ends at $e$. According to Lemma \ref{lemm bound ending GLn}, when $e(\Theta_l)=e-1$, we deduce that the largest segment in $\bar \m$ ending at $e$ is $\Theta_l^+$, so $\bar \Delta_1=\Theta_l^+$. In the second case, $[e,e]$ is the largest possible segment ending at $e$, so we necessarily have $\bar \Delta_1=[e,e]$. 

Now consider $\bar \Delta_i$ and assume that it is either a point $[j,j]$ or $\Theta_r^+$.
\begin{itemize}
    \item If $\bar \Delta_i=[j,j]$, we have three possibilities.
    \begin{itemize}
        \item Assume that $j-1> e(\Theta_l)+1$; then $[j-1,j-1]\in\bar \m$, and this is the largest possible segment that is smaller than $[j,j]$ and that ends at $j-1$. Thus, we deduce that $k\ge i+1$ and that $\bar \Delta_{i+1}=[j-1,j-1]$.

        \item Assume that $j-1=e(\Theta_l)+1$. The segment $\Theta^+_l\in\bar \m$ ends at $j-1$ and is smaller than $[j,j]$, so we know that $k\ge i+1$. Now we show that $\bar \Delta_{i+1}=\Theta_l^+$. To do this, it suffices to prove there is no segment $u$ in $\bar \m$ strictly larger than $\Theta_l^+$ that ends at $j-1$. However, such a segment would have to belong to $\m$, and this would be a larger segment than $\Theta_l$, which would imply that the sequence $\Theta$ does not stop with $\Theta_l$. Thus, $\bar \Delta_{i+1}=\Theta_l^+$.
        \item Assume that $\Theta_1$ is empty and $j=b$. The fact that $\Theta_1$ is empty implies that there are no segments in $\m$ that end at $b-1$, the same being also true in $\bar \m$, and we then deduce in this case that $k=i$.
    \end{itemize}
    \item If $\bar \Delta_i=\Theta_r^+$, denote by $j$ the endpoint of $\bar \Delta_i$. We have two possibilities:
    \begin{itemize}

    \item If $j>b$, then $r\ge 2$ and the segment $\Theta_{r-1}^+\in\bar \m$ ends at $j-1$. Since $\Theta_{r-1}\le \Theta_r$, we have $\Theta_{r-1}^+\le \Theta_r^+$ and we deduce that $k\ge i+1$. We now prove that $\bar \Delta_{i+1}=\Theta_{r-1}^+$. Assume there exists $u\in \bar \m$ that ends at $j-1$, that is smaller than $\Theta_r^+$, and that is strictly larger than $\Theta_{r-1}^+$. Then this segment lies in $\m$, but this contradicts the minimality of $\Theta_r$, so we deduce there is no such $u$ and $\bar \Delta_{i+1}=\Theta_{r-1}^+$ 

    \item If $j=b$, then $r=1$. Assume there exists a segment $u$ in $\bar \m$ which ends at $b-1$ and is strictly smaller than $\Theta_1^+$. Then $u$ has to be in $\m$, and this segment is strictly smaller than $\Theta_1$, so it contradicts the definition of $\Theta_1$. We conclude that there is no such $u$ and that $k=i$.
    \end{itemize}
\end{itemize}

By induction, all of this implies that $\bar \m_1=[e(\bar \Delta_k),e(\bar \Delta_1)]=s$  and:
$$\sum_{i=1}^k\bar \Delta_i=\sum_{r=1}^l\Theta_r^++\sum[j,j]$$
where the second sum is indexed as in the definition of $\CT_s(\m)$.
We then get:
\begin{align*}
\bar \m^\#&=\bar \m-\sum_{i=1}^k\bar \Delta_i+\sum_{i=1}^k\bar \Delta_i^-\\
&=\left(\m-\sum_{r=1}^l\Theta_r+\sum_{r=1}^l\Theta_r^++\sum[j,j]\right)-\sum_{r=1}^l\Theta_r^++\sum[j,j]+\sum_{r=1}^l(\Theta_r^+)^-+\sum[j,j]^-\\
&=\m-\sum_{r=1}^l\Theta_r+\sum_{r=1}^l\Theta_r\\
&=\m 
\end{align*}
since $[j,j]^-=\varnothing$ for any $j$ and $(u^+)^-=u$ for any segment.
This concludes the proof.
\end{proof}

This proposition provides a new algorithm to compute $\m^t$. Indeed, write $\m=s_1+\ldots+s_n$ with $s_1\le'\ldots\le's_n$, then we have:
$$\m=\T_{s_n}\circ\ldots \T_{s_1}(\varnothing)$$
and according to Proposition \ref{prop duality GLn}, we also have:
$$\m^t=\CT_{s_n}\circ\ldots\circ\CT_{s_1}(\varnothing)$$

\begin{rem}
The characterization of co-tempered data in this case was already established in \cite{zel} (it follows from 9.3, 9.5 and 8.4). Indeed, in this case, it is well known that co-tempered multisegments  are the $\m$ such that:
\begin{itemize}
  \item $\m$ consists only of points (segments of length 1).
  \item for any $x\in\mathbb R$, we have $m_\m([x,x]_{\rho_u})=m_\m([-x,-x]_{\rho_u})$
  \item for any $x\in\mathbb R^+$, we have $m_\m([x,x]_{\rho_u})\ge m_\m([x+1,x+1]_{\rho_u})$
\end{itemize}
We can easily recover this result inductively by applying the M{\oe}glin-Waldspurger algorithm to such a multisegment, or by using our new algorithm on a tempered multisegment. We then see that while the M{\oe}glin-Waldspurger algorithm preserves co-temperedness, our new algorithm preserves temperedness. 
\end{rem}

\subsection{Ugly parity case}

This subsection will be particularly short because, as we will see, according to the Lanard-Mínguez algorithm, the ugly parity case follows immediately from the previous case of $\GL_n$.

Suppose here that $\rho$ is of ugly parity. By definition of $\Symm_\rho^\e(G)$ in this case, all the signs are $+1$; then we can identify any element of $\Symm_\rho^\e(G)$ with its underlying multisegment which lies in $\Mult_\rho\times \Mult_{\rho^\vee}$.

Take $\m\in\Symm^\e_\rho$ and write $\m = \m_\rho + \m_{\rho^{\vee}}$ with $\m_\rho \in \Mult_\rho$ and $\m_{\rho^{\vee}} \in \Mult_{\rho^{\vee}}$. According to Remark 5.1.3 of \cite{LM}, $\AD_\rho(\m)=\m_\rho^t+(\m_\rho^t)^{\vee}$.

Because of this remark, our new algorithm to compute the M{\oe}glin-Waldspurger dual automatically gives a new algorithm for the ugly parity case. One just needs to apply it to $\m_\rho$ and then add its symmetric to the result.

This remark also implies the following characterization for co-tempered data in the ugly parity case.

\begin{prop}\label{thm ugly}
Let us take $\rho\in\Cusp^\GL$ of ugly parity and $\m\in\Symm_\rho^\e(G)$. We write $\m = \m_\rho + \m_{\rho^{\vee}}$ with $\m_\rho \in \Mult_\rho$ and $\m_{\rho^{\vee}} \in \Mult_{\rho^{\vee}}$. Then $\m$ is co-tempered if and only if the following conditions are true for any $\alpha\in\{\rho_u,\rho_u^\vee\}$:
  \begin{itemize}
    \item $\m_\alpha$ consists only of points (segments of length 1).
    \item for any $x\in\mathbb R$, we have $m_{\m_\alpha}([x,x]_\alpha)=m_{\m_\alpha}([-x,-x]_\alpha)$
    \item for any $x\in\mathbb R^+$, we have $m_{\m_\alpha}([x,x]_\alpha)\ge m_{\m_\alpha}([x+1,x+1]_\alpha)$
  \end{itemize}
\end{prop}

\begin{rem}
    Boiling down to the case of $\GL_n$ was not a necessity to define the algorithm or to get the characterization of the co-tempered data, we could have copied what we did for $\GL_n$ by replacing $\Mult$ with $\Symm^\e_\rho(G)$ and the M{\oe}glin-Waldspurger algorithm. 

    The rest of the article is to do such a thing for the bad and good parity cases since in these cases we cannot directly used the case of $\GL_n$. In the next Subsection we address the bad parity cases in which we can adapt everything pretty effortlessly.
\end{rem}

\subsection{Bad parity case}

Although this case shares some very strong similarities with the case of $\GL_n$, we cannot derive it directly from the case of $\GL_n$ as we did for the ugly parity case. The main difference from the ugly parity case is that in the bad case, we have $\rho=\rho^\vee$.

However, by the definition of $\Symm_{\rho}^\e(G)$ in this case, all the signs are $+1$; thus, we can identify any element of $\Symm_{\rho}^\e(G)$ with its underlying multisegment, which lies in $\Mult_{\rho}$.

We will proceed in a very analogous manner, and both the definitions and the results only need to be adapted.

First, we recall the Lanard-Mínguez recursive definition of $\ADp$ in this case. Let $\m\in\Symm_\rho^\e(G)$. To begin, let $e_{max}$ be the maximum of the right endpoints $e(s)$ among the segments $s\in\m$. Now, let $\Delta_1$ be the largest element of $\m$ such that $e(\Delta_1)=e_{max}$. Then, we recursively define the sequence $\Delta_1\geq \cdots\geq \Delta_l$. Suppose $\Delta_i$ has been constructed for an index $i\geq 0$; we then define $\Delta_{i+1}$ as the largest segment in $\m$ satisfying the following three conditions:
\begin{enumerate}
  \item $e(\Delta_{i+1})=\Delta_i-1$
  \item $\Delta_{i+1}\leq \Delta_i$
  \item If there exists $j<i+1$ such that $\Delta_{i+1}^\vee=\Delta_{j}$, then $m_{\m}(\Delta_j)\geq 2$
\end{enumerate}
If no such segment exists in $\m$, then we set $l=i$ and we stop. 
From this sequence, we define:
$$\m_1=[e(\Delta_l),e(\Delta_1)]_{\rho_u}+[-e(\Delta_1),-e(\Delta_l)]_{\rho_u}
$$
and
$$\m^\#=\m+\sum_{i=1}^l(\Delta_i^--\Delta_i+{}^-(\Delta_i^\vee)-\Delta_i^\vee)$$
This finally allows us to define $\ADp$ recursively by the following formula:
$$\ADp(\m)=\m_1+\ADp(\m^\#)$$
In what follows, we will denote $s(\m)=[-e(\Delta_1),-e(\Delta_l)]_{\rho_u}$.

Let $s$ be a segment with a strictly positive center. We now define $\Symm_\rho^{\le' s}$ as the set of data $\m \in \Symm^\e_\rho(G)$ such that every segment $\Delta \in \m$ satisfies $\Delta\le' s$.

For $\m \in\Symm_\rho^{\le' s}$, we define the operator $\T_s: \Symm_\rho^{\le' s} \to \Symm_\rho^{\le' s}$ by:
    $$\T_s(\m) = \m + s+s^\vee$$

Now let $\DSymm_\rho^{\le' s}$ be the dual $\ADp(\Symm_\rho^{\le' s})$ of $\Symm_\rho^{\le' s}$. This definition is not intrinsic because we need to compute the dual of a datum to determine whether it belongs to $\DSymm_\rho^{\le' s}$.

However, if we take $\m\in\Symm^\e_\rho(G)$, by applying the Lanard-Mínguez algorithm, we obtain two data $\m_1$ and $\m^\#$ such that $\ADp(\m)=\m_1+\ADp(\m^\#)$, where the datum $\m_1$ is $s(\m)+s(\m)^\vee$, with $s(\m)$ being the largest segment in $\ADp(\m)$ for the order $<'$.

Then, we have:
$$\DSymm_\rho^{\le's}=\{\m\in\Symm_\rho^\e(G)~|~s(\m)\le's\}$$

We now fix a segment $s$ and we define an operator $\CT_s$ on $\DSymm_\rho^{\le's}$. To do this, we first define an increasing sequence of segments of $\m$.

Define $\Theta^s_1$ as the smallest segment in $\m$ for the order $<$ that ends at $b(s)-1$, and as $\varnothing$ if there are no segments in $\m$ satisfying this condition.

Now we recursively define the sequence $\Theta^s_1\le\ldots\le\Theta^s_l$, with $\Theta^s_i$ being the smallest segment in $\m$ satisfying the following three conditions: 
\begin{enumerate}
    \item $e(\Theta^s_i)=e(\Theta^s_{i-1})+1$
    \item $\Theta^s_i>\Theta^s_{i-1}$
    \item If there exists $j<i+1$ such that $(\Theta_{i+1}^s)^\vee=\Theta^s_j$, then $m_\m(\Theta_{j}^s)\ge 2$
\end{enumerate}

If no such segment exists in $\m$, then we set $l=i$ and we stop.
Of course, if $\m=\varnothing$, we have $\Theta_1^s=\varnothing$ (regardless of what $s$ is).
Even if $\Theta^s_i$ depends on the segment $s$, since this choice has been fixed, we will denote it by $\Theta_i$ in what follows.

Finally we define:
$$\CT_s(\m)=\m+\sum_{i=1}^l(\Theta_i^+-\Theta_i+{}^+(\Theta_i^\vee)-\Theta_i^\vee)+\sum([j,j]+[j,j]^\vee)$$

where the last sum is indexed by the set $\{j\in e(s)-\mathbb N~|~j> e(\Theta_l)+1\}$.

Here is the duality result for this case:

\begin{prop}\label{prop duality bad}
    For any $\m\in\DSymm_\rho^{\le's}$, we have $ \ADp(\CT_s(\m))=\T_s(\ADp(\m))$.
\end{prop}
\begin{proof}
The idea of the proof is exactly the same as in the case of $\GL_n$: we first need to prove the statement of Lemma \ref{lemm bound ending GLn}, then we need to show that the algorithm exactly exhausts the sequence $(\Theta^s_i)_{1\le i\le l}$ and the points we add. This automatically implies that $\bar \m^\#=\m$ and $\m_1=s+s^\vee$.

The main difference throughout is the third point in the definition of the sequence $\Delta_1,\ldots,\Delta_k$ of Lanard-Mínguez. However, this difference is completely offset by our definition of $\Theta_1^s,\ldots,\Theta_l^s$ in this case.

For example, we prove in this case that if $\bar \Delta_i=\Theta_r^+$ and if $j:=e(\bar \Delta_i)>b$, then $k\ge i+1$ and $\bar \Delta_{i+1}=\Theta_{r-1}^+$.

Since $\Theta_{r-1}\le \Theta_r$, we have $\Theta_{r-1}^+\le \Theta_r^+$ and $e(\Theta_{r-1}^+)=e(\Theta_r^+)-1$. Moreover, if there exists $j<i+1$ such that $(\Theta_{i+1}^s)^\vee=\Theta_j^s$, then $m_\m(\Theta_{j}^s)\ge 2$, so we can deduce that $k\ge i+1$. The proof that $\bar \Delta_{i+1}=\Theta_{r-1}^+$ is exactly the same as for $\GL_n$. 

As a final example, we treat the case of the points. Assume that $\bar \Delta_i=[j,j]$ and assume that $j-1>e(\Theta_l)+1$. The only difference with the case of $\GL_n$ is that we need to prove that, if there exists $i'$ such that $\bar \Delta_{i'}=[-j+1,-j+1]$, we have $m_{\bar \m}([-j+1,-j+1])\ge2$. Assume such an $i'$ exists. Then we have both $j-1$ and $-j+1$ in $\{j\in e(s)-\mathbb N~|~j> e(\Theta_l)+1\}$. Since $[j-1,j-1]^\vee=[-j+1,-j+1]$, this implies that $m_{\bar \m}([-j+1,-j+1])=m_{\m}([-j+1,-j+1])+2\ge 2$.

\end{proof}

Again, this proposition provides a new algorithm to compute $\AD(\m)$. Indeed, write $\m=(s_1+s_1^\vee)+\ldots+(s_n+s_n^\vee)$ with $s_1\le'\ldots\le's_n$ and $c(s_i)\ge 0$ for any $1\le i\le n$; then we have:
$$\m=\T_{s_n}\circ\ldots \T_{s_1}(\varnothing)$$
and according to Proposition \ref{prop duality GLn}, we also have:
$$\ADp(\m)=\CT_{s_n}\circ\ldots\circ\CT_{s_1}(\varnothing)$$

We can now use this algorithm to obtain a characterization of the co-tempered data.

\begin{theorem}\label{thm bad}
  Let us take $\rho\in\Cusp^\GL$ of bad parity and $\m\in\Symm_\rho^\e(G)$. Then $\m$ is co-tempered if and only if the following conditions are true:
  \begin{itemize}
    \item $\m$ consists only of points
    \item for any $x\in\mathbb R^+$, we have $m_{\m}([x,x]_{\rho_u})\ge m_{\m}([x+1,x+1]_{\rho_u})$
  \end{itemize}
\end{theorem}

\begin{proof}
    Let $\m$ be a tempered datum; we will show that its dual satisfies all three conditions of the statement. Since $\ADp$ is a bijection, this is sufficient. Since Aubert duality preserves the supercuspidal support, it is sufficient to show that the second condition holds for $\m$, which immediately follows from the fact that $\m$ is tempered. Now we only need to show that $\ADp(\m)$ consists exclusively of points. We proceed by induction. Take $\m$ and a segment $s=[e,b]$ with $c(s)\ge 0$ such that the results hold for $\m$ and $\m \in\DSymm^{\le's}$, and we now prove that it also holds for $\T_{s}(\m)$. According to Proposition \ref{prop duality bad}, we have $\ADp(\T_{s}(\m))=\CT_s(\ADp(\m))$. Write $\m'=\ADp(\m)$, which consists exclusively of points by assumption. We show that $\Theta^s_1=\varnothing$. The smallest starting point of segments of $\m'$ is the same as the smallest starting point of segments of $\m$, which is $-e(s(\m))\ge -e=b$. Thus, the smallest right endpoint must be $\ge b$, meaning there are no segments in $\m'$ ending at $b-1$, whence $\Theta^s_1=\varnothing$. This implies that:
    $$\AD(\m)=\m'+\sum_{j\in[b,e]}([j,j]+[j,j]^\vee)$$
    Since $\m'$ consists exclusively of points, this property also holds for $\m$.

    This concludes the proof.
\end{proof}

\begin{rem}
    This proof can be adapted almost verbatim to prove the characterization of the co-tempered multisegments. We can also prove this characterization by applying the Lanard-Mínguez algorithm to a datum $\m$ satisfying both conditions of the statement. 
\end{rem}

The goal of this article is now to achieve the same result for the good parity case. Although the strategy and the underlying ideas are the same, this case will be much more technical. The main reason is that we will have signs on the centered segments, and these signs play an essential role in the computation of the Aubert dual, as seen in the definition of the Lanard-Mínguez algorithm.

We will be able to define the equivalent operators and prove the duality result, but in this case, it will not readily yield a straightforward description of the co-tempered data. However, it will still provide an interesting constructive description as well as new properties.

\section{The good parity case}\label{section good parity}

The case of good parity is significantly more challenging than the previous ones because the signs play an essential role. Let us now fix $\rho\in\Cusp^{GL}$ of good parity until the end of this article.

\subsection{Some definitions}\label{subsec some def}

We will now define some specific subsets of symmetrical Langlands data. In the notation of those sets, the cuspidal representation $\rho$ should appear, but we will drop it since $\rho$ has been fixed and will not change in what follows. We will also drop the $\rho$ as an index of  the sets we have already introduced and write for instance $\Symm^\e(G)$ instead of $\Symm^\e_\rho(G)$, the fact that we only work on the supercuspidal line being implicit. Moreover, any segments that we write will be supported on $\rho_u$, the unitarization of $\rho$, and we will write $[x,y]$ instead of $[x,y]_{\rho_u}$.

We now define what we can view as a filtration of $\Symm^\e(G)$. To do so, observe that for any $(\m,\e)\in\Symm^\e(G)$ we can see $(\m,\e)^{\ge 0}:=(\m^{>0}+\m^{=0},\e_{|\m^{>0}+\m^{=0}})$ as an element of $\mathbb N(\Seg^{>0}\cup (\Seg^{=0}\times\{1,-1\}))$.

We extend the order $\le'$ to $\Seg^{>0}\cup (\Seg^{=0}\times\{1,-1\})$  by saying, for two elements $\s$ and $\s'$ of this set, that $\s>'\s'$ if the inequality holds for the underlying segments. Note that this order is not total since we cannot compare $(s,1)$ and $(s,-1)$. This will not matter since a symmetrical datum cannot contain two such signed segments.

We define, for any $\s\in \Seg^{>0}\cup (\Seg^{=0}\times\{1,-1\})$, the set 
$$\Symm^{\le'\s}=\{(\m,\e)\in\Symm^\e(G)~|~\s'\le'\s,~ \forall\s'\in(\m,\e)^{\ge 0}\}.$$

We can dualize this construction: for any $\s\in\Seg^{>0}\cup (\Seg^{=0}\times\{1,-1\})$ we define $$\DSymm^{\le'\s}=\AD(\Symm^{\le'\s})$$

\begin{rem}\label{rem largest segment in the dual LM}
This definition is not intrinsic in the sense that we can not decide if $(\m,\e)\in\DSymm^{\le'\s}$ without computing the dual $\AD(\m,\e)$. To solve that we can use Lanard-Mínguez algorithm.

Take $(\m,\e)\in\Symm^\e(G)$, Lanard and Mínguez define in \cite{LM} a quadruple $(\m_1,\e_1,\m^\#,\e^\#)$ such that $\AD(\m,\e)=(\m_1,\e_1)+\AD(\m^\#,\e^\#)$ with $(\m_1,\e_1) $ is either of the form $s+s^\vee$ with $s\in \Seg^{>0}$ or $(s,\eta)\in\Seg^{=0}\times\{1,-1\}$. In Proposition 7.1.5 of \cite{LM} they prove that in both cases, such a $s$ is the largest segment in $\AD(\m,\e)$ for $\le '$. This allows us to define, for any $(\m,\e)\in\Symm^\e(G)$, an element $\s(\m,\e)\in\Seg^{>0}\cup (\Seg^{=0}\times\{1,-1\})$ (this is $s$ if $\m_1=s+s^\vee$ and $(\m_1,\e_1)$ otherwise) and then:
$$\DSymm^{\le'\s}=\{(\m,\e)\in\Symm^\e(G)~|~\s(\m,\e)\le' \s\}.$$
Finally, remark that if  $\s(\m,\e)=s$ or $\s(\m,\e)=(s,\eta)$, the ending point of $s$ is always, by definition, the largest ending points of segments of $\m$.
\end{rem}

Note that by definition, for any $\s,\s'\in \Seg^{>0}\cup (\Seg^{=0}\times\{1,-1\})$ with $\s\le' \s'$, we have $\Symm^{\le'\s}\subset \Symm^{\le'\s'}$ and $\DSymm^{\le'\s}\subset \DSymm^{\le'\s'}$. Also by definition we have $(\m,\e)\in\DSymm^{\le'\s(\m,\e)}$ for any $(\m,\e)\in \Symm^\e(G)$.

We now define some operators on these sets but before doing it, for any $\s\in \Seg^{>0}\cup (\Seg^{=0}\times\{1,-1\})$, write $\tilde \s=s+s^\vee$ if $\s=s\in\Seg^{>0}$ and $\tilde \s=\s$ otherwise, so that $\tilde \s\in\Symm^\e(G)$.

\begin{deff}
    Let $\s\in \Seg^{>0}\cup (\Seg^{=0}\times\{1,-1\})$. We define the following operator:
$$\begin{array}{ccccc}
\T_{\s}: & \Symm^{\le'\s} & \to & \Symm^{\le'\s} \\
  & (\m, \e) & \mapsto & (\m, \e)+\tilde \s \\
\end{array}$$
\end{deff}

Take $(\m,\e)\in\Symm^\e(G)$ and write $(\m,\e)^{\ge 0}=\s_1+\ldots+\s_n$ with $\s_i\in\Seg^{>0}\cup (\Seg^{=0}\times\{1,-1\})$ for any $1\le i\le n$ with $\s_1\le'\ldots\le'\s_n$. It is possible to do that even if $\le'$ is not total because by definition of $\Symm^\e(G)$, two identical centered segments in $\m$ must share the same sign. Then we get: $$(\m,\e)=\T_{\s_n}\circ\ldots\circ \T_{\s_1}(\varnothing)$$

This observation leads us to the following: suppose that for any $\s$ in $\Seg^{>0}\cup (\Seg^{=0}\times\{1,-1\})$ we have an explicit formula for the map $\AD\circ \T_\s\circ \AD:\DSymm^{\le'\s}\to \DSymm^{\le'\s}$, then we would get an algorithm to compute the dual $\AD(\m,\e)$ of any $(\m,\e)\in \Symm^\e(G)$.  Indeed, we have:
$$\AD(\m,\e)=(\AD\circ \T_{\s_n}\circ \AD)\circ \ldots\circ(\AD\circ \T_{\s_s}\circ \AD)(\varnothing)$$
and we could compute $\AD(\m,\e)$ in $n$ steps where each step is applying such a dual operator $\AD\circ\T_\s\circ \AD$.

The core of our work is to define explicitly, for any such $\s$, an operator $\CT_\s:\DSymm^{\le'\s}\to \DSymm^{\le'\s}$ and to verify that we have $\AD\circ \T_\s=\CT_{\s}\circ \AD$, or equivalently, that $\CT_\s$ is the dual operator we are interested in.

To define such an operator we need a very specific way to order the segments of a symmetrical datum introduced by Lanard and Mínguez.

\subsection{Lanard-Mínguez order}

To define their algorithm, Lanard and Mínguez needed to order in a specific way the segments of symmetrical Langlands data. We recall this enrichment and ordering of $\Symm_\rho^\e(G)$ given in Subsection 5.3 of \cite{LM}.

 We formally enrich segments with a label $\clubsuit \in \{\ge 0, =0, \le 0\}$ that encodes their position.
\bigskip

Let $\USeg$ be the set of labeled pairs $(\Delta, \clubsuit)$, where $\Delta \in \Seg$ and $\clubsuit$ satisfies the following conditions:
\begin{itemize}
\item If $c(\Delta) > 0$, then $\clubsuit$ is equal to $\ge 0$;
\item If $c(\Delta) < 0$, then $\clubsuit$ is equal to $\le 0$.
\end{itemize}

In the case $c(\Delta) = 0$, then $\clubsuit$ can be any of the three values. So only centered segments ($c(\Delta) = 0$) carry a nontrivial choice of label; in all other cases, the label is determined uniquely and may be omitted. In those cases, we will simply write $\Delta$ instead of $(\Delta, \clubsuit)$. For centered segments, we usually indicate the label by a superscript $\Delta^{\clubsuit}$, \textit{e.g.}, $[-a, a]^{\ge 0}$, $[-a, a]^{=0}$, or $[-a, a]^{\le'0}$.

We now define a natural involution $\iota$ and a map $\iota'$ on the set $\USeg$ using the contragredient. For $\Delta^{\clubsuit} \in \USeg$, we define $\iota(\Delta^{\clubsuit}) \in \USeg$ by:
\begin{enumerate}
  \item If $c(\Delta) > 0$, then $\iota(\Delta^{\ge 0}) =\iota'(\Delta^{\ge 0})=(\Delta^{\vee})^{\le'0}$.
  \item If $c(\Delta) < 0$, then $\iota(\Delta^{\le'0})^{\vee}=\iota'(\Delta^{\le'0})^{\vee} = (\Delta^{\vee})^{\ge 0}$.
  \item If $c(\Delta) = 0$ and $\clubsuit$ is $= 0$, then $\iota(\Delta^{= 0}) = \iota'(\Delta^{= 0}) = (\Delta^{\vee})^{= 0}$.
  \item If $c(\Delta) = 0$ and $\clubsuit$ is $\ge0$, then $\iota(\Delta^{\ge 0}) = \iota'(\Delta^{\ge 0}) =(\Delta^{\vee})^{\le'0}$.
  \item If $c(\Delta) = 0$ and $\clubsuit$ is $\le0$, then $\iota(\Delta^{\le'0}) = (\Delta^{\vee})^{\ge 0}$ but  $\iota'(\Delta^{\le'0}) = (\Delta^{\vee})^{\le'0}$ .
\end{enumerate}

Clearly $\iota$ is an involution but $\iota'$ is not. In \cite{LM}, the authors use a similar formula to extend the contragredient to $\USeg$. The contragredient is similar to the map $\iota'$ but with $\mathcal S^{\ge 0}$ instead of $\mathcal S^{\le'0}$.

The contragredient and the maps $\iota$ and $\iota'$ naturally extend to multisets. Let $\USymm$ denote the multisets in $\USeg$ that are symmetric under $\iota$, i.e., those satisfying $\iota(\m) = \m$.  In what follows, we will need $\iota'$ (and not the contragredient as in \cite{LM}). \\

There is a natural surjection
\[
  p: \USymm \twoheadrightarrow \Symm
\]
which forgets the labels. This projection has a section
\[
  s: \Symm \to \USymm
\]
where $s(\m)$ equals
\[
\sum_{\Delta \in \m, c(\Delta) \neq 0} \Delta + \sum_{\Delta \in \m, c(\Delta) = 0}  \left\lfloor \frac{m_{\m}(\Delta)}{2} \right\rfloor \Delta^{\le'0}+ \left\lfloor \frac{m_{\m}(\Delta)}{2} \right\rfloor \Delta^{\ge 0}+  \left(m_{\m}(\Delta) -2 \left\lfloor \frac{m_{\m}(\Delta)}{2} \right\rfloor\right) \Delta^{=0}.
\]

Let $\Delta_1,\Delta_2 \in \Symm$. We define an order relation $\prec$ on the segments of $\USeg$ supported in $\Z_\rho$ in the following way:
\begin{itemize}
  \item $(\Delta_1,\le 0) \prec (\Delta_2,= 0)$.
  \item $(\Delta_1,= 0) \prec (\Delta_2,\ge 0)$.
  \item If $\clubsuit$ is $\ge 0$ or $\le 0$, then $\Delta_1^{\clubsuit} \preceq \Delta_2^{\clubsuit}$ if and only if $\Delta_1 \le \Delta_2$.
  \item If $\clubsuit$ is $= 0$, then $\Delta_1^{\clubsuit} \preceq \Delta_2^{\clubsuit}$ if and only if $e(\Delta_1) \le e(\Delta_2)$.
\end{itemize}
The transitive closure of these relations defines an order on $\USeg$.

Finally, we add signs to centered segments. Let $\USymm^{\varepsilon}(G)$ be the set of pairs $(\m,\varepsilon)$
with $\m \in \USymm$ and $\varepsilon: \{ \Delta \in \m,~c(\Delta)=0 \}  \to \{-1,1\}$. The order $\prec$ on $\USymm$ extends naturally to an order on $\USymm^{\varepsilon}$. The maps $p$ and $s$ give maps $p: \USymm^{\varepsilon}(G) \to \Symm^{\varepsilon}(G)$ and $s: \Symm^{\varepsilon}(G) \to \USymm^{\varepsilon}(G)$.

\subsection{Definition of $\Theta^{k}$}\label{subsec Theta}

Let $k\in 1/2\cdot\Z$, let $(\m,\e)\in \Symm^\e(G)$ be a symmetrical Langlands datum and let us write $(y,\e)=s(\m,\e)\in\SSymm^\e(G)$. In this subsection we define $\Theta^k=\Theta^k(\m,\e)\in \SSymm$.

We start by defining $\Theta^k_1$ as the smallest $u\in y$ such that:
\begin{enumerate}
  \item $e(u)=k$.
  \item if there is an integer $i\ge 1$ such that $p(u)=[-i+1/2,i-1/2]$, then $\e(u)=(-1)^i$.
\end{enumerate}
with $\Theta^k_1=\varnothing$ if there is no such $u$.

Now suppose we already have defined $\Theta^k_n$ for $n\ge 1$. We define $\Theta^k_{n+1}$ as the smallest element $u\in y$ such that:
\begin{enumerate}
  \item $u\succ\Theta_n$ if $\Theta_n\neq \varnothing$
  \item $e(u)=e(\Theta^k_{n})+1$
  \item if $c(u)=0$ and $c(\Theta^k_n)=0$, then $\e(u)=-\e(\Theta^k_n)$
\end{enumerate}
If there is no such element $u\in y$ we set $\Theta^k_{n+1}=\varnothing$. At the end, we define:
$$\Theta^k=\sum_{i=1}\Theta^k_i$$

We define a little variation of $\Theta^k$ by defining $\Theta^{k,0}$ in the same way except that we ask for $\Theta_1\in\mathcal S^{=0}\cup\mathcal S^{\ge0}$.

By induction we see that for any integer $n\ge 1$ such that $\Theta^k_{n}\neq \varnothing$, we have $e(\Theta^k_n)=n+k-1$ so we deduce that the terms become necessarily empty at some points.

We call the sequence $(\Theta^k_i)_{i\geq 1}$ the \textit{k-th increasing initial sequence} extracted from $(\m,\e)$ and we write $\Theta^k(\m,\e)$ or just $\Theta^k$ if there is no ambiguity on the symmetrical datum. Of course, if $\m=\varnothing$, we get $\Theta^k=\varnothing$ for any integer $k$.\\

Here is a first lemma that follows from the definition of $\Theta^k$.

\begin{lem}\label{lemme initial sequence} We prove two different points:
\begin{enumerate}
  \item Suppose that there exists $i\geq 1$ such that $\Theta^k_i\neq \varnothing$ is in $\mathcal S^{\geq 0}$ and is centered. Then if $\Theta^k_{i+1}$ is non-empty, it is not centered.
  \item If $\Theta^k_i$ is of length $1$ and $p(\Theta^k_i)\neq[0,0]$, then $\Theta^k_{i+1}$ is either a point or $\varnothing$.
\end{enumerate}
\end{lem}

\begin{proof}
  For the first point, take such an index $i$. We know that $\Theta^k_{i+1}\succ\Theta^k_i$ so, since $\Theta^k_i\in\mathcal S^{\geq 0}$ they are both in $\mathcal S^{\ge 0}$ and by definition of the order $<$ on segments, we cannot have  $p(\Theta^k_{i+1})={}^+p(\Theta^k_i)^+$; we deduce that $\Theta^k_{i+1}$ is not centered. Finally, if $i>1$, with what we have just seen, we know that $\Theta^k_{i-1}$ is not in $\mathcal S^{\geq 0}$ otherwise, $\Theta^k_i$ would not be centered. For the last point, it comes directly from the definition of the order.
\end{proof}

We now define $\CT_{\s}$.

\subsection{Definition of $\CT_\s$}\label{subsec def CT}
We first fix $\s\in\Seg^{>0}\cup (\Seg^{=0}\times\{1,-1\})$ that we write either $\s=s$ or $\s=(s,\eta)$. Now take a fixed symmetrical datum $(\m,\e)\in\DSymm^{\le'\s}$. In what follows we let $\tilde 1$ be $1$ if $\rho$ is of the same type as $G$ and $\tilde 1=3/2$ otherwise. Finally let $(b,e)=(b(s),e(s))$ if $\s=s\in \Seg^{>0}$ and $(b,e)=(\tilde 1,e(s))$ if $\s=(s,\eta)\in\Seg^{=0}\times \{1,-1\}$.

\begin{deff}
    We define $\Theta^\s(\m,\e)$. Recall that $\s(\m,\e)\le'\s$. Let $n_0$ be the number of centered segments (counted with multiplicity) in $\m$ and let $\alpha_0$ be either $\e([0,0])$ if $[0,0]\in\m$ or -1 otherwise.
\begin{itemize}
    \item If $\s=(s,\eta)$:
    \begin{itemize}
        \item if $\eta=(-1)^{n_0+1}\cdot\eta_0$, define $\Theta^\s(\m,\e)=\varnothing$,
        \item otherwise define $\Theta^\s(\m,\e)=\Theta^{b-1,0}(\m,\e)$.
    \end{itemize}

    \item Otherwise, start by considering $\Theta^{b-1}(\m,\e)$ and denote by $f$ the largest ending point of the segments in this sequence if it is not empty, and $f=b$ otherwise:
    \begin{itemize}
        \item if $f\ge 0$, then define $\Theta^\s(\m,\e)=\Theta^{b-1}(\m,\e)$
        \item otherwise, define $\Theta^\s(\m,\e)=\Theta^{b-1}(\m,\e)+\Theta^{\tilde 1-1,0}(\m,\e)$.
    \end{itemize}
\end{itemize}
\end{deff}

\begin{rem}
    Here are some remarks:
    \begin{enumerate}
        \item We have $\Theta^\s(\varnothing)=\varnothing$
        \item If $\s=s$ with $b(s)\ge 0$, then clearly one has $f\ge 0$ and thus $\Theta^\s(\m,\e)=\Theta^{b-1}(\m,\e)$
        \item In each case one can have $\Theta^\s(\m,\e)=\varnothing$.
        \item The largest ending point of segments of $\Theta^\s(\m,\e)$ is always less or equal than $e(s)$ by definition of $\DSymm^{\le'\s}$.
    \end{enumerate}
\end{rem}

We will define $(\bar \m,\bar \e)=\CT_\s(\m,\e)$. We do this in two steps. First we explain how to extend the segments in $\Theta^\s(\m,\e)$ to get a new datum $(\tilde \m,\tilde \e)$ which is the most difficult part. Then the second step is to define $(\bar \m,\bar \e)$ from $(\tilde \m,\tilde \e)$ by adding some segments (mostly some points and sometimes some segments of length 2).\\

 Let $\e_0'$ be 1 if $\s=(s,\eta)$ and $\e_0'=-1$ otherwise. Put $(y,\e)=s(\m,\e)$ and write:
$$y=\sum_{j=1}^m\Phi_j$$
with $\Phi_1\preceq \dots\preceq\Phi_m$. Here we assumed that $m\geq 1$, since $\m$ is not empty. Write $l$ for the number of non-empty segments in $\Theta:=\Theta^\s(\m,\e)$ (that could be 0) and, for any $1\le i\le l$, let $\Theta_i$ the $i-th$ segments in $\Theta$ so that $\Theta=\Theta_1+\ldots+\Theta_l$ in the increasing way.

We construct two sequences of integers $i_1,\dots,i_l$ and $i'_1,\dots,i'_l$. Let $1\leq j\leq l$ and define
\begin{align*}
    i_j&:=\min\{i\in\{1,\dots,m\},\Phi_i=\Theta_j\}\\
    i'_j&:=\min\{i\in\{1,\dots,m\},\Phi_i=\iota'(\Theta_j)\}
\end{align*}

Note that even if these definitions share some similarities with the one in \cite{LM}, there is a subtle difference. The definition of $i_j'$ uses the involution $\iota$ and not the contragredient.

We can now define
$$\tilde \m=\tilde{\Phi}_1+\dots+\tilde{\Phi}_m
$$

with the following formula:
$$
\tilde{\Phi}_i = \left\{
  \begin{array}{ll}
    p(\Phi_i) & \mbox{if } i\notin\{i_1,\dots,i_l\}\mbox{ and }i\notin\{i'_1,\dots,i'_l\} \\
    p(\Phi_i)^+ & \mbox{if } i\in\{i_1,\dots,i_l\}\mbox{ and }i\notin\{i'_1,\dots,i'_l\}\\
    ^+p(\Phi_i) & \mbox{if } i\notin\{i_1,\dots,i_l\}\mbox{ and }i\in\{i'_1,\dots,i'_l\}\\
    ^+p(\Phi_i)^+ & \mbox{if } i\in\{i_1,\dots,i_l\}\mbox{ and }i\in\{i'_1,\dots,i'_l\}
  \end{array}
\right.
$$

\begin{rem}\label{rem bar m sym}
    An important remark is that by construction $\tilde \m$ is a symmetrical multisegment. Even if we used $\iota'$ instead of the contragredient, it does not matter since both of these applications on $\USeg$ give the contragredient on $\Seg$ through the projection $p$.
\end{rem}

We now define the signs $\tilde \e$. Let us take $1\leq i\leq m$ such that $c(\tilde\Phi_i)=0$.

\begin{enumerate}
    \item  If there exists $1\leq j\leq l$ such that $\tilde \Phi_i=\tilde \Phi_{i_j}$ with $c(\Theta_j)=0$, define $\tilde \e(\tilde \Phi_i)=\e_0'\cdot\e(\Theta_j)$.
    \item  If there exists $1\leq j\leq l$ such that $\tilde \Phi_i=\tilde \Phi_{i_j}$ with $c(\Theta_j)=-1/2$, then we necessarily have $l>j$ and $c(\Theta_{j+1})=0$ (see Lemma \ref{lemma sign tilde}) and we define $\tilde \e(\tilde \Phi_i)=-\e(\Theta_{j+1})$. 
    \item Otherwise, we necessarily have $\tilde \Phi_i\in\m$ and we define $\tilde \e(\tilde \Phi_i)=\e_0'\cdot\e(\tilde \Phi_i)$.
    $\tilde \e(\tilde \Phi_i)$.

\end{enumerate}

\begin{rem}
    Here are some remarks:
    \begin{itemize}
        \item In the situation of the case (2), if there are no centered segments in $\m$ ending in $e(\Theta_j)+1$, then $\Theta_{j+1}=\Theta_j^\vee$ and so we should have $\tilde \Phi_i={}^+p(\Theta_j)^+$ which implies that $\tilde \Phi_i$ is not centered.
        \item Note that cases (1) and (2) never happen simultaneously and case (2) can only happen when $\s=s$ with $b(s)\le 0$.
    \end{itemize}
\end{rem}

The last step to define $\CT_\s(\m,\e)=(\bar \m,\bar \e)$ is to add points, or sometimes segments of length 2, to the datum $(\tilde \m,\tilde \e)$. 

Start by defining $T=\{e(\Theta_1)+1,\ldots,e(\Theta_l)+1\}$. By definition we know that $\min(T)\ge b$. Now, since $\s(\m,\e)\le' \s$ and since the ending point of $\s(\m,\e)$ is the largest one in $\m$, we know that $e(\Theta_l)\le e$ and so $\max(T)\le e+1$ but as we will see it after (in Lemma \ref{lemma bound ending point}), we have $e(\Theta_l)\le e-1$ and then $T\subseteq [b,e]$.

Now define
$$I=\{i\in [b,e]\backslash T~|~ i\le-b+1\mbox{ and } -i+1\notin T\}$$ 
and $I^c$ its complement in $[b,e]\backslash T$. 

Finally we define a symmetrical multisegment $\delta$. First we set $\delta=\varnothing$ if $\s=s\in\Seg^{>0}$. Now assume $\s=(s,\eta)\in\Seg^{=0}\times\{-1,1\}$. If $\rho$ is of the same type as $G$ we define $\delta=[0,0]$. If $\rho$ is not of the same type as $G$ then $\delta=[-1/2,1/2]$ if $\eta=(-1)^{n_0}\cdot \eta_0$ and $\delta=[-1/2,-1/2]+[1/2,1/2]$ if $\eta=(-1)^{n_0+1}\cdot \eta_0$.

Define:
$$\bar \m=\tilde \m+\sum_{i\in I}[i-1,i]+\sum_{i\in I^c}([i,i]+[-i,-i])+\delta$$

For any $i\in[b,e]\backslash T$, let $s_i$ be $[i-1,i]$ if $i\in I$ and $[i,i]$ otherwise so that:
$$\bar \m=\tilde \m+\sum_{i\in [b,e]\backslash T}s_i+\delta$$

Note that for any $i\in[b,e]\backslash T$, we have $(i\in I)\iff (-i+1\in I)$ but since $[i-1,i]^\vee=[-i,-i+1]$ it means that for any $i\in[b,e]\backslash T$, $s_i$ has length 2 if and only if $s_{1-i}$ has length 2. So, since $\tilde \m$ is symmetric (see Remark \ref{rem bar m sym}), $\bar \m$ is also symmetric.

\begin{rem}\label{rem I}
    If $\s=s\in \Seg^{>0}$, we have $b=b(s)$ and then $b(s)\ge 1$ implies that $[b,-b+1]$ is empty and so is $I$ which implies that we get $\bar \m$ from $\tilde \m$ only by adding some points.

    If $\s=(s,\eta)\in\Seg^{=0}\times \{1,-1\}$, we have $b=\tilde 1$ that is $0$ when $\rho$ is of the same type as $G$ and $1/2$ otherwise. In both cases we get $I=\varnothing$.
\end{rem}

Finally we need to define $\bar \e$. Let $\eta_0$ be either $\e([0,0])$ if $[0,0]\in\m$, or $\e([-1/2,1/2])$ if $[-1/2,1/2]\in$ or 1 otherwise. Take $\Delta\in\bar \m$ a centered segment. If  $\Delta\notin\{[0,0],[-1/2,1/2]\}$ then $\Delta\in\tilde \m$, define $\bar \e(\Delta)=\tilde \e(\Delta)$.
If $[0,0]\in\bar \m$:
\begin{itemize}
    \item If $\s=(s,\eta)\in\Seg^{=0}\times \{-1,1\}$, define:
    
    $
\bar \e([0,0]) = \left\{
  \begin{array}{ll}
    \e_0'\cdot\eta_0=-\eta_0 & \mbox{if } \eta=(-1)^{n_0+1}\cdot \eta_0 \\
    -\e_0'\cdot\eta_0=\eta_0 & \mbox{otherwise }
    \end{array}
\right.
$
\item If $\s=s\in\Seg^{>0}$, define:

$
\bar \e([0,0]) = \left\{
  \begin{array}{ll}
    -\e_0'\cdot\eta_0=-\eta_0 & \mbox{if there exists } 1\le i\le l\mbox{ such that } p(\Theta_i)=[0,0]  \\
    \e_0'\cdot\eta_0=\eta_0 & \mbox{otherwise }
  \end{array}
\right.
$

\end{itemize}

Finally, if $[-1/2,1/2]\in\bar \m$, define:

$
\bar \e([-1/2,1/2]) = \left\{
  \begin{array}{ll}
    \e_0' & \mbox{if there exists } 1\le i\le l\mbox{ such that } p(\Theta_i)=[-1/2,1/2]  \\
    \e_0'\cdot\eta_0 & \mbox{otherwise }
  \end{array}
\right.
$

 \begin{rem} 
 Assume that $\rho$ is of the same type as $G$ and $[0,0]\in\m$.
 Note that if $\s=(s,\eta)\in\Seg^{=0}\times \{-1,1\}$, we have $p(\Theta_1)=[0,0]$ if and only if $\eta\neq (-1)^{n_0+1}\cdot \eta_0$. So, in this case, we can just define $\bar \e([0,0])=-\e_0'\cdot \e([0,0])$ if there exists $1\le i\le l$ such that $p(\Theta_i)=[0,0]$ and  $\bar \e([0,0])=\e_0'\cdot \e([0,0])$ otherwise. Now assume that $\rho$ is not of the same type as $G$ and $[-1/2,1/2]\in\m$, then if $\s=(s,\eta)\in\Seg^{=0}\times \{-1,1\}$ we always have $\bar \e([-1/2,1/2])=\e_0'$ because $p(\Theta_1)=[-1/2,1/2]$ if and only if $\eta_0=-1$.
 \end{rem}

This ends the definition of our symmetrical data $(\bar \m,\bar \e)$.

\begin{deff}
    We define $\CT_\s(\m,\e)=(\bar \m,\bar \e)$.
\end{deff}

We now provide an example:
\tikzset{
    segment/.style={line width=1.5pt, black},
    segment blue/.style={line width=1.5pt, blue},
    segment red/.style={line width=1.5pt, red},
    start pt/.style={circle, fill=blue, inner sep=1.5pt},
    end pt/.style={circle, fill=red, inner sep=1.5pt},
    point halo/.style={circle, fill=black, inner sep=2pt},
    point blue/.style={circle, fill=blue, inner sep=2pt},
    point red/.style={circle, fill=red, inner sep=2pt},
    theta link/.style={line width=3pt, green!50, opacity=0.7},
    point theta/.style={circle, fill=green!50, opacity=0.7,inner sep=2pt}
}

\begin{ex}
We reproduce a detailed manual calculation to further illustrate the algorithm. Let $\mathfrak{s} = [-3, 5] \in \Seg^{>0}$, meaning $b=-3$, $e=5$, and $\e'_0=1$. We apply $\CT_\s$ to the symmetrical datum $\mathfrak{m} = [-5,-4] + 2\cdot[-1,1] + [0,0] + [4,5]\in\DSymm^{\le'[-3, 5]} $, with signs $\e([-1,1])=-1$  and $\e([0,0])=1$.
\begin{enumerate}[noitemsep,topsep=0pt]
    \item \textbf{Initial sequence $\Theta^\s$:} We compute $\Theta^{b-1}(\mathfrak{m},\e) = \Theta^{-4}(\mathfrak{m},\e)$. We order the segments of $y=s(\mathfrak{m},\e)$ according to the Lanard-Mínguez order: we write $s=\Phi_1+\Phi_2+\Phi_3+\Phi_4+\Phi_5$ with $\Phi_1 = [-5,-4]^{\le 0}$, $\Phi_2 = [-1,1]^{\le 0}$, $\Phi_3 = [0,0]^{=0}$, $\Phi_4 = [-1,1]^{\ge 0}$ and $\Phi_5 = [4,5]^{\ge 0}$.

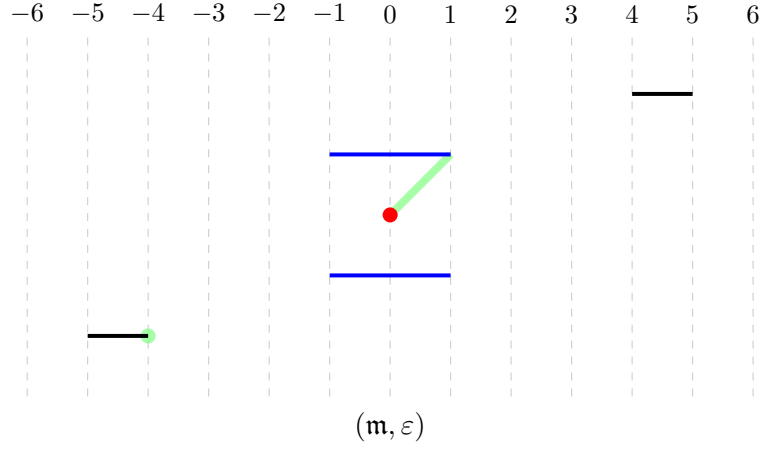
\begin{center}
\begin{tikzpicture}[scale=0.8]
    \foreach \x in {-6,...,6} {
        \draw[dashed, gray!40] (\x, -3.0) -- (\x, 3.0);
        \node[above] at (\x, 3.0) {\small $\x$};
    }
    
    \draw[theta link] (0, 0.0) -- (1, 1.0);
    \node[point theta] at (-4, -2.0) {}; \node[right] at (-4, -2.0) {};

    \draw[segment] (-5, -2.0) -- (-4, -2.0);
    \fill[start pt] (-5, -2.0) circle; \fill[end pt] (-4, -2.0) circle;
    
    \draw[segment blue] (-1, -1.0) -- (1, -1.0);
    \fill[start pt] (-1, -1.0) circle; \fill[end pt] (1, -1.0) circle;

    \node[point red] at (0, 0.0) {}; \node[right] at (0.1, 0.0) {};
    
    \draw[segment blue] (-1, 1.0) -- (1, 1.0);
    \fill[start pt] (-1, 1.0) circle; \fill[end pt] (1, 1.0) circle;
    
    \draw[segment] (4, 2.0) -- (5, 2.0);
    \fill[start pt] (4, 2.0) circle; \fill[end pt] (5, 2.0) circle;

    \node at (0, -3.5) {$(\m,\e)$};
 
\end{tikzpicture}
\end{center}

    The smallest segment ending at $-4$ is $\Theta_1 = \Phi_1 = [-5,-4]$. No segment ends at $-3$, so $\Theta^{-4}$ stops here with $f = -4 < 0$. 
    We then append $\Theta^{0,0}(\mathfrak{m},\e)$. The smallest segment in $\mathcal{S}^{=0} \cup \mathcal{S}^{\ge 0}$ ending at $0$ is $\Theta_2 = \Phi_3 = [0,0]$. The smallest segment strictly greater than $\Theta_2$ ending at $1$ is $\Theta_3 = \Phi_4 = [-1,1]$. 
    Thus, $\Theta = [-5,-4] + [0,0] + [-1,1]$. The set of ends plus one is $T = \{-3, 1, 2\}$.
    
    \item \textbf{Extension $\tilde{\mathfrak{m}}$:} The indices of $\Theta$ in $y$ are $i_1=1, i_2=3, i_3=4$. Applying $\iota'$ yields $\iota'(\Phi_1) = \Phi_5 \implies i'_1=5$, $\iota'(\Phi_3) = \Phi_3 \implies i'_2=3$, and $\iota'(\Phi_4) = \Phi_2 \implies i'_3=2$. 
    \begin{itemize}
        \item $\{i_j\} \setminus \{i'_j\} = \{1, 4\} \implies \tilde{\Phi}_1 = [-5,-3]$ and $\tilde{\Phi}_4 = [-1,2]$.
        \item $\{i'_j\} \setminus \{i_j\} = \{5, 2\} \implies \tilde{\Phi}_5 = [3,5]$ and $\tilde{\Phi}_2 = [-2,1]$.
        \item $\{i_j\} \cap \{i'_j\} = \{3\} \implies \tilde{\Phi}_3 = [-1,1]$.
    \end{itemize}
    This yields $\tilde{\mathfrak{m}} = [-5,-3] + [-2,1] + [-1,1] + [-1,2] + [3,5]$.
    
 
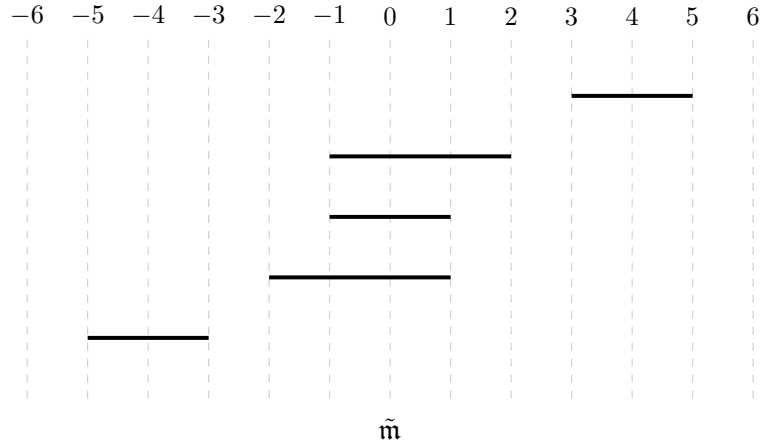
\begin{center}       
\begin{tikzpicture}[scale=0.8]
    \foreach \x in {-6,...,6} {
        \draw[dashed, gray!40] (\x, -3.0) -- (\x, 3.0);
        \node[above] at (\x, 3.0) {\small $\x$};
    }

    \draw[segment] (-5, -2.0) -- (-3, -2.0);
    \fill[start pt] (-5, -2.0) circle; \fill[end pt] (-3, -2.0) circle;
    
    \draw[segment] (-2, -1.0) -- (1, -1.0);
    \fill[start pt] (-2, -1.0) circle; \fill[end pt] (1, -1.0) circle;

    \draw[segment] (-1, 0.0) -- (1, 0.0);
    \fill[start pt] (-1, -1.0) circle; \fill[end pt] (1, -1.0) circle;
    
    \draw[segment] (-1, 1.0) -- (2, 1.0);
    \fill[start pt] (-1, 1.0) circle; \fill[end pt] (2, 1.0) circle;
    
    \draw[segment] (3, 2.0) -- (5, 2.0);
    \fill[start pt] (3, 2.0) circle; \fill[end pt] (5, 2.0) circle;

    \node at (0, -3.5) {$\tilde \m$};
 
\end{tikzpicture}
\end{center}
    
    \item \textbf{Sets $I, I^c$ and $\bar{\mathfrak{m}}$:} We evaluate $[b,e] \setminus T = [-3,5] \setminus \{-3,1,2\} = \{-2, -1, 0, 3, 4, 5\}$. We keep elements $i$ satisfying $i \le -b+1=4$ and $1-i \notin T$.
    This filters the set down to $I = \{-2, 3\}$, leaving $I^c = \{-1, 0, 4, 5\}$.
    The set $I$ contributes $[-3,-2] + [2,3]$. 
    The set $I^c$ contributes $[-1,-1] + [1,1] + 2\cdot[0,0] + [4,4] + [-4,-4] + [5,5] + [-5,-5]$. 
     Adding everything together gives the final $\bar{\mathfrak{m}}$.
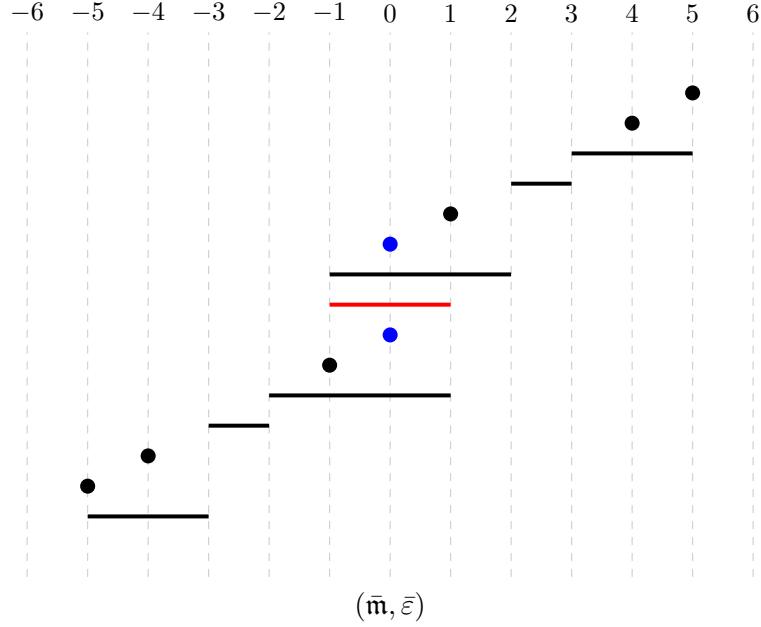
     \begin{center}
     \begin{tikzpicture}[scale=0.8]
    \foreach \x in {-6,...,6} {
        \draw[dashed, gray!40] (\x, -4.5) -- (\x, 4.5);
        \node[above] at (\x, 4.5) {\small $\x$};
    }


    \draw[segment] (-5, -3.5) -- (-3, -3.5);
    \fill[start pt] (-5, -5.0) circle; \fill[end pt] (-3, -5.0) circle;
    
     \node[point halo] at (-5, -3) {}; \node[right] at (0.1, 0.0) {};
     
     \node[point halo] at (-4, -2.5) {}; \node[right] at (0.1, 0.0) {};
     
     \draw[segment] (-3, -2.0) -- (-2, -2.0);
    \fill[start pt] (-3, -2.0) circle; \fill[end pt] (-2, -2.0) circle;
    
    \draw[segment] (-2, -1.5) -- (1, -1.5);
    \fill[start pt] (-2, -1.5) circle; \fill[end pt] (1, -1.5) circle;
    
     \node[point halo] at (-1, -1.0) {}; \node[right] at (0.1, 0.0) {};
     
      \node[point blue] at (0, -0.5) {}; \node[right] at (0.1, 0.0) {};

    \draw[segment red] (-1, 0.0) -- (1, 0.0);
    \fill[start pt] (-1, -1.0) circle; \fill[end pt] (1, -1.0) circle;
    
    
     \draw[segment] (-1, 0.5) -- (2, 0.5);
    \fill[start pt] (-1, 0.5) circle; \fill[end pt] (2, 0.5) circle;

    \node[point blue] at (0, 1) {}; \node[right] at (0.1, 0.0) {};
    
     \node[point halo] at (1, 1.5) {}; \node[right] at (0.1, 0.0) {};
    
     \draw[segment] (2, 2) -- (3, 2);
    \fill[start pt] (2, 2) circle; \fill[end pt] (3,2) circle;
    
    \draw[segment] (3, 2.5) -- (5, 2.5);
    \fill[start pt] (3, 2.0) circle; \fill[end pt] (5, 2.0) circle;
    
    \node[point halo] at (4, 3) {}; \node[right] at (0.1, 0.0) {};
    \node[point halo] at (5, 3.5) {}; \node[right] at (0.1, 0.0) {};

    \node at (0, -5) {$(\bar \m,\bar \e)$};
 
\end{tikzpicture}
\end{center}
    
    \item \textbf{Signs $\bar{\e}$:} The centered segment $[-1,1]$ originates from $\Theta_2 = [0,0]$. Its new sign is $\bar{\e}([-1,1]) = \e'_0 \cdot \e([0,0]) = 1$. For the 2 copies of $[0,0]$, the rule gives $\bar{\e}([0,0]) = -\e([0,0]) = -1$.

\end{enumerate}
       
\end{ex}

To finish this part we prove a short lemma we used to define $\tilde \e$.

\begin{lem}\label{lemma sign tilde}
    Let $1\le i\le m$ and assume there exists $1\le j\le l$ such that $\tilde \Phi_i=\tilde \Phi_{i_j}$ with $c(\Theta_j)=-1/2$, then we necessarily have $l>j$, and if $c(\tilde \Phi_i)=0$, then $c(\Theta_{j+1})=0$
\end{lem}

\begin{proof}
    If $\Theta_j$ has center $-1/2$ then $\Theta_j^\vee$ has center $1/2$ and $e(\Theta_j^\vee)=e(\Theta_j)+1$ so since $\Theta_j^\vee\succeq\Theta_j$ we have $l>j$. Now, if $\Theta_{j+1}=\Theta_j^\vee$ this implies $i_j=i'_{j+1}$, resulting in $\tilde \Phi_{i_j}={}^+p(\Theta_j)^+$ which is not centered. Then in this case we must have $\Theta_{j+1}\prec\Theta_j^\vee$ and since $c(\Theta_j^\vee)=1/2$ we have $c(\Theta_{j+1})=0$.
\end{proof}

\section{The duality result and its proof}\label{section proof}

\subsection{Main result}

Take $\s\in\Seg^{>0}\cup \Seg^{=0}\times \{1,-1\}$. The goal of this section, now that we have built our operator $\CT_\s$, is to prove that it acts as intended, namely as the exact dual of $\T_\s$. This is the content of the following theorem:

\begin{theorem}\label{main theorem}
Let $\s\in\Seg^{>0}\cup \Seg^{=0}\times \{1,-1\}$, then we have the following correspondence of operators under Aubert duality:

$$\CT_{\s}\stackrel{\AD}{\longleftrightarrow} \T_{\s}$$
By this, we mean that the following diagram commutes:
\[\begin{tikzcd}[sep=small]
	{\Symm^{\le'\s}} && {\DSymm^{\le'\s}} \\
	\\
	{\Symm^{\le'\s}} && {\DSymm^{\le'\s}}
	\arrow["\AD", tail reversed, from=1-1, to=1-3]
	\arrow["{\T_{\s}}"', from=1-1, to=3-1]
	\arrow["{\CT_{\s}}", from=1-3, to=3-3]
	\arrow["\AD", tail reversed, from=3-1, to=3-3]
\end{tikzcd}\]
\end{theorem}

\subsection{Sketch of the proof of the main result}

Let us outline the proof of Theorem \ref{main theorem}. Start by fixing $\s\in\Seg^{>0}\cup\Seg^{=0}\times\{1,-1\}$ and $(\m,\e)\in \DSymm^{\le'\s}$.

Denote $(\bar \m,\bar \e)=\CT_\s(\m,\e)$. The idea is to apply the Lanard-Mínguez algorithm once to the symmetrical datum $(\bar \m,\bar \e)$ in order to obtain the following equality:
$$\AD(\bar \m,\bar \e)=(\bar \m_1,\bar \e_1)+\AD(\bar \m^\#,\bar \e^\#)$$
where the data $(\bar \m_1,\bar \e_1)$ and $(\bar \m^\#,\bar \e^\#)$ are constructed from $(\bar \m,\bar \e)$ by Lanard and Mínguez in \cite{LM}. By definition of $\s(\m,\e)$ we can reformulate it:
$$\AD(\bar \m,\bar \e)=\T_{\s(\bar \m,\bar \e)}(\AD(\bar \m^\#,\bar \e^\#))$$

So we are now left to prove that $\s(\bar \m,\bar \e)=\s$ and $(\bar \m^\#,\bar \e^\#)=(\m,\e)$ so that:
$$\AD(\CT_\s(\m,\e))=\AD(\bar \m,\bar \e)=\T_\s(\AD(\m,\e))$$
which is exactly the content of Theorem \ref{main theorem} if we apply the involution $\AD$ on both sides.

This strategy of proof highlights the following fact: if $\#$ is the map $(\m,\e)\mapsto(\m^\#,\e^\#)$ defined in \cite{LM}, then the fiber $\#^{-1}(\m,\e)$ is parametrized by the elements $\s\in\Seg^{>0}\cup \Seg^{=0}\times \{1,-1\}$ such that $\s\ge'\s(\m,\e)$, and the datum in the fiber corresponding to such an $\s$ is $\CT_\s(\m,\e)$. In other words, our operators are sections of the map $\#$.

Now let us say a few words about the way to prove that $(\bar \m^\#,\bar \e^\#)=(\m,\e)$. The datum $(\bar \m,\bar \e)$ is obtained by adding some new segments of length 1 or 2 and by extending some segments of $(\m,\e)$; on the other hand, $(\bar\m^\#,\bar \e^\#)$ is obtained by shortening some segments of $(\bar \m,\bar \e)$. Roughly speaking, the idea is to show that the segments we extend or add by applying $\CT_\s$ are the ones we shorten by applying $\#$.

\subsection{Proof}

Recall that $\s$ and $(\m,\e)$ have been fixed in the last subsection. We denote by $(\bar \m,\bar \e)$ the datum $\CT_\s(\m,\e)$. Again, we write $(y,\e)=s(\m,\e)$ and  $y=\Phi_1+\ldots+\Phi_m$
with $\Phi_1\preceq \dots\preceq\Phi_m$. In order to lighten the notation we will write $\Theta=\Theta^\s(\m,\e)$ and $\Theta=\Theta_1+\ldots+\Theta_l$ with the endings being increasing. Now we also write $(\bar y,\bar \e)=s(\bar \m,\bar \e)$ and $\bar y= \Lambda_1+\ldots+ \Lambda_{k}$ with $\Lambda_1\preceq \dots\preceq\Lambda_{k}$.

When exposing their algorithm in \cite{LM}, Lanard and Mínguez defined from $(\m,\e)\in \Symm_\rho^\e(G)$ the tuple $(l,(\Delta_j)_{1\leq j \leq \bar l},n_0,\e_0,\e_1,\m_1,(i_j)_j,(i'_j)_j,\e^\#,\m^\#)$. In the following, we need to apply the Lanard-Mínguez algorithm to $(\bar \m,\bar \e)$, so we introduce the tuple $$(\bar l,(\bar \Delta_j)_{1\leq j \leq l},\bar n_0,\bar \e_0,\bar \e_1,\bar \m_1,(\bar i_j)_j,(\bar i'_j)_j,\bar \e^\#,\bar \m^\#)$$ as the one attached to $(\bar \m,\bar \e)$ with the same procedure as the one exposed in \cite{LM}. 

Recall that $e$ is the ending point of the underlying segment of $\s$ which is fixed, and $b$ is either its beginning if its center is strictly positive, and $\tilde 1$ if it is centered.

The first step is to show that $p(\bar \Delta_1)$ is either $p(\Theta_l)$ if $e\in T$ and $s_e$ otherwise. This will essentially be the content of the Proposition \ref{prop key tilde Theta and Delta}. To this end we first prove this preliminary lemma that also justifies that $T\subseteq [b,e]$.

\begin{lem}\label{lemma bound ending point}
    If $\Theta$ is not empty, we have $e(\Theta_l)\le e-1$.
\end{lem}

\begin{proof}
Note that $e(\Theta_l)$ has to be smaller than or equal to the ending point of the underlying segment of $\s(\m,\e)$, which is the largest ending point in $\m$. Since $(\m,\e)\in\DSymm^{\le'\s}$, we have $\s(\m,\e)\le'\s$ so the largest ending point in $\m$ is smaller than or equal to $e$. If this ending point is strictly smaller than $e$, both points of the Lemma are true. Now assume the ending point of the underlying segment of $\s(\m,\e)$ is $e$. If $\s(\m,\e)$ is centered, because of the definition of $<'$, the only possibility is that $\s(\m,\e)=\s=(s,\eta)$ since we assumed they share the same ending point $e$. However according to Lanard-Mínguez algorithm, we have $\eta=(-1)^{n_0+1}\cdot \eta_0$ so, in this case, by definition, $\Theta=\Theta^\s(\m,\e)=\varnothing$. Thus, we may assume that $\s(\m,\e)=[b',e]\in \Seg^{>0}$ with $b'\ge b$ the beginning of $s$ which may be $-e$. We start by considering the initial sequence $(\Delta_i)_{1\le i\le k}$ defined from $(\m,\e)$ by Lanard and Mínguez. We have $k=b'-e+1$, $e(\Delta_1)=e$ and $e(\Delta_k)=b'$. By definition of $\Delta$, there are no segments in $\m$ smaller than $\Delta_k$ ending in $b'-1$.

Assume first that $\Theta=\Theta^{b-1}(\m,\e)$. If $e(\Theta_l)< b'-1\le e-1$ then we are done. Now if $e(\Theta_l)\ge b'-1$, there exists $1\le i\le l$ such that $e(\Theta_i)=b'-1$ but then, by definition of $\Delta$ and $k$, we have $\Theta_i\succ\Delta_k$. Consequently, $\Theta_{i+j}\succ\Delta_{k-j}$ for any $j\ge 0$ such that $i+j\le l$ and $k-j\ge 1$. If $e(\Theta_l)=e$ then $\Theta_l\succeq \Theta_{l-1}\succ\Delta_1$ which is impossible since $\Delta_1$ is the largest segment in $\m$ that ends in $e$. Then $e(\Theta_l)\le e-1$.

Now assume that $\Theta=\Theta^{\tilde 1-1,0}(\m,\e)$. If $b'> \tilde 1$, the exact same argument applies. Now if $b'\le\tilde 1$, there exists $1\le j\le k$ such that $e(\Delta_j)=\tilde 1$. Since we assumed that $\e_0=1$, we have either $\Delta_j\in \mathcal S^{\le'0}$ or $\Delta_j=[-1/2,1/2]$ with sign $\e(\Delta_j)=1$. Then, by definition of $\Theta^{\tilde 1-1,0}$ we have either $\Theta_1\succ\Delta_j$ (because $\Theta_1$ is in $\mathcal S^{=0}\cup\mathcal S^{\ge 0}$), or $\Theta_1=\varnothing$. In the second case the result is direct and in the first one, as before we can deduce that if $e(\Theta_l)=e$ we would have $\Theta_l\succ \Delta_1$ which is impossible.

Finally, if $\Theta=\Theta^{b-1}(\m,\e)+\Theta^{\tilde 1-1,0}(\m,\e)$. If $\Theta^{\tilde 1-1,0}(\m,\e)$ then we already treated that case so we can assume that it is non-empty and then we can do the same thing as in the previous case.
\end{proof}

Here is another lemma before the key proposition that will lead us to the proof of Theorem \ref{main theorem}.

\begin{lem}\label{lem induction Delta easy cases} We have these two points:
\begin{enumerate}
    \item For any $1\le i\le \bar l-1$, if there exists $2\le j\le l$ such that $p(\bar \Delta_i)=\tilde \Theta_j$ then $p(\bar \Delta_{i+1})=\tilde \Theta_{j-1}$. Moreover, if $\bar \Delta_i$ is centered, we either have $c(\Theta_r)=0$ and $\bar \Delta_i\in\mathcal S^{=0}$ or $c(\Theta_r)=-1/2$ and $\bar \Delta_i\in \mathcal S^{\le'0}$.
    
    \item For any $1\le i\le \bar l-1$, if $\ell(\bar \Delta_i)=1$ and $[b(\bar \Delta_i)-1,e(\bar \Delta_i)-1]\in\bar \m$ then $p(\bar \Delta_{i+1})=[e(\bar \Delta_i)-1,e(\bar \Delta_i)-1]$. Now if $\ell(\bar \Delta_i)=2$ and $[e(\bar \Delta_i)-1,e(\bar \Delta_i)-1]\in\bar \m$ then $p(\bar \Delta_{i+1})$ is either $[b(\bar \Delta_i),e(\bar \Delta_i)-1]$ if this segment exists in $\bar \m$ and $[b(\bar \Delta_i),e(\bar \Delta_i)-1]$ otherwise.
\end{enumerate}
    
\end{lem}

\begin{rem}
    Note that the condition $e(\Theta_{j-1})=e(\Theta_j)-1$ in the first point may not be verified only if $\Theta=\Theta^{b-1}(\m,\e)+\Theta^{\tilde 1-1,0}(\m,\e)$ with both summand being non-empty and $j-1$ being the number of non-empty segments in $\Theta^{b-1}(\m,\e)$.
\end{rem}

\begin{proof}
Take such an $i$ and such a $j$ as in the first point. 
\begin{itemize}[noitemsep,topsep=0pt]
    \item Suppose $\Theta_{j-1}$ and $\Theta_j$ are both in $\mathcal S^{\ge 0}$. Then $\Delta_i\in\mathcal S^{\ge 0}$  by assumption. We have $c(\tilde \Theta_{j-1})=c( \Theta_{j-1})+1/2>0$ and so if $\alpha\in\bar y$ is such that $p(\alpha)=\tilde \Theta_i$, then $\alpha\in\mathcal S^{\ge 0}$.
    
    Now we show that $\bar \Delta_{i+1}=\alpha$. Since $\alpha$ and $\bar \Delta_i$ are both in $\mathcal S^{\ge 0}$ we have $\alpha\preceq\bar \Delta_i$ if and only if $p(\bar \Delta_i)=\tilde \Theta_j\ge \tilde \Theta_{j-1}$ but this is true because $p(\Theta_j)\ge p(\Theta_{j-1})$ since $\Theta_j$ and $\Theta_{j-1}$ are both in $\mathcal S^{\ge 0}$. The last point is the maximality of $\alpha$. Take $u\in\bar y$ ending in $e(\tilde \Theta_{j-1})$ with $u\succ\alpha$. Then by definition of $\bar \m$ we have $p(u)\in\m$ and since $u\succeq \alpha$ and $c(\tilde \Theta_{j-1})>0$ we have $c(u)>0$ and so the only $\alpha\in y$ such that $p(\alpha)=p(u)$ is in $\mathcal S^{\ge 0}$, i.e. $u\in y$. Consequently, $u\succeq \Theta_{j-1}$ and since $e(u)=e(\Theta_{j-1})+1$ we have $u\succeq \Theta_j$ and then $p(u)\succ \tilde \Theta_j$ so $u\succ \bar \Delta_i$. This proves that $\bar \Delta_{i+1}=\alpha$ and then $p(\bar \Delta_{i+1})=p(\alpha)=\tilde \Theta_{j-1}$.

    \item Suppose $\Theta_{j-1}\in\mathcal S^{=0}$ and $\Theta_{j}\in \mathcal S^{\ge 0}$. Then $c(\tilde \Theta_j)=c(\Theta_j)+1/2>0$. Take $\alpha\in \bar y$ such that $p(\alpha)=\tilde \Theta_{j-1}$. Since $\tilde \Theta_{j-1}$ is centered, we may have $\alpha \in \mathcal S^{\ge 0}\cup \mathcal S^{=0}\cup \mathcal S^{\le'0}$. The multiplicity $m_\m(p(\Theta_j))$ has to be even and so, by definition of $\tilde \m$, we have $m_{\bar m}(\tilde \Theta_{j-1})=m_{\m}(p(\Theta_j))-2+1$  odd and then we can take $\alpha\in\mathcal S^{=0}$.
    
    We now show that $\bar \Delta_{i+1}=\alpha$. We have $\alpha\in\mathcal S^{=0}$ so $\alpha\preceq \bar \Delta_i$. Note that even if $\Theta_j$ is centered, $\tilde \Theta_j$ is not. Now of course $e(\alpha)=e(\Theta_j)=e(\tilde \Theta_j)-1=e(\bar \Delta_i)-1$. Now take $u\succ \alpha$ with $e(u)=e(\alpha)$. 
    Either $p(u)=p(\alpha)$ with $u\in\mathcal S^{\ge 0}$ or either $e(u)>e(\alpha)$. In both cases, if $u\preceq \bar \Delta_j$, as in the previous points of this proof, it would imply the existence of a segment in $y$ contradicting the minimality of $\Theta_j$. Thus $u\succ \bar \Delta_j$ and $\bar \Delta_{j-1}=\alpha.$

    \item Suppose $\Theta_{j-1}$ and $\Theta_{j}$ are both in $\mathcal S^{= 0}$. We have $c(\tilde \Theta_{j-1})=0$  so if $\alpha\in\bar y$ is such that $p(\alpha)=\tilde \Theta_{j-1}$ we may have $\alpha \in \mathcal S^{\ge 0}\cup \mathcal S^{=0}\cup \mathcal S^{\le'0}$. Since $\Theta_{j-1}\in \mathcal S^{=0}$ we have $\tilde \Theta_{j-1}=p(\Theta_j)$ and since $\Theta_j\in\mathcal S^{=0}$ we have $m_\m(p(\Theta_j))$ is odd so by definition of $\m$ we have $m_{\bar \m}(\tilde \Theta_{j-1})=m_\m(p(\Theta_j))+1-1=m_\m(p(\Theta_j))$ odd. We then let $\alpha$ be in $\mathcal S^{=0}$ and we show that $\alpha=\bar \Delta_{i+1}$. 
    
    Both $\alpha$ and $\bar \Delta_i$ are in $\mathcal S^{=0}$ and $e(\alpha)=e(\bar \Delta_i)-1$ so $\alpha\prec\bar \Delta_i$. Now, $\bar \e(\alpha)= \e_0'\cdot\e(\Theta_{j-1})=- \e_0'\cdot\e(\Theta_j)=-\e(\tilde \Theta_j)=-\e(\bar \Delta_i)$. And finally, this case, there are no possible segment strictly between $\alpha$ and $\bar \Delta_i$ so we have $\tilde \Delta_{i+1}=\alpha$ and $p(\bar \Delta_{i+1})=p(\alpha)=\tilde \Theta_{j-1}$.
    
    \item Suppose $\Theta_{j-1}\in\mathcal S^{\le0}$ and $\Theta_{j}\in \mathcal S^{\ge 0}$. Take $\alpha\in\bar y$ such that $p(\alpha)=\tilde \Theta_{j-1}$. If $c(\Theta_{j-1})<1/2$ then $\alpha$ has to be in $\mathcal S^{\le'0}$. If $c(\Theta_{j-1})\in\{-1/2,0\}$, then $c(\tilde \Theta_{j-1})=0$ and so we may have $\alpha \in \mathcal S^{\ge 0}\cup \mathcal S^{=0}\cup \mathcal S^{\le'0}$ and we have to explain how to choose $\alpha$. If $\Theta_j$ is centered then $m_\m(p(\Theta_j))$ is even and  we have $\tilde \Theta_{j-1}=p(\Theta_j)$ so by definition of $\bar \m$ we have $m_{\bar \m}(\tilde \Theta_{j-1})=m_\m(p(\Theta_j))-2+1$ odd if $c(\Theta_{j-1})=0$ and $m_{\bar \m}(\tilde \Theta_{j-1})=m_\m(p(\Theta_j))-2+2$ even if $c(\Theta_{j-1})=-1/2$. In the first case  we take $\alpha\in\mathcal S^{=0}$ and in the second one we take $\alpha\in\mathcal S^{\le'0}$. Now finally, if $c(\Theta_j)>0$, we have $\tilde \Theta_j\notin \m$ so in particular $m_\m(\tilde \Theta_j)=0$ which is even so we do the same.

    We now prove that $\alpha=\bar \Delta_{i+1}$ start with the case $c(\alpha)<0$. First remark that in all cases, we have $e(\alpha)=e(\tilde \Theta_{j-1})=e(\tilde \Theta_j)-1=e(\bar \Delta_i)-1$ and $c(\bar \Delta_i)=c(\Theta_{j})+1/2>0$. Now suppose $c(\tilde \Theta_{j-1})\le-1/2$. We have $\alpha\in\mathcal S^{\le'0}$ and so $\alpha\prec \bar \Delta_i\in\mathcal S^{\ge 0}$. As in the previous cases, using $\Theta_j$ minimality in $\m$, we prove that if $u\in\bar y$ ends in $e(\alpha)$ and $u\succ \alpha$ then $u\succ \bar \Delta_i$ and so $\bar \Delta_{i+1}=\alpha$. Now if $c(\Theta_{j-1})\in\{-1/2,0\}$ this is very similar to the case $(\Theta_{j-1},\Theta_j)\in\mathcal S^{=0}\times \mathcal S^{\ge 0}$. The only specificity is to justify that $\bar \Delta_{i+1}\neq \alpha^\vee\in\mathcal S^{\ge 0}$. To do that we show $\alpha^\vee\succ \bar \Delta_i$. 
    We have $p(\bar \Delta_i)=\tilde \Theta_j$. Since $c(\Theta_j)=-1/2$, we have $\Theta_{j+1}\preceq \Theta_j^\vee$. We deduce that $\tilde \Theta_j\le{}^+p(\Theta_j^\vee)^+$. The segment $p(\alpha)$ is centered and ends in $e(\Theta_{j-1})+1=e({}^+p(\Theta_{j-1}^\vee)^+)+2$ so $p(\alpha)\ge {}^+p(\Theta_j^\vee)^+$ so $\alpha^\vee\succ \bar \Delta_i$ since $\alpha^\vee\in\mathcal S^{\ge 0}$ and $p(\bar \Delta_i)=\tilde \Theta_j\le {}^+p(\Theta_j^\vee)^+$.

    \item Suppose $\Theta_{j-1}\in\mathcal S^{\le0}$ and $\Theta_{j}\in \mathcal S^{= 0}$. This case is very similar to the previous one. The only new thing is a compatibility of signs we will check when $c(\Theta_{j-1})=-1/2$.
    In this case $\bar \e(\tilde \Theta_{j-1})=-\e(\Theta_j)=-\e_0\cdot\bar \e(\tilde \Theta_j)=-\e_0\cdot\e(\bar \Delta_{i})$. But when $\s=(\s,\eta)$ the elements of $\Theta$ are all in $\mathcal S^{=0}\cup \mathcal S^{\ge 0}$ so we have $\e_0=1$ and then $\bar \e(\alpha)=\bar \e(\tilde \Theta_{j-1})=-\e(\bar \Delta_{i})$ which is what we wanted for $\alpha$ to be $\bar \Delta_{i+1}$.

    \item Suppose $\Theta_{j-1}$ and $\Theta_{j}$ are both in $\mathcal S^{\le'0}$. Because of the order on $\mathcal S^{\le'0}$ and the definition of $\Theta$, they cannot be both centered and if $\Theta_{j-1}$ was centered we would have $\Theta_j\in \mathcal S^{\ge 0}$. Then we deduce that $\Theta_{j-1}$ is not centered and that $\Theta_j$ might be centered. The proof this is very similar to the previous case.
\end{itemize}

 The second point is  clear.
\end{proof}

\begin{prop}\label{prop key tilde Theta and Delta} Both following points are true.
    \begin{itemize}
        \item Assume that $\s=(s,\eta)\in \Seg^{=0}\times \{1,-1\}$. Then we get $\bar l=e-b+2$ and for any $1\le i\le e-b+1$, either $j:=e(\bar \Delta_i)=e-i+1\in T$, and there exists $1\le r\le l$ such that $p(\bar \Delta_i)=\tilde \Theta_r$, or $j\in[b,e]\backslash T$ and $p(\bar \Delta_i)=s_j=[j,j]$. If $p(\bar \Delta_i)=\tilde \Theta_r$ is centered, then either $c(\Theta_r)=0$ and $\bar \Delta_i\in \mathcal S^{=0}$ or $c(\Theta_r)=-1/2$ and $\bar \Delta_i\in\mathcal S^{\le'0}$. Finally, if $\rho$ is not of the same type as $G$ then $\bar \Delta_l=[1/2,1/2]$ if $\eta=(-1)^{n_0}$ and $\bar \Delta_{\bar l}=[-1/2,1/2]^{\le'0,=0}$ otherwise, and if $\rho$ is of the same type as $G$, we have $\bar \Delta_{\bar l}=[0,0]^{\ge 0,=0}$.

        \item Assume that $\s=s\in\Seg^{>0}$. Then we get $\bar l=e-b+1$ and for any $1\le i\le e-b+1$ let $j:=e(\bar \Delta_i)=e-i+1$. Either $j\in T$ and there exists $1\le r\le l$ such that $p(\bar \Delta_i)=\tilde \Theta_r$, or $j\notin T$ and $p(\bar \Delta_i)=[j,j]$ if $j\in I^c$ and $p(\bar \Delta_i)\in\{[j,j],s_j=[j-1,j]\}$ if $j\in I$. Finally, in all cases, $\bar \e_0=1$.
    \end{itemize}
     In both cases we get $\e_0'=\bar \e_0$ and $\bar \m_1=s$.

\end{prop}

\begin{rem}
    Note that when $\s=(s,\eta)\in\Seg^{=0}\times \{1,-1\}$ we always have $I=\varnothing$ (see Remark \ref{rem I}).
\end{rem}

\begin{proof}
    We prove this by induction on $1\le i\le e-b+1$.
    
    We begin with the initialization at $i=1$. According to Lemma \ref{lemma bound ending point} we have $e(\Theta_l)\le e-1$.
    Assume that $e\in T$ or equivalently that $e(\Theta_l)=e-1$. Let $\alpha\in\bar y$ such that $p(\alpha)=\tilde \Theta_l$. If $c(\tilde \Theta_l)>0$ then $\alpha\in\mathcal S^{\ge 0}$, if $c(\tilde \Theta_l)<0$ we have $\alpha\in\mathcal S^{\le'0}$ and finally if $c(\tilde \Theta_l)=0$, then $\alpha$ can be in $\mathcal S^{\ge 0}\cup \mathcal S^{=0}$ so we choose it to be the largest possible, and then $\alpha\in\mathcal S^{=0}\cup\mathcal S^{\ge 0}$. Now we show that $\bar \Delta_1=\alpha$. First of all, $e(\alpha)=e(\tilde \Theta_l)=e(\Theta_l)+1=e-1+1=e$. Now take $u\in \bar y$ such that $e(u)=e(\alpha)$ and $u\succeq \alpha$. If $u\neq \alpha$ we get $p(u)\in\m$ by definition of $\bar \m$. Since $u\succ\alpha$ and $e(u)=e(\alpha)$, we have $b(u)>b(\alpha)$ so even if $c(\alpha)=0$ we have $c(u)>0$ and then $u\succ\Theta_l$ in $y$. But it is impossible because if it was true we would get $\Theta_{l+1}=u$. So it means that $u=\alpha$ then $\alpha$ is the largest segment in $\bar y$ ending in $e$, i.e. $\bar \Delta_1=\alpha$ and so $p(\bar \Delta_1)=p(\alpha)=\tilde \Theta_l$. Now if $\Theta$ is empty or $e(\Theta_l)<e-1$, we get $e\in[b,e]\backslash T$ and $s_e\in\bar \m$. If $e\notin I$ we have $s_e=[e,e]$ so this is the largest possible segment ending in $e$ thus $p(\bar \Delta_1)=s_e$. Now if $e\in I$ (which happens if and only if $b=-e+1$ and $b\notin T$) we have $s_e=[e-1,e]\in\bar \m$ so $p(\bar \Delta_1)=[e,e]$ if $[e,e]\in\bar \m$ and $p(\bar \Delta_1)=[e-1,e]$ otherwise. 
    Now assume that the result has been proved for $1\le i\le e-b$;  we show it is still true for $i+1\le e-b+1$. Denote by $j$ the end of $\bar \Delta_i$ which is $e-i+1\ge b+1$, there are many possibilities depending on whether $j$ is in $T$ or $[b,e]\backslash T$, and the same for $j-1$. Note that if $(j,j-1)\in T^2$ then the result follows from the first point of Lemma \ref{lem induction Delta easy cases}, and if $(j-1,j)$ is in $([b,e]\backslash T)^2$ then the result follows from the second point of Lemma \ref{lem induction Delta easy cases}. The cases we need to treat now are the following ones: $(j,j-1)\in T\times [e,b]\backslash T$ and $(j,j-1)\in [e,b]\backslash T\times T$. First assume that $j\in T$ and $j-1\notin T$. Then there exists $1\le r\le l$ such that $p(\bar \Delta_i)=\tilde \Theta_r$ with $\Theta_r$ being either the first term of $\Theta^{b-1}(\m,\e)$ or the first term of $\Theta^{\tilde 1-1,0}(\m,\e)$. The first term is excluded because the end of $\bar \Delta_i$ is $e-i+1\ge b+1$ and the end of $\tilde \Theta_r$ is $b$.  Then $\Theta_r$ has to be the first term of $\Theta^{\tilde 1-1,0}(\m,\e)$ with $b< \tilde 1$ and so it is a segment of length $\le 2$ and then by the second point of Lemma \ref{lem induction Delta easy cases} we deduce the result. Finally assume that $j\in[b,e]\backslash T$ and $j-1\in T$. Then $p(\bar \Delta_i)$ is either $[j-1,j]$ or $[j,j]$. Since $j-1\in T$ there exists $\alpha\in\bar y$ and $1\le r\le l$ such that $p(\alpha)=\tilde \Theta_r$ with $e(\tilde \Theta_r)=j-1$. If $c(\tilde \Theta_{r})>0$ (resp. $<0$), we necessarily have $\alpha\in\mathcal S^{\ge 0}$ (resp. $\mathcal S^{\le'0}$). If $c(\tilde \Theta_r)=0$ then $\Theta_r\in\mathcal S^{=0}$ and since $j\notin T$ we have $m_{\m}(\tilde \Theta_r)=0$ and then $m_{\bar \m}(\tilde \Theta_r)=1$ so we necessarily have $\alpha\in\mathcal S^{=0}$. Now that $\alpha$ has been fixed, we show $\alpha= \bar\Delta_{i+1}$. First note that $e(\alpha)=j-1=e(\bar \Delta_i)-1$. Now since $p(\bar \Delta_i)$ is either $[j,j]$ or $[j-1,j]$ and $e(\alpha)=e(\bar \Delta_i)-1$, $\alpha$ and $\bar \Delta_i$ cannot be both centered (and then there are no signs compatibility to check). Now we prove that $\alpha\prec\bar \Delta_i$. If $c(\alpha)=0$ we have $\alpha\in\mathcal S^{=0}$ and $\bar \Delta_i\in\mathcal S^{\ge 0}$ since it cannot be centered and it ends strictly after $\alpha$ so the inequality is clear. Now if $c(\alpha)\neq 0$. We could have $\bar \Delta_i\in\mathcal S^{=0}\cup \mathcal S^{\ge 0}$ if $c(\alpha)=-1/2$ but then $\alpha\in\mathcal S^{\le'0}$ so we have the inequality. In all the other cases we have either $\alpha$ and $\bar \Delta_i$ in $\mathcal S^{\ge 0}$ or $\mathcal S^{\le'0}$. Then we need to compare them with the order $\le$. We have $\ell(\alpha)=\ell(\tilde \Theta_r)\ge \ell(\Theta_r)+1\ge 2$ and $\ell(\bar \Delta_i)\le 2$. Then $\alpha$ ends in $e(\Delta_{i-1})$ and $\ell(\alpha)\ge \ell(\bar \Delta_i)$ so $\alpha\prec \bar \Delta_i$. Now we need to show that $\alpha$ is the largest segment verifying these points. The segments of $\bar y$ ending in $j-1$ are of different types: we have $\alpha$, we have the segments of $y$ ending in $j-1$ and also, eventually $[j-1,j-1]$ if $j-1\in I^c$. The segments in $y$ ending in $j-1$ have to be smaller than $\alpha$ because otherwise, we would have $j\in T$. So now it suffices to prove that $-j+1\notin I^c$. If $j>-b+1$ then $1-j$ does not even belong in $[b,e]$. Now if $j\le -b+1$, we have $-j+1\in I$ if and only if $j\in I$ if and only if $-j+1\notin T$. Thus, either $-j+1$ is in $T$, or $-j+1\in I$ but we never have $-j+1\in I^c$. This concludes the proof that $\bar \Delta_{i+1}=\alpha$ and then $p(\bar \Delta_{i+1})=\tilde \Theta_r$.

    Now to finish, suppose that $i=e-b+1$. We already constructed $\bar \Delta_i$ that ends in $b$. 
\begin{itemize}[noitemsep,topsep=0pt]
        \item  First assume that $\s=s\in\Seg^{>0}$. If $\Theta^{b-1}(\m,\e)$ is not empty, then $b\in T$  and $p(\bar \Delta_{e-b+1})=\tilde \Theta_1$. Then if there exist $u\in\bar y$ such that $u\preceq \bar \Delta_{e-b+1}$ and $e(u)=b-1$, then it would contradict the minimality of $\Theta_1$ in $y$. Thus we have $\bar l=b-e+1$. Now at this point we have to prove that $\bar \e_0=1$. To do that we need to justify that for any $1\le i\le \bar l$, we have $\bar \Delta_i\notin\{[0,0]^{\ge 0},[0,0]^{=0}\}$ when $\rho$ is of the same type as $G$ and, otherwise, that $\bar \Delta_i\neq [1/2,1/2]$, and if $[-1/2,1/2]$ is in $\bar \m$ with sign -1, then $\bar \Delta_i\notin\{[-1/2,1/2]^{\ge 0},[-1/2,1/2]^{=0}\}$.

        \begin{itemize}
            \item  First assume that $\rho$ is of the same type as $G$. First, since $e(\bar \Delta_i)=j$, if $j\neq 0$ we are done. If $j=0$ we have two possibilities. If $0\in T$ then $\bar \Delta_i$ is of length $\ge 2$ and so it cannot be $[0,0]$. Now assume that $0\notin T$.
            If $m_{\m}([0,0])=0$ we have $m_{\tilde \m}([0,0])=0$,  if $m_{\m}([0,0])$ is even and non-zero we have $[0,0]^{\ge 0}\in\Theta$ so $m_{\tilde \m}([0,0])=m_{\m}([0,0])-2$ even, finally if $m_{\m}([0,0])$ is odd we have $[0,0]^{=0}\in \tilde \m$ so $m_{\tilde \m}([0,0])=m_{\m}([0,0])-1$ even. Finally, $m_{\bar \m}([0,0])=m_{\tilde \m}([0,0])+2\cdot\delta_{0\in I^c}$ so it is still even in any cases.
            At this stage we know that $\bar \Delta_i\neq [0,0]^{=0}$. The only possibility to have $\bar \Delta_i=[0,0]^{\ge0}$ is if $i\ge 1$ and $\bar \Delta_{i-1}=[1,1]$. We will prove that $[1,1]\notin\bar \m$ in this case. To do that we prove that $[1,1]\notin \m$ and $i\in I$. The reason why $[1,1]\notin \m$ is because we assumed $0\notin T$ which is impossible if $[1,1]$ is in $\m$ because we would also have its contragredient $[-1,-1]\in\m$. Finally, $1\in I$ because $-(-1)+1=0\notin T$. We conclude in this case that $\bar \e_0=1$.

            \item  The case where $\rho$ is not of the same type as $G$ is similar. If the ending $j$ of $\bar \Delta_i$ is not $3/2$ the result is true so we assume $j=3/2$.
            
             If $3/2\in T$, let us take $1\le r\le l$ such that $p(\bar \Delta_i)=\tilde \Theta_r$. We have 4 possibilities. We may have $p(\bar \Delta_i)=[a,3/2]$ with $a<-3/2$ when $\Theta_r=[a,1/2]$ or $p(\bar \Delta_i)=[-3/2,3/2]$ when $\Theta_r\in\{[-1/2,1/2]^{\le0},[-1/2,1/2]^{=0}\}$ or $p(\bar \Delta_i)=[-1/2,3/2]$ when $\Theta_r=[-1/2,1/2]^{\ge 0}$ or finally $p(\bar \Delta_i)=[1/2,3/2]$ when $\Theta_r=[1/2,1/2]$. In the first case there is nothing to say. In the second one, if $\bar \Delta_i\in\mathcal S^{\le'0}$ we are done but its not necessarily the case, however, since $p(\Theta_r)=[-1/2,1/2]$, we have  $\bar \e([-1/2,1/2])=\e'_0=1$ so we are done in this case. In the third one, as in the previous case, we get $\bar \e([-1/2,1/2])=\e_0=1$ so we are done. In the fourth case, we either have $\e([-1/2,1/2])=+1$ or $m_{\m}([-1/2,1/2])=0$ in both cases we get $\bar \e([-1/2,1/2])=\eta_0=1$ so we are done. 
             
             If $3/2\notin T$, then $p(\bar \Delta_i)$ may be $[1/2,3/2]$ or $[3/2,3/2]$. First, the fact that $1/2\notin T$ implies that $[1/2,1/2]\notin \m$ and we have $1/2\in I$ so $[1/2,1/2]\notin \bar \m$. Secondly, as before, we either have $\e([-1/2,1/2])=+1$ or $m_{\m}([-1/2,1/2])=0$ and in both cases we get $\bar \e([-1/2,1/2])=+1$. So finally, we are done and $\bar \e_0=1$.
        \end{itemize}

    \item  Now assume that $\s=(s,\eta)\in\Seg^{=0}\times \{1,-1\}$. In this case $b=\tilde 1$. Now by definition of $\bar \m$ we have $\delta^{\ge 0}$ in $\bar \m$.

    \begin{itemize}
        \item  If $\rho$ is of the same type as $G$, we have $\delta^{\ge 0}=\delta=[0,0]$. There are three possibilities for $p(\bar \Delta_i)$: $[1,1]$ or $[0,1]$ or $[-1,1]$. In the first two cases, there is nothing more to say. In the last one we prove that $m_{\bar \m}([0,0])$ is odd and $\bar \e([-1,1])=-\e([0,0])$. Having $p(\bar \Delta_i)=[-1,1]$ implies that $\Theta_1=[0,0]^{=0}$ so $m_{\m}([0,0])$ is odd and $m_{\bar \m}([0,0])=m_{\tilde \m}([0,0])+2=m_{\m}([0,0])+2$ is also odd.
        Since $p(\Theta_1)=[0,0]$ we have $\bar \e([0,0])=\e([0,0])$ and $\bar \e([-1,1])=-\e([0,0])=-\bar \e([0,0])$.  In the end we obtain $\bar \Delta_{i+1}\in\{[0,0]^{=0},[0,0]^{\ge 0}\}$ so $\bar l=e-b+2$ and $\bar \e_0=-1$.
        
        \item If $\rho$ is not of the same type as $G$, we have two possibilities for $\delta^{\ge 0}$, either $\delta^{\ge 0}=\delta=[-1/2,1/2]$, or $\delta^{\ge 0}=[1/2,1/2]$.
        \begin{itemize}
            \item The case $\delta^{\ge 0}=[1/2,1/2]$ happens only if $\eta=(-1)^{n_0}$ and in this case we necessarily have $\bar \Delta_i=[3/2,3/2]$ so we deduce that $\bar l=i+1$ and $\Delta_{\bar l}=[1/2,1/2]$.

            \item In the other case, we have $\delta^{\ge 0}=[-1/2,1/2]$ and $\eta=(-1)^{n_0+1}$. If $p(\Theta)=[-1/2,1/2]$ we get $\bar \e([-1/2,1/2])=\e_0'=-1$, otherwise we necessarily have $\eta_0=1$ and $\bar \e([-1/2,1/2])=\e_0'\cdot \eta_0=-1$. We have four possibilities: indeed $p(\bar \Delta_i)$ may be $[3/2,3/2]$ or $[1/2,3/2]$ or $[-1/2,3/2]$ or finally $[-3/2,3/2]$. In the first case we have $3/2\notin T$ so $[1/2,1/2]\notin \m$ and then $[1/2,1/2]\notin\bar \m$ so we are done. In the second and third cases, we have nothing to say. In the last case, we prove that $m_{\bar \m}([-1/2,1/2])$ is odd and that $\bar \e([-3/2,3/2])=1$. Since $p(\bar \Delta_i)=[-3/2,3/2]$ we have $\Theta_1=[-1/2,1/2]^{=0}$ so we deduce that $m_{\m}([-1/2,1/2])$ is odd and $m_{\tilde \m}([-1/2,1/2])=m_{\m}([-1/2,1/2])-1$ and so $m_{\bar \m}([-1/2,1/2])=m_{\tilde \m}([-1/2,1/2])+1=m_{\m}([-1/2,1/2])$ odd. Finally, $\Theta_1=[-1/2,1/2]$ implies that $\e([1/2,1/2])=-1$ so $\e([-3/2,3/2])=\e_0'\cdot \e([-1/2,1/2])=(-1)\cdot (-1)=1$. At the end, we get that in all four cases, we have $\bar l=i+1=e-b+2$ with $\bar \Delta_{\bar l}\in\{[-1/2,1/2]^{=0},[-1/2,1/2]^{\ge0}\}$ and $\bar \e_0=-1$. 
        \end{itemize}    
    \end{itemize}
    \end{itemize}

\end{proof}

\begin{prop}\label{prop bar e1}
If $\s=(s,\eta)\in\Seg^{=0}\times \{1,-1\}$ we have $\bar \e_1(s)=\eta$ and then in general, we have $(\bar \m_1,\bar \e_1)=\s$.
\end{prop}

\begin{proof}
    To do that we relate $\bar n_0$ to $n_0$. Recall that $n_0$ is the number of centered segments in $\m$ and $\bar n_0$ is the number of centered segments in $\bar \m$. Let $\tilde n_0$ the number of centered segments in $\tilde \m$. By definition of $\tilde \m$ we have $\tilde n_0=n_0-2\cdot\{1\le i\le l~|~\Theta_i\in\mathcal S^{\ge 0}\}$ and $\bar n_0=n_0+1$ if $\delta$ is a centered segment and $\bar n_0=\tilde n_0$ otherwise. So finally $(-1)^{\bar n_0}=-(-1)^{n_0}$ if and only if $\rho$ is of the same type as $G$ or $\rho$ is not of the same type as $G$ and $\eta=(-1)^{n_0}$; and $(-1)^{\bar n_0}=(-1)^{n_0}$ if and only if $\rho$ is not of the same type as $G$ and  $\eta=(-1)^{n_0}$.
    
    If $\rho$ is of the same type as $G$, in both cases we get $\bar \e([0,0])=(-1)^{n_0+1}\cdot \eta\cdot\e_0'=(-1)^{n_0}\cdot \eta$ and then  $\bar \e_1(s)=(-1)^{\bar n_0+1}\cdot\bar \e([0,0])=(-1)^{\bar n_0+1}\cdot(-1)^{n_0}\cdot\eta=\eta$.
    
    If $\rho$ is not of the same type as $G$, we get $(-1)^{\bar n_0}=\eta$ and so $\bar \e_1(s)=\eta$.
\end{proof}

In order to prove Theorem \ref{main theorem} we are now left to prove that $(\bar \m^\#,\bar \e^\#)=(\m,\e)$. We will mainly use the key Proposition \ref{prop key tilde Theta and Delta} but we need it to complete it with a lemma in order to show that $\bar \m^\#=\m$.

\begin{lem}\label{lemma Delta and s of length 2}
Take $1\le i\le \bar l$ and $j:=e-i+1$, then if $j\in I$ and $j\le\tilde 1-1\in\{-1/2,0\}$, we have $p(\bar \Delta_{i})=s_j=[j-1,j]$.
\end{lem}

\begin{proof}
    This lemma relies mainly on the following observation, that is for any $1\le i\le \bar l-1$, if $\bar \Delta_i$ and $\bar \Delta_{i+1}$ are both in $\mathcal S^{\ge 0}$ or $\mathcal S^{\le'0}$, then $\ell(\bar \Delta_i)\le \ell(\bar \Delta_{i+1})$. Indeed, if $e(\bar \Delta_{i+1})=e(\bar \Delta_i)-1$ and $\ell(\bar \Delta_i)>\ell(\bar \Delta_{i+1})$ we would have $\bar \Delta_{i+1}\succ \bar \Delta_i$ which is impossible by definition.

    Assume there exists $1\le i\le \bar l$ such that $j\in I$ and $j\le -\tilde 1-1$ but $p(\bar \Delta_i)=[j,j]$. Define $j_0$ as the largest element of $j+\mathbb N$ such that for any $j'\in[j,j_0]$ we have $j'\in I$, $j'\le-\tilde 1+1$ and $[j',j']\in\tilde \m$. By maximality $j_0+1$ does not verify these conditions. If $j_0+1>-\tilde 1+1$ we have $j_0=-\tilde 1+1$ and we deduce that for any $j'\in [j,-\tilde 1+1]$ we have $[j',j']\in\bar \m$ but also its contragredient and then we must have $j\in T$ which is absurd since we assumed $j\in I$. Now assume that $j_0+1\le -\tilde 1+1$ and so $j_0<-\tilde1+1$. There are two possibilities, either $j_0+1\in T$ or $[j_0+1,j_0+1]\notin \bar \m$. Suppose the second one is true. By our preliminary observation we see that for any $j'\in[j,j_0]$ we have $p(\bar \Delta_{i-j'+j})=[j',j']$ so in particular $p(\bar \Delta_{i-j_0+j})=[j_0,j_0]$. Since $j_0<-\tilde 1+1$ we clearly have $i-j_0+j>1$ and the only possibility for $p(\bar \Delta_{i-j_0+j-1})$ is $[j_0+1, j_0+1]$ which means that this segment is in $\bar \m$. Finally, assume $j_0+1\in T$, then there exists $1\le r\le l$ such that $p(\bar \Delta_{i-j_0+j-1})=\tilde \Theta_r$ of length $\ge 2$ which is impossible because since $\Delta_{i+j_0-j}$ and $\Delta_{i+j_0-j-1}$ are in $\mathcal S^{\le'0}$ we need $1=\ell(\bar \Delta_{i+j_0-j})\ge \ell(\bar \Delta_{i-j_0+j-1})$. All the cases led to an absurdity which proves there are no such $i$.
 \end{proof}

We now have everything to prove Theorem \ref{main theorem} and the proof will just be a verification.

\begin{proof}

As we have already explained it, it is sufficient to prove that $(\bar \m_1,\bar \e_1)=\s$ and $(\bar \m^\#,\bar \e^\#)=(\m,\e)$. The first points have already been proven in Propositions \ref{prop key tilde Theta and Delta} and \ref{prop bar e1}. Here we prove that $(\bar \m^\#,\bar \e^\#)=(\m,\e)$.

So if we write $(y,\e)=s(\m,\e)$ and $y=\Phi_1+\dots+\Phi_m$, we want to prove that $\bar{\m}^\#=p(\Phi_1)+\dots+p(\Phi_m)~(=\m)$. Recall that we have
$$\bar{\m}=\tilde{\Phi}_1+\dots+\tilde{\Phi}_m+\sum_{i\in I}s_i+\sum_{i\in I^c}(s_i+s_i^\vee)+\delta.$$
Among those $\bar m$ segments we have $(p(\bar \Delta_i))_{1\le i\le \bar l}$. If $e-i+1\in T$, we have $p(\bar \Delta_i)\in\{\tilde \Phi_k~|~1\le k\le m\}$, if $e-i+1\in I^c$ we have $p(\bar \Delta_i)=s_i=[i,i]$, if $e-i+1\in I$, we may have  $p(\bar \Delta_i)=[i,i]\in\{\tilde \Phi_k~|~1\le k\le m\}$ if $e-i+1\ge 1/2$ but if it is not the case we have $p(\bar \Delta_i)=s_i=[i-1,i]$. This follows from Proposition \ref{prop key tilde Theta and Delta} and Lemma \ref{lemma Delta and s of length 2}.

Here is a last important detail needed because the definition of $\bar \m$ uses the involution $\iota'$ and the definition of $\m^\#$ uses the other involution $(\cdot)^\vee$. Assume that $p(\bar \Delta_i)=\tilde \Theta_r$. 

We show that $\bar \Delta_i=\bar \Delta_i^\vee$ if and only if $\iota'(\Theta_r)=\Theta_r$. Assume that $\bar \Delta_i=\bar \Delta_i^\vee$, then $\bar \Delta_i$ is centered, so is $\tilde \Theta_r$, and $\bar \Delta_i\in\mathcal S^{=0}\cup \mathcal S^{\ge 0}$. According to Proposition \ref{prop key tilde Theta and Delta}, if $\bar \Delta_i$ is centered, then $\bar \Delta_i\in\mathcal S^{=0}\cup \mathcal S^{\le'0}$. So we deduce that $\bar \Delta_i\in\mathcal S^{=0}$. But, again according to Proposition \ref{prop key tilde Theta and Delta}, $\Theta_r$ is centered and thus it has to be in $\mathcal S^{=0}\cup \mathcal S^{\le'0}$ so that $c(\tilde \Theta_r)=0$ and then we have $\iota'(\Theta_r)=\Theta_r$.  

Now assume $\bar \Delta_i\neq \bar \Delta_i^\vee$. The first possibility is $c(\bar \Delta_i)\neq 0$ but then either $\Theta_r$ is not centered or $\Theta_r$ is centered but in $\mathcal S^{\ge 0}$, in both cases we have $\iota'(\Theta_r)\neq \Theta_r$. The second possibility is $\bar \Delta_i$ is centered and in $\mathcal S^{\le'0}$, but then according to Proposition \ref{prop key tilde Theta and Delta}, we have $c(\Theta_r)=-1/2$ so $\Theta_r$ is not centered and $\iota'(\Theta_r)\neq \Theta_r$.

Because of what we have just said, we can write $\bar y=\bar \Lambda_1+\ldots+\bar \Lambda_{\bar m}$ such that for any $1\le i\le m$ we have $p(\bar \Lambda_i)=\tilde \Phi_i$ and $\{i_1,\ldots,i_l\}=\{i_k,~1\le k\le l~| e-k+1\in T\}$ and $\{i_1',\ldots,i_l'\}=\{i_k',~1\le k\le l~|~ e-k+1\in T\}$. Moreover we can also suppose that $p(\bar \Lambda_{\bar m})=\delta^{\ge 0}$ when $\s=(s,\eta)\in\Seg^{=0}\times \{1,-1\}$.

 Take $1\le i\le \bar m$ if $\s=s\in\Seg^{>0}$ and $1\le i\le \bar m-1$. We now compute $\bar \Lambda_i^\#$.

 Assume $\s=(s,\eta)\in\Seg^{=0}\times \{1,-1\}$. According to Proposition \ref{prop key tilde Theta and Delta} we know that $\bar \Delta_{\bar l}$ may be $[1/2,1/2]$, $[-1/2,1/2]^{=0,\ge 0}$ or $[0,0]^{=0,\ge 0}$ so we deduce in each case that  $\bar \Lambda_{\bar i_{\bar l}}^\#=\bar \Lambda^\#_{\bar i_{\bar l'}}=\varnothing$.

 From now on assume $i\notin\{\bar i_{\bar l},\bar i_{\bar l}'\}$.

 \begin{itemize}[noitemsep,topsep=0pt]
     \item If $i\notin \{\bar i_1,\ldots,\bar i_{\bar l}\}$ and $i\notin \{\bar i_1,\ldots,\bar i_{\bar l}\}$ we have $\bar \Lambda_i^\#=p(\bar \Lambda_i)$. It may only happen if $1\le i\le m$ but then $i\notin\{i_1,\ldots,i_l\}\cup\{i_1',\ldots,i_l'\}$ so $p(\bar \Lambda_i)=\tilde \Phi_i=p(\Phi_i)$.

     \item If $i\in \{\bar i_1,\ldots,\bar i_{\bar l}\}$ and $i\notin \{\bar i_1,\ldots,\bar i_{\bar l}\}$ we have $\bar \Lambda_i^\#=p(\bar \Lambda_i)^-$. There are three cases. The first possibility is $i=i_k\le m$ with $e-k+1\in T$ and so there exists $1\le r\le l$ such that $p(\bar \Lambda_i)=p(\bar \Delta_k)=\tilde \Theta_r$ with $i=i_r\notin\{i_1',\ldots,i'_l\}$ so  $\tilde \Theta_r=p(\Phi_i)^+$. Whence $\bar \Lambda_i^\#=(p(\Phi_i)^+)^-=p(\Phi_i)$. The second possibility is $i=\bar i_k\le m$ with $e-k+1\in I$ such that $p(\bar \Delta_k)\neq s_{e-k+1}$. In this case we get $p(\bar \Lambda_i)=p(\bar \Delta_k)$ is a singleton so $\bar \Lambda_i^\#=\varnothing$. Finally, the last possibility is $i>m$ and such that $p(\bar \Lambda_i)=s_j$ with $j\in I^c$ and again we get $\bar \Lambda_i^\#=\varnothing$.

     \item If $i\notin \{\bar i_1,\ldots,\bar i_{\bar l}\}$ and $i\in \{\bar i_1,\ldots,\bar i_{\bar l}\}$ we have $\bar \Lambda_i^\#={}^-p(\bar \Lambda_i)$. There are 3 cases.  The first possibility is $i=i_k\le m$ with $e-k+1\in T$ and so there exists $1\le r\le l$ such that $p(\bar \Lambda_i)=p(\bar \Delta_k)=\tilde \Theta_r$ with $i=i'_r\notin\{i_1,\ldots,i_l\}$ so  $\tilde \Theta_r={}^+p(\Phi_i)$. Whence $\bar \Lambda_i^\#={}^-(^+p(\Phi_i))=p(\Phi_i)$.
     The second possibility is $i>m$ and $p(\bar \Lambda_i)=s_j^\vee=[-j,-j]$ with $j\in I^c$. Then $\bar \Lambda_i^\#={}^-p(\bar \Lambda_i)=\varnothing$. The third possibility is $i>m$ and $p(\bar \Lambda_i)=s_j$ with $j\in I$ and $p(\bar \Delta_{e-j+1})\neq s_j$. In this case we get $\bar \Lambda_i^\#={}^-s_j=[j,j]$.

     \item If $i\in \{\bar i_1,\ldots,\bar i_{\bar l}\}$ and $i\in \{\bar i_1,\ldots,\bar i_{\bar l}\}$ we have $\bar \Lambda_i^\#={}^-p(\bar \Lambda_i)^-$. There are two cases. The first possibility is $i=i_k\le m$ with $e-k+1\in T$ and so there exists $1\le r\le l$ such that $p(\bar \Lambda_i)=p(\bar \Delta_k)=\tilde \Theta_r$ with $i=i_r\in\{i_1',\ldots,i'_l\}$ so  $\tilde \Theta_r={}^+p(\Phi_i)^+$. Whence $\bar \Lambda_i^\#={}^-(^+p(\Phi_i)^+)^-=p(\Phi_i)$. The only other possibility is $i>m$ such that $p(\bar \Lambda_i)=s_j$ with $j\in I$ and we then get $\bar \Lambda_i^\#={}^-p(\bar \Lambda_i)^-=\varnothing$.
     
 \end{itemize}
 
Let us summarize the situation:

For any $1\le i\le m$,
\begin{itemize}[noitemsep,topsep=0pt]
    \item $\bar \Lambda_i^\#=\varnothing$ if $i=\bar i_k$ with $e-k+1\in I$ such that $p(\bar \Delta_k)\neq s_{e-k+1}$ (in this case $p(\Phi_i)=[e-k+1,e-k+1])$,
    \item $\bar \Lambda_i^\#=p(\Phi_i)$ otherwise
\end{itemize}
and, for any $i>m$,
\begin{itemize}[noitemsep,topsep=0pt]
    \item $\bar \Lambda_i^\#=[j,j]$ if $p(\bar \Lambda_i)=s_j$ with $j\in I$ such that $p(\bar \Delta_{e-j+1})\neq s_{j}$,
    \item $\bar \Lambda_i^\#=\varnothing$ otherwise.
\end{itemize}

Denote by $I_s$ the subset of $I$ that consists in the $j\in I$ such that $p(\bar \Delta_{e-j+1})\neq s_j$, then we have:
This implies that $$\bar \m^\#=\sum_{1\le i\le m}\bar \Lambda_i^\#+\sum_{m+1\le i\le \bar m}\bar \Lambda_i^\#=\left(\sum_{1\le i\le m}p(\Phi_i)-\sum_{j\in I_s}[j,j]\right)+\sum_{j\in I_s}[j,j]=\sum_{1\le i\le m}p(\Phi_i)=\m$$
To finish the proof we now have to check that $\bar \e^\#=\e$. Take $1\le i\le \bar m$ and assume that $\bar \Lambda_i^\#$ is centered and non-empty. We have just proven that $\bar \Lambda_i^\#=p(\Phi_i)$ so we now have to prove that $\bar \e^\#(\bar \Lambda_i^\#)=\e(\Phi_i)$.

We now assume that $\bar \Lambda_i^\#=p(\Phi_i)$ is centered. Recall that we have $\e_0'=\bar \e_0$ according to Proposition \ref{prop key tilde Theta and Delta}. 
\begin{itemize}[noitemsep,topsep=0pt]
  \item If there exists $1 \le j \le \bar l$ such that $\bar \Lambda^{\#}_i=\bar \Lambda^{\#}_{\bar i_j}$ and $c(\bar \Delta_j)=0$; then $\varepsilon^\#(\bar \Lambda^{\#}_i)=\varepsilon_0 \cdot \varepsilon(\bar \Delta_j)$. In this case we necessarily have $\ell(p(\bar \Lambda_i))=\ell(\bar \Lambda_i^\#)+2\ge  3$ so $\bar \e(\tilde \Phi_i)=\tilde \e(\tilde \Phi_i)$ and there exists $1\le r\le l$ such that $i=i_r$ and $p(\bar \Delta_j)=\tilde \Theta_r$. We have $\bar \e^\#(\bar \Lambda_i^\#)=\bar \e_0\cdot\bar\e(\bar \Delta_j)=\bar \e_0\cdot\e_0'\cdot\e(\Theta_r)=\e(\Theta_r)=\e(\Phi_i)$.

  \item If there exists $1 \le j \le l$ such that $\bar \Lambda^{\#}_i=\bar \Lambda^{\#}_{\bar i_j}$ and $c(\bar \Delta_j)=1/2$. Then, we necessarily have $\bar \Lambda^{\#}_i \in \bar \m$. This argumentissimilar to Lemma \ref{lemma sign tilde}, if $c(\bar \Delta_j)=1/2$ we know that $\bar l\ge j+1$ and $\bar \Delta_{j+1}\succeq\bar \Delta_j^\vee$ 
  (since $\bar \Delta_j^{\vee}\prec\bar \Delta_j$ and $e(\bar \Delta_j^\vee)=e(\bar \Delta_j)-1$) but if we have $\bar \Delta_{j+1}=\bar \Delta_j^\vee$ we must have $c(\bar \Lambda_i^\#)=1/2$ so $\bar \Delta_{j+1}$ is centered and $\bar \Lambda_i^\#=p(\bar \Delta_{j+1})\in\bar \m$.  Then by definition $\bar \varepsilon^\#(\bar \Lambda^{\#}_i)=-\bar \varepsilon_0  \cdot \bar \varepsilon(\bar \Lambda^{\#}_i)$. We then have three possibilities:
  \begin{itemize}
      \item If $\bar \Lambda_i^\#=[0,0]$, we have $\Theta_1=[0,0]^{\ge 0}$ so $\bar \e([0,0])=-\e_0'\cdot \e([0,0])$ and thus $\bar \e^\#([0,0])=-\bar \e_0\cdot(-\e_0'\cdot\e([0,0]))=\e([0,0])$.

      \item If $\bar\Lambda_i^\#=[-1/2,1/2]$, we have $\Theta_1=[-1/2,1/2]^{\ge 0}$ so $\e([-1/2,1/2])=-1$ but $\bar \e([-1/2,1/2])=\e_0'$ so $\bar \e^\#([-1/2,1/2])=-\bar \e_0\cdot \e_0'=-1=\e([-1/2,1/2]).$

      \item Otherwise, we have $\bar \e(\bar \Delta_{j+1})=\tilde \e(\bar \Delta_{j+1})$ and there are $1\le r\le l$ such that $p(\bar\Delta_j)=\tilde \Theta_r$ and then $\bar \e(\bar \Delta_{j+1})=\e_0'\cdot \e(\Theta_r)$. We have $p(\bar \Delta_j)=\tilde \Theta_{r+1}$ and $\Phi_i=\Theta_{r+1}$ but by definition $\e(\Theta_{r+1})=-\e(\Theta_r)$ so $\bar \e^\#(\bar \Lambda_i^\#)=-\bar \e_0\cdot \bar \e(\bar \Delta_{j+1})=-\e(\Theta_r)=\e(\Phi_i)$.
  \end{itemize}

  \item Otherwise, $\bar \e^\#(\bar \Lambda^{\#}_i)=\bar \e_0 \cdot \bar \e(\bar \Lambda^{\#}_i)$. Again we consider three possibilities:

  \begin{itemize}
      \item If $\bar \Lambda_i^\#=[0,0]$ then, in this case $\bar \e([0,0])=\e_0'\cdot \e([0,0])$ so $\bar \e^\#(\bar \Lambda_i^\#)=\bar \e_0\cdot\e_0'\cdot\e([0,0])=\e([0,0])$.

      \item If $\bar \Lambda_i^\#=[-1/2,1/2]$, we have $\bar \e([-1/2,1/2])=\e_0'\cdot\e([-1/2,1/2])$ and thus $\bar \e^\#([-1/2,1/2])==\bar \e_0\cdot \e_0'\cdot \e([-1/2,1/2])=\e([-1/2,1/2])$.

      \item Otherwise, $\bar \e(\tilde \Phi_i)=\tilde \e(\tilde \Phi_i)$. In this case $\bar \Lambda_i^\#=\tilde \Phi_i=p(\Phi_i)$  and we have  $\bar \e(\tilde \Phi_i)=\e_0'\cdot\e(\Phi_i)$. So we get $\bar \e^\#(\bar \Lambda_i^\#)=\bar \e_0\cdot \e_0'\cdot \e(\Phi_i)=\e(\Phi_i)$.
  \end{itemize}
  
\end{itemize}

\end{proof}

\section{Some implications on co-tempered data}\label{section implications on anti-temp data}
In this final section we restrict our focus on tempered and co-tempered symmetrical data. Our goal is to use our new algorithm on the tempered symmetrical data in order to get some new results on co-tempered data.
In all of this section, we suppose $\rho$ is a fixed supercuspidal of good parity. We work on the set $\Symm_\rho^{\e}(G)$ and, just as before, we will drop the index $\rho$ from the sets and $\rho_u$ from the segments.

\subsection{Restriction of the algorithm on tempered data}

Restricting our algorithm to tempered data allows some simplifications that we give in this first subsection. Roughly, the simplifications come from the fact that we only have to consider the operators $\CT_{s,\eta}$ restricted to co-tempered data for $(s,\eta)\in\Seg^{=0}\times \{-1,1\}$.

We first introduce some multisegments parameterized by a pair of integers that will be useful in our constructions. We denote by:
$$\m(l,k)=\sum_{l+2\leq i\leq k, ~2|i-k}^{k/2}([(i-1)/2,(i-1)/2]+[-(i-1)/2,-(i-1)/2])$$
for any integers $k\geq  l\geq 0 $ such that $k$ and $l$ have the same parity, with, by convention, $\m(k,k)$ being  $\varnothing$. Notice that the ends and the beginnings of the segments in $\m(l,k)$ are all integers if $l$ and $k$ are odd and all half-integers if $l$ and $k$ are even.

\begin{deff}
    Let $\Temp$ be the set of tempered data, i.e. the set of the elements  $(\m,\e)\in\Symm^\e(G)$ such that $\m$ only contains centered segments. Let also $\ATemp$ be the set of co-tempered data, i.e. the set $\AD(\Temp)$. Finally for any $(s,\eta)\in\Seg^{=0}\times \{-1,1\}$, we define $\Temp^{\le (s,\eta)}=\Temp\cap \Symm^{\le (s,\eta)}$ and $\ATemp^{\le (s,\eta)}=\ATemp\cap \DSymm^{\le (s,\eta)}$.
\end{deff}

In the following we fix $(s,\eta)\in\Seg^{=0}\times \{-1,1\}$ and $(\m,\e)\in\ATemp^{\le (s,\eta)}$ and we explain how to simplify the definition of $\CT_{(s,\eta)}(\m,\e)$.

First we treat the case $\m=\varnothing$. We can verify that we have:
$$
\CT_{s,\eta}(\varnothing,+)=\left\{
  \begin{array}{ll}
   \m(1,k)+([0,0],\eta) & \mbox{if } \rho\mbox{ has the same type as }G \\
   \m(2,k)+([-1/2,1/2],\eta) & \mbox{if } \rho\mbox{ has the opposite type as }G\mbox{ and }\eta=-1 \\
  \m(0,k) & \mbox{if } \rho\mbox{ has the opposite type as }G\mbox{ and }\eta=1
  \end{array}
\right.
$$

This follows from our general definition and we can also, as in the proof, verify it by applying one step of the Lanard-Mínguez algorithm.

For the rest of this Subsection, assume that $\m$ is non empty.

\begin{lem}\label{lemma repet signs}
    If $\rho$ is of the same type as $G$ we have $[0,0]\in\m$. If $\rho$ is not of the same type as $G$ we have $[1/2,1/2]$ or $[-1/2,1/2]$ in $\m$.
    
    In general, we have $\s(\m,\e)=(s',(-1)^{n_0+1}\cdot \alpha_0)$ with $\ell(s')\le \ell(s)$. Moreover, if $\ell(s)=\ell(s')$ implies $\eta=(-1)^{n_0+1}\cdot \alpha_0$.
\end{lem}

\begin{rem}
    
Recall that $\alpha_0$ is $\e([0,0])$ if $[0,0]\in\m$ and 1 otherwise. Then, when $\rho$ is of the same type as $G$, we always have $(-1)^{n_0+1}\cdot \alpha_0=(-1)^{n_0+1}\cdot \e([0,0])$, and when $\rho$ is not of the same type as $G$, we always have $(-1)^{n_0+1}\cdot \alpha_0=(-1)^{n_0}$.
\end{rem}

\begin{proof}
    We apply Lanard-Mínguez algorithm to $(\m,\e)$. Since $(\m,\e)\in \ATemp$ we have  $\s(\m,\e)\in \Seg^{=0}\times \{-1,1\}$ so $\e_0=-1$ and it implies that $p(\Delta_l)=[0,0]$ so $[0,0]\in \m$ when $\rho$ is of the same type as $G$, and $p(\Delta_l)\in\{[1/2,1/2],[-1/2,1/2]\}$ so one of these two segments is in $\m$. Write $\s(\m,\e)=(s',\eta')$, by definition of the Lanard-Mínguez algorithm we have $\eta'=(-1)^{n_0+1}\cdot\e([0,0])=(-1)^{n_0+1}\cdot\alpha_0$ if $\rho$ is of the same type as $G$. Now assume that $\rho$ is not of the same type as $G$, we have $\eta'=(-1)^{n_0}=(-1)^{n_0+1}\cdot\alpha_0$. The rest of the statement follows immediately from the definition of $\DSymm^{\le (s,\eta)}$.
\end{proof}

In the following we denote $\s(\m,\e)=(s',\eta')$.

\begin{prop}\label{prop repet signs}
    We have $\Theta:=\Theta^\s(\m,\e)=\varnothing$ if and only if $\eta=\eta'$.
\end{prop}

\begin{proof}
    This follows from Lemma \ref{lemma repet signs}. Indeed we know that $[0,0]\in\m$ when $\rho$ is of the same type as $G$, and $[-1/2,1/2]$ or $[1/2,1/2]$ is in $\m$ otherwise. This implies that $\Theta=\varnothing$ if and only if $\eta=(-1)^{n_0+1}\cdot \alpha_0=\eta'$.
\end{proof}

Now remark that when $\s=(s,\eta)$ we always have $I=\varnothing$. It implies that the definition of $\bar \m$ is particularly simple when $\eta=\eta'$ since $\Theta=\varnothing$.

\begin{prop}
    If $\eta=\eta'$, we have: 
    $$
\bar{\m}=\left\{
  \begin{array}{ll}
   \m+\m(1,k)+[0,0] & \mbox{if } \rho\mbox{ has the same type as }G \\
  \m+\m(0,k) & \mbox{if } \rho\mbox{ is not of the same type as }G
  \end{array}
\right.
$$
If $\m$ is empty and $\rho$ has the same type as $G$ we have $\bar \e([0,0])=\eta$, else  $\bar \e=-\e$.
\end{prop}

\begin{proof}
    This immediately follows the definition when we have $I=\varnothing$ and $\Theta=\varnothing$.
\end{proof}

When $\eta=\eta'$ there is no simplification of the definition of $\tilde \m$. The definition of $\tilde \e$ can be simplified in two ways: first we can replace $\e_0'$ by -1, secondly we can remove the case $c(\Theta_j)=-1/2$ since the elements of $\Theta$ are all in $\mathcal S^{= 0}\cup\mathcal S^{\ge 0}$ so they have a positive center.

The fact that $I=\varnothing$ can slightly simplify the definition of $\bar \m$.

\begin{prop}

If $\eta\neq \eta'$, denote by $\Theta_m$ the segment of $\Theta$ with the largest end, we have:
    $$
\bar{\m}= \left\{
  \begin{array}{ll}
   \tilde \m+\m(2e(\Theta_m)+3,\ell(s))+[0,0] & \mbox{if } \rho\mbox{ has the same type as }G \\
  \tilde \m+\m(2e(\Theta_m)+3,\ell(s))+[-1/2,1/2] & \mbox{if } \rho\mbox{ is not of the same type as }G
  \end{array}
\right.
$$
\end{prop}

\begin{rem}
    The fact that $2e(\Theta_m)+3\le \ell(s)$ follows from $\s(\m,\e)\le'(s,\eta)$, because since $\eta\neq \eta'$ we have $\ell(s')\le\ell(s)-2$.
\end{rem}

\begin{proof}
    Again, this proof consists in applying the definition of $\bar \m$ with $I=\varnothing$.
\end{proof}

Finally, the definition of $\bar \e$ is the same for the segments $\Delta$ with $\ell(\Delta)\ge 3$, i.e. $\bar \e(\Delta)=\tilde \e(\Delta)$. For the segments $[0,0]$ and $[-1/2,1/2]$ the definition of $\bar \e$ can be simplified:

\begin{prop}
If $\eta=-\eta'$, we have $\bar \e([0,0])=\e([0,0])$ and $\e([-1/2,1/2])=-1$.
\end{prop}

\begin{proof}
    The proof is direct.
\end{proof}

We now provide an example of our algorithm in its simplified version to compute the dual of a tempered datum:
\tikzset{
    segment/.style={line width=1.5pt, black},
    segment blue/.style={line width=1.5pt, blue},
    segment red/.style={line width=1.5pt, red},
    start pt/.style={circle, fill=blue, inner sep=1.5pt},
    end pt/.style={circle, fill=red, inner sep=1.5pt},
    point halo/.style={circle, fill=black, inner sep=2pt},
    point blue/.style={circle, fill=blue, inner sep=2pt},
    point red/.style={circle, fill=red, inner sep=2pt},
    theta link/.style={line width=3pt, green!50, opacity=0.7},
    point theta/.style={circle, fill=green!50, opacity=0.7,inner sep=2pt}
}

\begin{ex}
We reproduce a detailed manual calculation to further illustrate the simplified algorithm on tempered data. Take $\m=[0,0]+2\cdot[-1,1]+[-2,2]$ with $\e([0,0])=1$, $\e([-1,1])=-1$ and $\e([-2,2])=1$. Write $s_1=[0,0]$, $s_2=[-1,1]$, $s_3=[-1,1]$ and $s_4=[-2,2]$. Finally write $\e_i=\e(s_i)$ for any $1\le i\le 4$. We have $s_1\le'\ldots\le's_4$ and so:
$$\AD(\m,\e)=\CT_{s_4}\circ \ldots\circ \CT_{s_1}(\varnothing)$$
We then need to apply 5 steps of our algorithm.
For any $0\le i\le 4$, denote by $(\m^i,\e^i)=\CT_{(s_i,\e_i)}\circ\ldots\CT_{(s_1,\e_1)}(\varnothing)$ with $(\m^0,\e^0)$ being the empty datum. Our goal is to compute $(\m^4,\e^4)=\AD(\m,\e)$.

\begin{enumerate}
    \item  The first step is to compute $(\m^1,\e^1)=\CT_{s_1,\e_1}(\m^0,\e^0)$. Since $\m^0$ is empty we get $(\m^1,\e^1)=\m(1,\ell(s_1))+([0,0],\e_1)=([0,0],1)$.
    \item We now compute $(\m^2,\e^2)=\CT_{s_2,\e_2}(\m^1,\e^1)$. First note that $\e_2\neq \e_2$. We have $s(\m^2)=[0,0]^{=0}$ so $m=1$ and $\Theta_1=[0,0]^{=0}$. So finally we deduce that $\m^2={}^+[0,0]^++[0,0]+\m(3,3)=[1,1]+[0,0]$ and : $\e^2([0,0])=\e^1([0,0])=1$, $\e^2([-1,1])=-\e^1([0,0])=-1$.
    \item We now compute $(\m^3,\e^3)=\CT_{s_3,\e_3}(\m^2,\e^2)$. Since $\e_3=\e_2$ this step is easy. We have $\m^3=\m^2+\m(1,3)+[0,0]=\m^2+[0,0]+[-1,-1]+[1,1]$ and $\e^3=-\e^2$.
    \item For the last step we compute $\AD(\m,\e)=(\m^4,\e^4)=\CT_{s_4,\e_4}(\m^3,\e^3)$. Here we have $\e_4\neq \e_3$. In the increasing way for $\prec$ we have $s(\m^3)=[-1,-1]^{\le 0}+[0,0]^{\le 0}+[-1,1]^{=0}+[0,0]^{\ge 0}+[1,1]^{\ge 0}$. Moreover $\e^3([0,0])=-1$ and $\e^3([-1,1])=1$. We get $m=2$ and: $\Theta_1=[0,0]^{\ge 0}$, $\Theta_2=[1,1]^{\ge 0}$. So we get $\m^4={}^-[-1,-1]+{}^^-[0,0]+[-1,1]+[0,0]^++[1,1]^++[0,0]+\m(5,5)=[-2,-1]+[-1,0]+[0,0]+[-1,1]+[0,1]+[1,2]$. And finally, $\e^4([0,0])=\e^3([0,0])=-1$ and $\e^4([-1,1])=-\e^3([-1,1])=-1$.
    
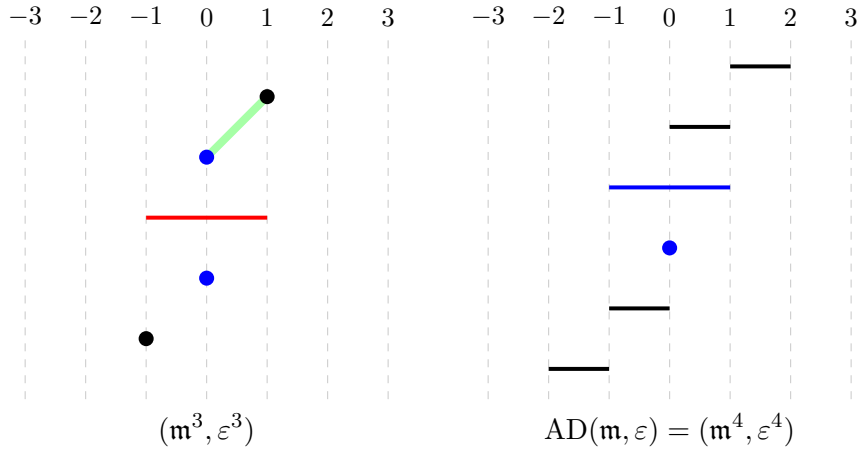
\begin{center}
    \begin{tikzpicture}[scale=0.8]
    \foreach \x in {-3,...,3} {
        \draw[dashed, gray!40] (\x, -3.0) -- (\x, 3.0);
        \node[above] at (\x, 3.0) {\small $\x$};
    }
    
    \draw[theta link] (0, 1.0) -- (1, 2.0);

    \node[point halo] at (-1, -2.0) {};
    \node[point blue] at (0, -1.0) {};

    \draw[segment red] (-1, 0) -- (1, 0);
    \fill[start pt] (-1, 1.0) circle; \fill[end pt] (1, 1.0) circle;
    
    \node[point blue] at (0, 1.0) {};
    
    \node[point halo] at (1, 2.0) {};

    \node at (0, -3.5) {$(\m^3,\e^3)$};
 
\end{tikzpicture}
\hspace{0.5cm}
\begin{tikzpicture}[scale=0.8]
    \foreach \x in {-3,...,3} {
        \draw[dashed, gray!40] (\x, -3.0) -- (\x, 3.0);
        \node[above] at (\x, 3.0) {\small $\x$};
    }
    
    \draw[segment] (-2, -2.5) -- (-1, -2.5);
    \draw[segment] (-1, -1.5) -- (0, -1.5);
    
    \node[point blue] at (0, -0.5) {};
    \draw[segment blue] (-1, 0.5) -- (1, 0.5);
   
    \draw[segment] (1, 2.5) -- (2, 2.5);
    \draw[segment] (0, 1.5) -- (1, 1.5);

    \node at (0, -3.5) {$\AD(\m,\e)=(\m^4,\e^4)$};
 
\end{tikzpicture}
\end{center}

\end{enumerate}
\end{ex}

\subsection{Constructive description of the co-tempered data}

Our algorithm naturally induces a new description of co-tempered data:

\begin{theorem}\label{coro pairing temp}
For any co-tempered symmetrical Langlands datum $(\m,\e)\in\ATemp$, there exists a unique tuple $(s_i,\eta_i)_{1\leq i \leq r}$ with $r\geq 1$, $s_1\subseteq \dots\subseteq s_r$ some centered segments, and $\eta_i\in\{\pm1\}$ for all $1\leq i\leq r$ with $\eta_{i+1}=\eta_i$ if $s_{i+1}=s_i$ for all $1\leq i\leq r-1$, such that:
$$(\m,\e)=\CT_{s_r,\eta_r}\circ\dots\circ \CT_{s_1,\eta_1}(\varnothing)$$
Moreover, in this case, we have
$\AD(\m,\e)=\sum_{i=1}^r(s_i,\eta_i)$.
\end{theorem}

\begin{proof}
    The dual $\AD(\m,\e)$ is a tempered datum so we can write $$\AD(\m,\e)=\sum_{i=1}^r(s_i,\eta_i)$$ We can assume that $(s_1,\eta_1)\le'\ldots\le'(s_r,\eta_r)$ and then, write $$\AD(\m,\e)=\T_{(s_r,\eta_r)}\circ\ldots\circ\T_{(s_1,\eta_1)}(\varnothing)$$ so by applying Theorem \ref{main theorem} exactly $r$ times, we get:
    $$(\m,\e)=\CT_{s_r,\eta_r}\circ\dots\circ \CT_{s_1,\eta_1}(\varnothing)$$
\end{proof}

\begin{rem}
  Note that by definition of $\CT_{s,\eta}$, the number of segments in $\CT_{s,\eta}(\m,\e)$ is strictly bigger than the number of segments in $\m$; we deduce that a tempered datum always have less segments than its co-tempered dual. 
\end{rem}

In the remainder of this section, we will use this new description of co-tempered data to get some new results about co-tempered data. These results include: reduction properties that facilitate the computation of the dual of a tempered datum (\ref{thm red}, \ref{prop red first segment}), and a direct description of the tempered part of co-tempered data (\ref{thm im phi=T}).

Our initial objective was to obtain a simple, direct characterization of co-tempered representations, or a closed formula for the dual of a tempered datum, which is a similar problem. One might hope to deduce such a result from our new constructive characterization. However, extracting a closed-form formula requires a precise understanding of how the initial sequence $\Theta$ evolves after each successive application of our operators. This combinatorial tracking turns out to be a non-trivial problem. 

To address this challenge, we will introduce the class $\Temp^d$ of \textit{decreasing} tempered data. Although we do not currently have a representation-theoretic interpretation for this class, it provides a framework where this combinatorial tracking can be fully resolved (see Proposition \ref{prop initial sequence}). Consequently, we obtain an explicit, direct formula for the Aubert dual of any element in $\Temp^d$ (see Theorem \ref{thm explicit tempd}).

\subsection{A useful reduction}\label{subsec red}

We start with a way to reduce the problem of computing the dual to a smaller class of tempered symmetrical data.

\begin{deff}
Let $(\m,\e)\in\Temp$; write $\m=s_1+\ldots+s_n$ with $s_1\subseteq \ldots\subseteq s_n$. There exist $1\le t \le n$ and a partition $\sqcup_{1\le j\le t} I_j=\llbracket 1,n\rrbracket$ such that for any $1\le j\le t$ and any $i,i'\in I_j$, we have $\e(s_i)=\e(s_{i'})$ and such that $I_j$ is an interval for any $1\le j\le t$. Now define $$\m^{red}=\sum_{j,~|I_j|~odd}s_{\max(I_j)}+\sum_{j,~|I_j|~even}(s_{\min(I_j)}+s_{\max(I_j)})$$

We say that $\m$ is reduced if $\m=\m^{red}$ and we write $\Temp_{red}$ for the set of reduced tempered symmetrical Langlands data. Finally, write $\ATemp_{red}\subset \ATemp$ for the dual of $\Temp_{red}$.
\end{deff}

\begin{lem}\label{lem differ on points}

Take $(\m,\e)$ and $(\m',\e')$ in $\ATemp$. Assume that there exists a set $X$ such that:
  \begin{itemize}
    \item $(\m,\e)=(\m',\e')+\sum_{x\in X}[x,x]$
    \item For any $x\in X$, we have $m_{\m'}([x,x])\ge 1$.
  \end{itemize}
Then $\Theta(\m,\e)=\Theta(\m',\e')$ and then, for any $k\ge 0$ and $\eta\in\{\pm1\}$ such that $(\m,\e)$ and $(\m',\e')$ are in $\ATemp_{k,\eta}$:
\begin{itemize}
    \item $\CT_{s,\eta}(\m,\e)=\CT_{s,\eta}(\m',\e')+\sum_{x\in X}[x,x]$
    \item For any $x\in X$, we have $m_{\CT_{s,\eta}(\m')}([x,x])\ge 1$.
\end{itemize}
with $s$ the centered segment of length $k$.
\end{lem}

\begin{proof}
  Write $\Theta=\Theta(\m,\e)$ and $\Theta'=\Theta(\m',\e')$. Suppose that $\Theta$ contains a point; according to the last point of Lemma \ref{lemme initial sequence}, there exist two integers $i\le n$ such that $\Theta_j$ is a point for any $i\le j\le n$. Since $(\m,\e)$ and $(\m',\e')$ only differ on points, it is clear that $\Theta'_j=\Theta_j$ for any $1\le j\le i-1$. Consequently, $s=\Theta'_i$ cannot be of length $\ge 2$, because then this segment would also be in $\m$ and it would be smaller than $\Theta_i$, so we would have $\Theta_i=s$ of length 2. Then, we deduce that $\Theta'_i$ is also a point, and thus this has to be $\Theta_i$. With the same idea, we deduce that $\Theta_j=\Theta_j'$ for any $i\le j\le n$, and so $\Theta'=\Theta$ (and it would be the same if we assumed that $\Theta'$ contained a point). Since $\Theta=\Theta'$, we deduce the 2 points we claimed by definition of $\CT_{s,\eta}$.
\end{proof}

We now state the reduction theorem:

\begin{theorem}\label{thm red}
  Take $(\m,\e)\in \Temp$. We have:
  \begin{enumerate}
    \item $\AD(\m,\e)=\AD(\m^{red},\e^{red})+\sum_{s\in (\m-\m^{red})}\sum_{x\in s}[x,x]$
    \item $\Theta(\m,\e)=\Theta(\m^{red},\e^{red})$
  \end{enumerate}
\end{theorem}

\begin{proof}

  We prove that whenever we have a succession of three segments of the same sign, the result is true with the symmetrical datum obtained by removing both largest ones instead of the reduced symmetrical datum. Since the reduced symmetrical datum can be obtained by doing this operation multiple times, we get our result.

  Take $k\ge 0$, $\eta\in\{\pm1\}$, and $(\m,\e)\in \ATemp_{k,\eta}$. Now, let $m\ge 3$ be an integer and let $s_1\subseteq\ldots\subseteq s_m$ be segments with $\ell(s_1)\ge k$. Also, take some signs $\eta_i\in\{\pm1\}$ for any $1\le i\le m$ with $\eta_i=\eta$ for any $1\le i\le 3$. Define $(\m^3, \e^3)=\CT_{s_m,\eta_m}\circ \ldots\circ \CT_{s_{3},\eta_3}(\m,\e)$ and $(\m^1,\e^1)=\CT_{s_m,\eta_m}\circ\ldots\circ \CT_{s_1,\eta_1}(\m,\e)$.

  Observe that since $\eta_1=\eta_2=\eta_3=\eta$, we have $$\CT_{s_3,\eta_3}\circ \CT_{s_2,\eta_2}\circ \CT_{s_1,\eta_1}(\m,\e)=\CT_{s_1,\eta_1}(\m,\e)+\sum_{x\in s_2}[x,x]+\sum_{x\in s_3}[x,x]$$
 
  But, regardless of the value of $\eta(\m,\e)$ is, we have
  $$\CT_{s_3,\eta_3}(\m,\e)=\CT_{s_1,\eta_1}(\m,\e)+\sum_{x\in s_3 \backslash s_1}[x,x]$$
  Then, by combining the last two equalities, we get
  $$\CT_{s_3,\eta_3}\circ \CT_{s_2,\eta_2}\circ \CT_{s_1,\eta_1}(\m,\e)=\CT_{s_3,\eta_3}(\m,\e)+\sum_{x\in s_2}[x,x]+\sum_{x\in s_1}[x,x]$$
and since $s_2$ and $s_1$ are both included in $s_3$, we can apply Lemma \ref{lem differ on points} exactly $m-3$ times to deduce that:
\begin{itemize}
  \item $(\m^3,\e^3)=(\m^1,\e^1)+\sum_{x\in s_2}[x,x]=\sum_{x\in s_1}[x,x]$
  \item $\Theta(\m^3,\e^3)=\Theta(\m^1,\e^1)$
\end{itemize}
This ends the proof.
  
\end{proof}

\begin{rem}
If we combine this reduction with our algorithm, we can compute the dual of a tempered datum $(\m,\e)$ in $m^{red}$ steps, where $m^{red}$ is the number of segments in $\m^{red}$.
\end{rem}

Here is another proposition that allows reduction:

\begin{prop}\label{prop red first segment}
  For any $(\m,\e)\in\Temp$ and $s$ its largest segment, let $\m^+$ be the multisegment obtained from $\m$ by replacing one copy of $s$ with ${}^+s^+$. Then:
  $$\AD(\m^+,\e)=\AD(\m,\e)+[b(s)-1,b(s)-1]+[e(s)+1,e(s)+1]$$
\end{prop}

\begin{proof}
  We apply the Lanard-Mínguez algorithm and we find $\m^+_1=[b(s)-1,b(s)-1]+[e(s)+1,e(s)+1]$ and $((\m^+)^\#,\e^\#)=(\m,\e)$, which proves the result.
\end{proof}

\subsection{Some additional definitions}

Now, define a map $(.)_{temp}$ that extracts the tempered part of a symmetrical Langlands datum. Formally, for any symmetrical Langlands datum $(\m,\e)$, define $(\m,\e)_{temp}=(\m_{temp},\e_{temp})$ with $$(\m_{temp},\e_{temp})=(\sum_{s\in \m,~c(s)=0}s,\e_{|\{s\in\m|~c(s)=0\}})$$
We also write $(\m,\e)_{non-temp}=(\m_{non-temp},\e_{non-temp})$ with $\m_{non-temp}$ such that $\m=\m_{temp}+\m_{non-temp}$. Then $\m_{non-temp}$ contains only non-centered segments, and then $\e_{non-temp}$ is the empty function. In this case, we only write $\m_{non-temp}$ and we call it the non-tempered part of $(\m,\e)$.\\

We can easily see these maps on the regular Langlands data. Take $\pi=L(\mathfrak n;\phi,\eta)\in\Irr^G$ and write $(\m,\e)=\phi^{sym}_\pi$. Then $(\m,\e)_{temp}=\trans(\varnothing;\phi,\eta)$ and $(\m,\e)_{non-temp}=\trans(\mathfrak n;1,1)$.\\

We write $\varphi=(\AD(\cdot))_{temp}:\Temp\to\Temp$. We will compute the image of this map.\\

For any integer $j\ge 0$, define

$$\sigma_j = \left\{
  \begin{array}{ll}
    [-j+1,j-1] & \mbox{if } \rho \mbox{ is of the same type as }G \\
   
    [-j+1/2,j-1/2] & \mbox{if } \rho \mbox{ is not of the same type as } G
  \end{array}
\right.$$

Take $(\m,\e)\in \Temp$ and define $f(\m,\e)=((m_{\m}(\sigma_j),\e(\sigma_j)))_{j\ge 0}$. This formula defines a bijection $f:\Temp \to (\mathbb N\times\{\pm1\})^{(\mathbb N)}$. We define $(\mathbb N\times\{\pm1\})^{(\mathbb N)}$ as the subset of $(\mathbb N\times\{\pm1\})^{\mathbb N}$ whose elements are sequences such that every term, except a finite number of them, is $(0,1)$. We denote by 0 the element $((0,1),\ldots)$ and we identify $(\mathbb N\times\{\pm1\})^{(\mathbb N)}$ with $\cup_{n\ge 0}(\mathbb N\times\{\pm1\})^{n}$. \\

Finally, write $\m=s_1+\ldots+s_n$ with $s_1\subseteq\ldots\subseteq s_n$; then define $\mathcal E(\m,\e)=(\e(s_1),\cdots,\e(s_n))$. It defines a map $\mathcal E:\Temp \to\{\pm1\}^{(\mathbb N)}$. Write $\mathcal S=\{\pm1\}^{(\mathbb N)}$ that we also identify with $\cup_{n\ge 0}\{\pm1\}^n$. Be careful, $\mathcal E$ is not $pr_2\circ f$ since the multiplicity is taken into account for $\mathcal E$.

\subsection{Induction rules}
Let $k\ge 0$ be an integer, $s$ the centered segment of length $k$, $\eta$ a sign, and $(\m,\e)\in\ATemp_{k,\eta}$; we would like to express $f((\CT_{s,\eta}(\m,\e))_{temp})$ in terms of $f((\m,\e)_{temp})$.

\begin{deff}
  Let $n\ge0$ be an integer and $(m,e)=((m_i,e_i))_{1\le i\le n}\in(\mathbb N\times \{\pm1\})^n$.
  First define:
 
  \begin{itemize}
    \item $\delta_o(m,e)=((m_1,-e_1),(m_2,-e_1),\cdots,(m_n,-e_n))$
    \item $\delta_s(m,e)=((m_1+1,-e_1),(m_2,-e_1),\cdots,(m_n,-e_n))$
  \end{itemize}
  Now let
 
  \begin{itemize}
    \item $k_o(m,e) = \max\{j\in \mathbb N, m_i \mbox{ is odd and } e_i=(-1)^i~\mbox{ for any }i\le j\}$
    \item $k_s(m,e) = \max\{j\in \mathbb N, m_i \mbox{ is odd and } e_i=(-1)^{i+1}e_i~\mbox{ for any }i\le j\}$
  \end{itemize}
  In the following, we just write $k_o$ and $k_s$. We have $k_o\le n$; if $k_o=n$, write $(m_{k_o},e_{k_o})=(0,1)$. We define $\gamma_o(m,e)=(\bar m,\bar e)$ with $(\bar m_i,\bar e_i)=(m_i,e_i)$ for any $1\le i\le k_o$, and, if $n\ge k_o+2$, $(\bar m_i,\bar e_i)=(m_i,-e_i)$ for any $k_o+2\le i\le n$, and finally:
\[(\bar m_{k_o+1},\bar e_{k_o+1}) = \left\{
  \begin{array}{ll}
    (m_{k_o+1}-1,e_{k_o+1}) & \mbox{if } e_{k_o+1}=(-1)^{k_o+1} \mbox{ and } m_{k_o+1}\neq 0 \\
    (m_{k_o+1}+1,-e_{k_o}) & \mbox{if } e_{k_o+1}=(-1)^{k_o} \mbox{ or } m_{k_o+1}=0
  \end{array}
\right.\]
We also define $\gamma_s$ in the same way with $k_s$ instead of $k_o$.
\end{deff}

\begin{prop}\label{prop induction temp part} If $\rho$ and $G$ have a different type, denote $x=o$; otherwise, denote $x=s$. We have the following formula:

\[f((\CT_{s,\eta}(\m,\e))_{temp}) = \left\{
  \begin{array}{ll}
    \delta_x(f((\m,\e))_{temp}) & \mbox{if } \eta=\eta(\m,\e) \\
    \gamma_x((f(\m,\e))_{temp}) & \mbox{if } \eta=-\eta(\m,\e)
  \end{array}
\right.\]
\end{prop}

\begin{proof}
  We only have to use cautiously the definition of $\CT_{s,\eta}$; there is no difficulty here.
\end{proof}

\begin{rem}\label{rem phi T = gamma phi}
  Now, take $(\m,\e)\in\Temp_{k,\eta}$ and let $\eta'=\e(s_{max})$ with $s_{max}$ the largest segment in $\m$. According to Proposition \ref{prop induction temp part}, and using of course Theorem \ref{main theorem}, we have:
  $$
\varphi(\T_{s,\eta}(\m,\e)) = \left\{
  \begin{array}{ll}
    \delta_x(\varphi(\m,\e)) & \mbox{if } \eta=\eta' \\
    \gamma_x(\varphi(\m,\e)) & \mbox{if } \eta=-\eta'
  \end{array}
\right.
$$

\end{rem}

\subsection{Image of $\varphi$}
According to Remark \ref{rem phi T = gamma phi}, it is clear by induction that $\varphi$ factorizes through $\mathcal E$ since $\varphi(\T_{s,\eta}(\m,\e))$ depends only on $\eta'$ (and not on $s$). Denote by $\bar \varphi$ the map such that $\varphi=\bar \varphi\circ\mathcal E$.

We start by defining
$$\mathcal T=\{(\m,\e)\in\Temp~|~m_{\m}({}^-s^-)\neq 0 \mbox{ for any }s\in \m\}$$
Recall that by convention, $m_\m(\varnothing)=1$ for any $\m\in\Mult$.

\begin{prop}
The image of $\varphi$ is included in $\mathcal T$.
\end{prop}

\begin{proof}
  According to Proposition \ref{prop induction temp part}, we know that the image of $f\circ \varphi$ is the subset of $(\mathbb N\times\{\pm1\})^{(\mathbb N)}$ generated by $0$ under $\delta$ and $\gamma$. Denote by $X$ this set. One can verify that this is exactly the set of the elements $((m_i,e_i))_{1\le i\le n}\in(\mathbb N\times \{\pm1\})^{(\mathbb N)}$ such that if $m_i=0$ for some $i\ge 0$, then $m_j=0$ for any $j\ge i$. Denote by $Y$ this set. First, observe that we clearly have $f(\mathcal T)=X$; then, it suffices to prove that $X\subset Y$. First, observe that $0\in Y$ and that $\delta$ clearly preserves $Y$, since $\bar m_i=m_i$ for $i\ge 2$ and $m_1\neq 0\implies\bar m_1\neq 0$. Now let us show that $\gamma$ also preserves $Y$. Take $(m,e)=((m_i,e_i))_{i\ge 0}\in Y$ and write $\gamma(m,e)=((\bar m_i,\bar e_i))_{i\ge 0}$. Take $1\le i\le k$; we have $\bar m_k=m_k\neq 0$ by definition of $k$. Now, take $i\ge k+2$ such that $\bar m_{i}=0$. Since $\bar m_j=m_j$ for any $j\ge k+2$ (and so for any $j\ge i$) and $(m,e)\in Y$, we deduce that $\bar m_j=0$ for any $j\ge i$. Finally, we see that it is impossible to have $\bar m_{k+1}=0$, so we conclude that $Y\subset X$.
 
\end{proof}

\subsection{Computing some iterations of $\gamma$ and $\delta$ and a section of $\bar \varphi$}
In this subsection, we do an important computation with the operators $\delta_o,\delta_s,\gamma_o$ and $\gamma_s$. This computation allows us to compute a section of $\bar \varphi$ and will allow us to compute the tempered part of the dual of some tempered data.

\begin{prop}\label{prop calcul delta gamma}
Take an integer $r\ge 1$, $n_1\ge 0$ another integer, and, if $r\ge 2$, $n_2\ge \ldots\ge n_r\ge n_{r+1}\ge 0$ some integers with $2n_1\ge n_2$. Finally, take $x\in\{o,s\}$. Now, let $$(m,e)=\gamma_x^{n_{r+1}}\circ\delta_x \circ \gamma_x^{n_r}\circ \delta_x\circ \ldots\circ \delta_x \circ \gamma_x^{n_1}(0)$$
In the following, we will write $\Gamma_x(n_1,\ldots,n_{r+1})=\gamma_x^{n_{r+1}}\circ\delta_x \circ \gamma_x^{n_r}\circ \delta_x\circ \ldots\circ \delta_x \circ \gamma_x^{n_1}$.

If $x=o$, we have $(m,e)=(m_i,e_i)_{1\le i\le n_1}$, with $m_i=1+|M_i^o|$ and $e_i = (-1)^{i+|E^o_i|}$, where
\begin{itemize}
  \item $M^o_i=\{2\le j\le r+1~|~n_j=2i-1\}$
  \item $E^o_i=\{2\le j\le r+1~|~n_j\mbox{ is even and } n_j/2 <i\}$
\end{itemize}
for any $1\le i\le n_1$. Notice that $\delta_o(0)=0$, so $$\Gamma_o(n_1,\ldots,n_{r+1})(\delta_o(0))=\Gamma_o(n_1,\ldots,n_{r+1})(0)=(m,e)$$

If $x=s$, assume moreover that $2n_1>n_2$; we have $(m,e) = (m_i, e_i)_{i \ge 1}$, with $m_i = 1 + |M^s_i|$ and $e_i = (-1)^{i + r - 1 + |E^s_i|}$ where:
\begin{itemize}
  \item $M^s_i = \{ 2 \le j \le r+1 \mid n_j=2i-1 \}$
  \item $E^s_i = \{ 2 \le j \le r+1 \mid n_j \text{ is odd and } (n_j+1)/2 < i \}$
\end{itemize}
for any $1\le i \le n_1$. Now, remark that $\gamma_s^2(0)=((1,+1),(1,-1))$ and $\gamma_s\circ\delta_s=((1,-1),(1,+1))$ are almost the same but with opposite signs. Then, we deduce that $$\Gamma_s(n_1,\ldots,n_{r+1})(\delta_s(0))=(m,-e)$$

\end{prop}

\begin{proof}

Start with the case $x=o$. We prove this result by induction; assume the result is true for $(n_1,\ldots,n_{r+1})$ with $2n_1\ge n_2$ and $n_2\ge\ldots\ge n_{r+1}\ge 0$. Take $(\bar n_1,\ldots, n_{\bar r+1})$ verifying the same property. For any $1\le i\le \bar n_1$, we denote $\bar M_i^o$ and $\bar E_i^o$ for the sets we build from $(\bar n_i)_{1\le i\le \bar r+1}$ as we defined $(M_i^o,E_i^o)_{1\le i\le n_1}$ from $(n_i)_{1\le i\le r+1}$.

Take $(\bar n_1,\ldots,\bar n_{\bar r+1})=(n_1,\ldots,n_{r+1},0)$. We have $\bar n_1=n_1$. Now, remark that for any $1\le i\le n_1$, we have $\bar E_i^o=E_i^o\cup\{r+1\}$ and $\bar M_i^o=M_i^o$. Since, by definition, we have $\delta(m,e)=((m_i,-e_i))_{1\le i\le n_1}$, we deduce that $\delta_o(m,e)=((1+|\bar M_i^o|),(-1)^{i+|\bar E_i^o|})_{1\le i\le \bar n_1}$.

Now, assume that $r\ge 1$ and that $n_r>n_{r+1}$. Take $(\bar n_1,\ldots,\bar n_{\bar r+1})=(n_1,\ldots,n_{r+1}+1)$ and write $d=n_{r+1}$. Finally, write $(\bar m, \bar e)=\delta_o(m,e)$. First, remark that $M_i^o=E_i^o=\varnothing$ for any $1\le i\le \lfloor d/2\rfloor$.

If $d$ is odd, we have $(d+1)/2\le n_1$, and since $n_r>n_{r+1}$, we have $m_{(d+1)/2}=2$, so $k_o=(d-1)/2$, and since $e_{(d+1)/2}=(-1)^{(d+1)/2}$, we have $(\bar m_{(d+1)/2},\bar e_{(d+1)/2})=(m_{(d+1)/2}-1,e_{(d+1)/2})=(1+|\bar M^o_{(d+1)/2}|,(-1)^{(d+1)/2+|E^o_{(d+1)/2}|})$ because $\bar M^o_{(d+1)/2}= M^o_{(d+1)/2}\backslash\{r+1\}$ and $\bar E^o_{(d+1)/2}= E^o_{(d+1)/2}$.

If $d$ is even, if $d=2n_1$, write $(m_{d/2+1},e_{d/2+1})=(0,1)$. We have either $e_{d/2+1}=(-1)^{d/2}$ if $d<2n_1$, or $m_{d/2+1}=0$ otherwise. In both cases, we get $k_o=d/2$ and $(\bar m_{d/2+1},\bar e_{d/2+1})=(m_{d/2+1}+1,-e_{d/2})=(1+|\bar M^o_{d/2+1}|,(-1)^{d/2+1+|E^o_{d/2+1}|})$ because $\bar M^o_{d/2+1}=M^o_{d/2+1}\cup \{r+1\}$ and $\bar E^o_{d/2+1}=E^o_{d/2+1}\backslash \{r+1\}$.

So, in general, $k_o(m,e)=\lfloor d/2\rfloor$. For any $1\le i\le \lfloor d/2\rfloor$, we have $\bar M_i^o=M_i^o$ and $\bar E_i^o=E_i^o$. Now, if $\lfloor d/2\rfloor+2\le n_1$, take $\lfloor d/2\rfloor+2\le i\le n_1$. We have $\bar M_i^o=M_i^o$ and $\bar E_i^o=E_i^o\cup\{r+1\}$ if $d$ is odd, and $ E_i^o=\bar E_i^o\cup\{r+1\}$ if $d$ is even. So, finally, $(\bar m,\bar e)$ coincides with $((1+|M_i^o|),(-1)^{i+|E_i^o|})_{1\le i\le n_1}$.

To finish the proof, we need to show that the formula is true for $r=0$. We compute $\gamma_o^{n_1}=((1,-1), (1,1),\ldots,(1, (-1)^{n_1}))$ and this matches with the formula. This ends the proof.\\

The proof for $x=s$ is analogous to the case $x=o$.

\end{proof}

Take $((m_1,e_1),\ldots,(m_n,e_n))\in f(\mathcal T)$. For any $2\le i\le n$, write $\eta_i=1$ if $e_{i-1}=e_i$ and $\eta_i=0$ otherwise, and write $\eta_n=0$. Now, define for any $1\le i\le n$:
$$\beta^o_i=(\gamma_o^{2i-2}\circ \delta_o)^{\eta_i}\circ(\gamma_o ^{2i-1}\circ\delta_o)^{m_i-1}$$
and, for any $2\le i\le n$:
$$\beta^s_i=(\gamma_s^{2i-3}\circ \delta_s)^{\eta_i}\circ(\gamma_s ^{2i-2}\circ\delta_s)^{m_i-1}$$
Finally, write $\eta^o_1=\delta_{e_1=1}$ and $\eta_1^s=0$ if $e_1=(-1)^{m_1+\sum_{i=2}^n(m_i-1+\eta_i)}$ and 0 otherwise.

\begin{prop}\label{prop section}
We have:
$$((m_1,e_1),\ldots,(m_n,e_n))=\beta_1^o\circ\beta^o_2\circ\ldots\circ\beta^o_n\circ\gamma_o^n(0)=\delta_s^{m_1-1}\circ \beta_2^s\circ\ldots\circ\beta_n^s\circ\gamma_s^n\circ\delta_s^{\eta_1^s}(0)$$
\end{prop}

\begin{proof}
  For the first equality, we just need to apply the formula of Proposition \ref{prop calcul delta gamma}. The second one is a little bit trickier. We first apply the formula of Proposition \ref{prop calcul delta gamma} to compute $$\beta_2^s\circ\ldots\circ\beta_n^s\circ\gamma_s^n(0)=((1,e'_1),(m_2,e'_2),\ldots,(m_n,e'_n))$$ with $(e'_i)_{1\le i\le n}$ such that $e'_1=(-1)^{\sum_{i=2}^n(m_i-1+\eta_i)-1}$ and $e'_i=e'_{i+1}$ if and only if $e_i=e_{i+1}$ for any $1\le i\le n-1$. Then we see that:
  $$\delta_s^{m_1-1}\circ\beta_2^s\circ\ldots\circ\beta_n^s\circ\gamma_s^n(0)=((m_1,(-1)^{m_1-1}e'_1),(m_2,(-1)^{m_1-1}e'_2),\ldots,(m_n,(-1)^{m_1-1}e'_n))$$
  If $e_1=(-1)^{m_1-1}e_1'=(-1)^{m_1+\sum_{i=2}^n(m_i-1+\eta_i)}$, then:
  $$\delta_s^{m_1-1}\circ\beta_2^s\circ\ldots\circ\beta_n^s\circ\gamma_s^n(0)=((m_1,e_1),(m_2,e_2),\ldots,(m_n,e_n))$$
  otherwise,
  $$\delta_s^{m_1-1}\circ\beta_2^s\circ\ldots\circ\beta_n^s\circ\gamma_s^n(0)=((m_1,-e_1),(m_2,-e_2),\ldots,(m_n,-e_n))$$
 
  and, according to the last point of Proposition \ref{prop calcul delta gamma},
  $$\delta_s^{m_1-1}\circ\beta_2^s\circ\ldots\circ\beta_n^s\circ\gamma_s^n\circ \delta_s(0)=((m_1,e_1),(m_2,e_2),\ldots,(m_n,e_n))$$
 
  This ends the proof.
 
\end{proof}

\begin{deff}
For $x\in\{o,s\}$, use the notation $\theta_x(1)=\delta_x$ and $\theta_x(-1)=\gamma_x$.

For any $(m,e)=((m_1,e_1),\ldots,(m_n,e_n))$, there exists a unique $\e=(\e_1,\ldots,\e_u)\in \mathcal S$ such that $$\theta(\e_u\e_{u-1})\circ\ldots\circ\theta(\e_2\e_1)\circ\theta(\e_1)=\beta^o_1\circ\ldots\circ\beta^o_n\circ\gamma_o^n$$
We write $t'_o(m,e)$ for this element of $\mathcal S$.

For any $(m,e)=((m_1,e_1),\ldots,(m_n,e_n))$, there exists a unique $\e=(\e_1,\ldots,\e_u)\in \mathcal S$ such that
$$\theta(\e_u\e_{u-1})\circ\ldots\circ\theta(\e_2\e_1)\circ\theta(\e_1)=\delta_s^{m_1-1}\circ \beta_2^s\circ\ldots\circ\beta_n^s\circ\gamma_s^n\circ\delta_s^{\eta_1^s}$$
We write $t'_s(m,e)$ for this element of $\mathcal S$.

\end{deff}

\begin{rem}
  If $\rho$ has the same type as $G$, the map $t_s$ is a section of $\bar \varphi$ and if $\rho$ has a different type from $G$, the map $t_o$ is a section of $\bar \varphi$.
\end{rem}

The sections we built imply the following theorem:

\begin{theorem}\label{thm im phi=T}
  The image of $\varphi$ is $\mathcal T$.
\end{theorem}

\subsection{Definition of $\Temp^d$}
We recall that $\rho$ is fixed. We let $x$ be $o$ when $\rho$ and $G$ are of different types and $x=s$ when $\rho$ and $G$ are of the same type.

\begin{deff}
  We define $\mathcal S_{red}=\mathcal E(\Temp_{red})$. Notice that the surjective map $(\cdot)^{red}:\Temp\to \Temp_{red}$ induces a surjective map $\mathcal S\to \mathcal S^{red}$; we still denote this by $(\cdot)^{red}$.
\end{deff}

\begin{deff}\label{def alpha tempd}
Take $(\e_1,\ldots,\e_n)\in \mathcal S_{red}$ and define $$\check {\mathcal O}(\e)=\{1\le i\le n-1~| ~\e_i=\e_{i+1}\}\cup\{n \}$$
Now, write $\check {\mathcal O}(\e)=\{j_1<\ldots<j_{r}<j_{r+1}=n\}$; we define
$$\alpha(\e)=(2(j_1-\delta_{x=o}\delta_{\e_1=1}),j_2-j_1-1,\ldots,j_r-j_{r-1}-1,n-j_r-1)\in \mathbb N^{r+1}$$
If $r=0$, we define $\alpha(\e)=(2(n-\delta_{x=o}\delta_{\e_1=1}))$. Now, take $\e\in\mathcal S$; we define $\check {\mathcal O}(\e)=\check {\mathcal O}(\e^{red})$ and $\alpha(\e)=\alpha(\e^{red})$.
\end{deff}

Write $\alpha(\e)=(\alpha_1,\ldots,\alpha_{r+1})$; remark that if $\e_1=-\e_2$, then $\alpha_{r+1}>0$.

\begin{deff}
  Take $\e\in\mathcal S$ and write $\alpha(\e)=(\alpha_1,\ldots,\alpha_{r+1})$ with $r\ge 0$. We say that $\e$ is decreasing if $r=0$ or $r\ge 1$ and $\alpha_{i+1}\le \alpha_i$ and, if $x=s$, $\alpha_1>\alpha_2$. We write $\mathcal S^d$ for the set of decreasing $\e\in\mathcal S$. We write $\Temp^d=\mathcal E^{-1}(\mathcal S^d)$ and call it the set of decreasing tempered data. Finally, write $\ATemp^d\subset \ATemp$ for the dual of $\Temp^d$.
\end{deff}

\begin{rem}\label{rem, e, O, alpha}
 Take $\e,\e'\in\mathcal S_{red}$ such that $\e_1=\e'_1$; then:
 $$\e=\e'\Leftrightarrow \check {\mathcal O}(\e)=\check {\mathcal O}(\e')\Leftrightarrow \alpha(\e)=\alpha(\e')$$
\end{rem}

\begin{deff}
  We define $\Temp^d_{red}:=\Temp^d\cap\Temp_{red}=(\Temp^d)^{red}$, $\ATemp^d_{red}$ its dual, and $\mathcal S^d_{red}=(\mathcal S^d)^{red}$.
\end{deff}

In the following, we will give a direct formula to compute the dual of an element of $\Temp^d$.

\subsection{The tempered part of the dual of a decreasing tempered data}

We start with a remark.

\begin{rem}
We have $t_x(\m,\e)\in\Temp^d$ for any $(\m,\e)\in\mathcal T$. This implies that $\phi$ surjects $\Temp^d$ on $\mathcal T$. Notice that if $x=o$ we have $t_o(\m,\e)\in \Temp^d_{red}$.
\end{rem}

\begin{deff}\label{def sigmao} Write $x=o$ if $\rho$ is of the same type as $G$ and write $x=s$ otherwise. Take $(\e_1,\ldots,\e_n)\in \mathcal S^d_{red}$, write $(\alpha_1,\ldots,\alpha_{r+1})=\alpha(\e_1,\ldots,\e_n)$, define:
$$\sigma_o(\e_1,\ldots,\e_n)=f^{-1}(\Gamma_o(\alpha_1/2,\alpha_2,\ldots,\alpha_{r+1})(0))\in \Temp$$
and 
$$\sigma_s(\e_1,\ldots,\e_n)=f^{-1}(\Gamma_s(\alpha_1/2-1,\alpha_2,\ldots,\alpha_{r+1})((1,\e_1))\in \Temp$$

\end{deff}

\begin{prop}\label{prop temp part tempd}
Let $(\m,\e)\in \Temp^d$. Write $(\e_1,\ldots,\e_n)=\mathcal E(\m^{red},\e^{red})$, we have:
$$\AD(\m,\e)_{temp}=\sigma_x(\e_1,\ldots,\e_n)$$

\end{prop}

\begin{proof}

Thanks to Proposition \ref{thm red} we can assume that $(\m,\e)$ is reduced.  As before, write $\theta_x(-1)=\gamma_x$ and $\theta_x(1)=\delta_x$.  We have:
$$f(\AD(\m,\e)_{temp})=\theta_x(\e_n\e_{n-1})\circ\ldots\circ\theta_x(\e_2\e_1)\circ \theta_x(\e_1)(0)$$
This is clear by induction according to Proposition \ref{prop induction temp part}. For the initialization see that when $x=o$ we have for any centered segment $\sigma$:
$$
 f(\phi(\T_{\sigma,\eta}(\varnothing))) = \left\{
    \begin{array}{ll}
        \delta_o(0)=0 & \mbox{if } \eta=1  \\
        \gamma_o(0)=(1,-1) & \mbox{if } \eta=-1  
    \end{array}
\right.
$$
and if $x=s$, for any centered segment $\sigma$, this time:

$$
 f(\phi(\T_{\sigma,\eta}(\varnothing))) = \left\{
    \begin{array}{ll}
        \gamma_s(0)=(1,1) & \mbox{if } \eta=1  \\
        \delta_s(0)=(1,-1) & \mbox{if } \eta=-1  
    \end{array}
\right.
$$
This ends the initialization.

Finally remark that
$$\theta_x(\e_n\e_{n-1})\circ\ldots\circ\theta_x(\e_2\e_1)\circ \theta_x(\e_1)(0)=f(\sigma_x(\e_1,\ldots,\e_n))$$

\end{proof}

Then, by combining this proposition with Proposition \ref{prop calcul delta gamma}, we get a direct formula to compute the tempered part of the dual of any element of $\Temp^d$. In the following we will compute the non-tempered part.

\subsection{On the increasing initial sequence}\label{subsect initial sequence}

In this subsection, we compute the increasing initial sequence $\Theta$ of $\AD(\m,\e)$ for any $(\m,\e)\in \Temp^d$ with the assumption that both largest segments in $\m$ have opposite signs and $\alpha_{r+1}<\alpha_r$ where $(\alpha_1,\ldots,\alpha_{r+1})=\alpha(\mathcal E(\m^{red},\e^{red}))$. It will be useful to compute the non-tempered part of duals of elements of $\AD(\m,\e)$.

To do that, we introduce some notation:

\begin{deff}
  Take $n\ge 1$ and $-1\le k\le n/2$ two integers and $\la_1>\ldots>\la_n>0$ some integers that are either all even or all odd. We define the multisegment $\m_k(\la_1,\ldots,\la_n)$ as the only symmetrical multisegment containing only non-centered segments and such that $\m_k^{\ge 0}(\la_1,\ldots,\la_n)=s_1+\ldots+s_n$ is the only multisegment verifying the following properties:
  \begin{enumerate}
    \item $(\ell(s_i))_{1\le i\le n}$ is a decreasing sequence
    \item $e(s_{i+1})=e(s_i)+1$ for any $1\le i\le n-1$ and $e(s_n)=(\la_1-1)/2$
    \item $\m_k(\la_1,\ldots,\la_n)$ has $(\la_{i}-\la_{i+1})/2-1$ segments of length $i$ for any $1\le i\le n-1$, and also $\lceil \la_n/2\rceil+k$ segments of length $n$.
  \end{enumerate}
\end{deff}
Since $k\le n/2$, we verify that all the segments in $\m_k(\la_1,\ldots,\la_n)$ have a strictly positive center.
\begin{ex}
For any $l$ even, we have $\m_0(l)=\m(0,l)$ and $\m_{-1}(l)=\m(2,l)$, and for any $l$ odd, we have $\m_0(l)=\m(1,l)$ as introduced in Subsection \ref{subsec some def}. To give an explicit example, we have:
$$\m_1^{\ge 0}(10,6)=[0.5,1.5]+[1.5,2.5]+[2.5,3.5]+[4.5,4.5]$$
\end{ex}

\begin{prop}\label{prop initial sequence} Let $(\m,\e)\in\Temp^d$ and $\Theta$ be the increasing initial sequence attached to $\AD(\m,\e)$. Write $\m^{red}=s_1+\ldots+s_n$ with $s_1\subseteq\ldots\subseteq s_n$, $\lambda_i=\ell(s_{n-i+1})$ for any $1\le i\le n$, and $(\alpha_1,\ldots,\alpha_{r+1})=\alpha(\mathcal E(\m^{red},\e^{red}))$. Suppose $\e(s_n)\neq \e(s_{n-1})$.

First, assume that $\rho$ and $G$ are of different types; recall that in this case, $\sigma_i=[-i+1/2,i-1/2]$ for any integer $i\ge 0$.

Suppose that $r=0$; then, if $\e(s_1)=-1$:
\begin{itemize}
  \item $\Theta_i=\sigma_i^{=0}$ for any $1\le i \le n$.
 
  \item $\sum_{i\ge n+1}\Theta_i=\m^{\ge 0}_{-1}(\la_1,\ldots,\la_n)$
  \item If $(\m,\e)$ is reduced, we have $\AD(\m,\e)=\sum_{i=1}^n\sigma_i+\m_0(\la_1,\ldots,\la_n)$ with $\sigma_i$ having sign $(-1)^{i}$ for any $1\le i\le n$.
\end{itemize}

Suppose that $r=0$; then, if $\e(s_1)=1$:
\begin{itemize}
  \item $\Theta_i=\sigma_i^{=0}$ for any $1\le i \le n-1$.
  \item $\sum_{i\ge n}\Theta_i=\m^{\ge 0}_0(\la_1,\ldots,\la_n)$
  \item If $(\m,\e)$ is reduced, we have $\AD(\m,\e)=\sum_{i=1}^{n-1}\sigma_i+\m_1(\la_1,\ldots,\la_n)$ with $\sigma_i$ having sign $(-1)^{i}$ for any $1\le i\le n-1$.

\end{itemize}

Suppose that $r\ge 1$ and that $\alpha_{r+1}<\alpha_{r}$. Write $d=\alpha_{r+1}$ and $j=\lceil d/2\rceil$; then:
\begin{itemize}
  \item $\Theta_i=\sigma_i^{=0}$ for any $1\le i \le j-1$.
  \item $\Theta_j=\sigma_j^{\ge 0}$ if $d$ is odd and $\Theta_j=\sigma_j^{=0}$ if $d$ is even.
  \item $\sum_{i\ge j+1}\Theta_i=\m^{\ge 0}_{\lfloor d/2\rfloor}(\la_1,\ldots,\la_{d+1})$
 
  \item Take an integer $i$ such that $\ell(\Theta_i)=\ell(\Theta_{i+1})+1$; then there are no segments in $\AD(\m,\e)$ of the shape $[b(\Theta_i)+1, x]$ with $x\ge e(\Theta_{i})+2$.

\end{itemize}

Now, assume that $\rho$ and $G$ are of the same type; recall that in this case, $\sigma_i=[-i+1,i-1]$ for any integer $i\ge 0$.

Suppose that $r=0$; then:
\begin{itemize}
  \item $\Theta_i=\sigma_i^{=0}$ for any $1\le i \le n$.
  \item $\sum_{i\ge n+1}\Theta_i=\m^{\ge 0}_{-1}(\la_1,\ldots,\la_n)$
  \item If $(\m,\e)$ is reduced, we have $\AD(\m,\e)=\sum_{i=1}^n\sigma_i+\m_0(\la_1,\ldots,\la_n)$ with $\sigma_i$ having sign $(-1)^{i+1}\e(s_1)$ for any $1\le i\le n$.
\end{itemize}

Suppose that $r\ge 1$ and that $\alpha_{r+1}< \alpha_{r}$. Write $d=\alpha_{r+1}$ and $j=\lfloor d/2\rfloor+1$; then:
\begin{itemize}
  \item $\Theta_i=\sigma_i^{=0}$ for any $1\le i \le j-1$.
  \item $\Theta_j=\sigma_j^{\ge 0}$ if $d$ is even and $\Theta_j=\sigma_j^{=0}$ if $d$ is odd.
  \item $\sum_{i\ge j+1}\Theta_i=\m^{\ge 0}_{\lfloor (d-1)/2\rfloor}(\la_1,\ldots,\la_{d+1})$
 
  \item Take an integer $i$ such that $\ell(\Theta_i)=\ell(\Theta_{i+1})+1$; then there are no segments in $\AD(\m,\e)$ of the shape $[b(\Theta_i)+1, x]$ with $x\ge e(\Theta_{i})+2$.

\end{itemize}

\end{prop}

\begin{proof}
According to Theorem \ref{thm red}, it suffices to prove the result when $(\m,\e)$ is reduced, so we assume $(\m,\e)\in\Temp_{red}$.  When $r=0$, the first point comes from the formula of Proposition \ref{prop temp part tempd}, and when $r\neq 0$, the first two points come from this same formula (using $\alpha_r>d$ for the second one). We prove the other points by induction. Define $(\m',\e')$ with $\m'=s_1+\ldots+s_{n-1}$ and $\e'$ the restriction of $\e$. Since $\e(s_n)\neq \e(s_{n-1})$, we have $(\m^{red},\e^{red})=\T_{s_n,\e(s_n)}((\m')^{red},(\e')^{red})$; we can do that by induction. Write $\Theta'=\Theta(\AD(\m',\e'))$ and $\alpha'=\alpha(\AD(\m',\e'))$. We have $\AD(\m,\e)=\CT_{s_n,\e_n}(\AD(\m',\e'))$. Notice that since $(\m,\e)$ is reduced and $\alpha_{r+1}\ge 1$, we have $\alpha'=(\alpha_1,\ldots,\alpha_r,\alpha_{r+1}-1)$. We assume that the result is true for $(\m',\e')$, so we know what $\Theta'$ is, and then we apply the formula for $\CT_{s_n,\e_n}$.

First, if $r=0$ and $\e(s_1)=-1$: assume $n\ge 2$, then we also have $r'=0$, $\e'(s_1)=-1$, and we have $\alpha_1'=n-1$. It suffices to remark that:
$$\sum_{s\in\m^{\ge0}_{-1}(\la_2,\ldots,\la_n)}s^++\m^{\ge 0}(\la_2,\la_1)=\m_{-1}^{\ge0}(\la_1,\ldots,\la_n)$$
Then, we conclude by definition of $\CT_{s_n,\e_n}$.

The case $r=0$ and $\e(s_1)=1$ is similar when $n\ge 2$ since we have:
$$\sum_{s\in\m^{\ge0}_0(\la_2,\ldots,\la_n)}s^++\m^{\ge 0}(\la_2,\la_1)=\m_0^{\ge0}(\la_1,\ldots,\la_n)$$
In both cases, we need to initialize our induction. For any $i\ge 0$, we have $\AD(\sigma_i,-1)=[-0.5,0.5]+\m(2,2i)=\sigma_1+\m_{-1}(2i)$ with $[-0.5,0.5]$ having sign $-1$, and $\AD(\sigma_i,1)=\m(0,2i)$. This proves the case $n=1$; thus, we are done with the case $r=0$.

The case $r=0$ for $\rho$ of the same type as $G$ is completely analogous.\\

Now, suppose that $r\ge 1$. First, we treat the case $d$ even (and so $d'=d-1$ is odd). Then we already know by Proposition \ref{prop temp part tempd} that $\Theta_i=\sigma_i^{=0}$ for any $1\le i\le j$. Now, either we have $d=\alpha_1$ and then $j=\alpha_1/2$ so $\sigma_{j+1}$ is not in $\AD(\m,\e)$, or $d<\alpha_1$ and then $\sigma_{j+1}\in\AD(\m,\e)$ but since $d<\alpha_{r+1}$, we have
$\e(\sigma_{j+1})=(-1)^j$. In both cases, we deduce that $\Theta_{j+1}$ is not a centered segment. Observe that $j=j'$, and since $d'$ is odd, we have $\Theta'_j=\sigma_j^{\ge 0}$, so $\sigma_j^+$ is in $\AD(\m,\e)$. This is the smallest possible segment ending in $j+1/2$ and that is non-centered. We deduce that $\Theta_{j+1}=\sigma_j^+$. Now, remark that:
$$\sum_{s\in \m^{\ge 0}_{d/2-1}(\la_2,\ldots,\la_{d+1})}s^++\m^{\ge 0}(\la_2,\la_1)\in\AD(\m,\e)$$
This multisegment is just the non-centered segments of $\Theta'$ that we extended. Notice that this multisegment is $\m^{\ge 0}_{d/2-1}(\la_1,\ldots,\la_{d+1})$. The first segment in $\m^{\ge 0}_{d/2-1}(\la_1,\ldots,\la_{d+1})$ ends in $j+3/2=e(\Theta_{j+1})+1$ and is of the same length, so it has to be $\Theta_{j+2}$. Indeed, it is easy to see that if $s\in\Theta$, then ${}^-s^+\in\Theta$. Write $\m^{\ge0}_{d/2-1}(\la_1,\ldots,\la_{d+1})=t_{j+2}+\ldots+t_l$ with $e(t_i)+1=e(t_{i+1})$ for any $j+2\le i\le l-1$. We already proved that $t_{j+2}=\Theta_{j+2}$. We want to prove $\Theta_i=t_i$ for any $j+2\le i\le l$. Take $j+2\le i\le l-1$. Notice that by definition, we either have $t_{i+1}={}^-t_i^+$ and then $\ell(t_i)=\ell(t_{i+1})$, or $t_{i+1}={}^{--}t_i^+$. Suppose $t_i=\Theta_i$; then, if $t_{i+1}={}^-t_i^+$, we automatically have $\Theta_{i+1}=t_{i+1}$. If $t_{i+1}={}^{--}t_i^+$, we need to prove that ${}^-t_i^+\notin \AD(\m,\e)$. The only two possibilities to have ${}^-t_i^+\notin \AD(\m,\e)$ are $\Theta_{i}'={}^-(\Theta'_{i-1})^+$ or ${}^-t_i^+\in\AD(\m',\e')$ and it is not $\Theta_{i+1}$. The first thing never happens because, as we saw, $t_i=\Theta_{i-1}^+$ and $t_{i+1}=\Theta_{i}^+$, so we would have $t_{i+1}={}^-t_i^+$. The second case is impossible by our induction hypothesis. So, by induction, we deduce that $t_i=\Theta_i$ for any $j+1\le i\le l$. Then we have:
$$\sum_{i\le j+1}\Theta_i=\sigma_j^++\m^{\ge 0}_{d/2-1}(\la_1,\ldots,\la_{d+1})=\m^{\ge 0}_{d/2}(\la_1,\ldots,\la_{d+1})$$
The last point is true because if such a segment existed in $\AD(\m,\e)$ for the integer $i$, we would get a similar segment in $\AD(\m',\e')$ for $i-1$, which is impossible since we assumed the result was true for $(\m',\e')$.

The case of $\rho$ and $G$ of the same type and $d$ odd is very similar to the one we just treated.

The cases we need to treat are:
\begin{enumerate}
  \item $\rho$ and $G$ of the same type and $d$ even
  \item $\rho$ and $G$ of different types and $d$ odd
\end{enumerate}

They are both very similar; we treat the first one. So let us assume $\rho$ and $G$ are of different types and $d\ge 1$ is even.

We already know by Proposition \ref{prop temp part tempd} that $\Theta_i=\sigma_i^{=0}$ for any $1\le i\le d/2=j-1$. We know in this case that we have $d\le \alpha_2<\alpha_1$, but since $\alpha_1$ and $d$ are both even, we have $d\le \alpha_1-2$ and so $j=d/2+1\le \alpha_1/2$. Then again, by Proposition \ref{prop temp part tempd}, we have $m_{\AD(\m,\e)}(\sigma_{j})=2$ with sign $(-1)^{j+1}\e_1$, so $\Theta_{j}=\sigma_j^{\ge 0}$. For the non-tempered part, this is completely similar to the case we already treated ($r\ge 1$, $d$ even, $x=o$).

\end{proof}

\begin{rem}
  By combining this result with Theorem \ref{thm red} we get a formula to compute directly $\AD(\m,\e)$ for any $(\m,\e)\in \Temp$ such that $\alpha=(\alpha_1)$ (equivalently, $(\m,\e)$ never have an even number of consecutive segment sharing the same sign). This direct formula recovers the well-known fact that the supercuspidal representations are self dual.
\end{rem}



\subsection{A direct formula for $\Temp^d$}\label{subsect non temp part of the dual}

Using Proposition \ref{prop initial sequence}, we give a formula to compute the non-tempered part of the dual of an element of $\Temp^d$. It finally gives us a direct formula to compute the dual of any decreasing datum $(\m,\e)\in \Temp^d$.

Take $(\e_1,\ldots,\e_n)\in\mathcal S^d_{red}$ and write $\check {\mathcal O}=\check {\mathcal O}(\e_1,\ldots,\e_n)=\{j_1<\ldots<j_r<n=j_{r+1}\}$. Finally, write $\alpha(\e_1,\ldots,\e_n)=(\alpha_1,\ldots,\alpha_{r+1})$. For any $1\le i\le r+1$, write $j_i'=n-j_i+1$; see that $j_{r+1}'=1$.

Define $\Lambda(\e_1,\ldots,\e_n)$ as the following set:
$$\{(\la_1,\ldots,\la_n)\in(2\mathbb N)^n\cup (2\mathbb N+1)^n~|~\la_{j_i'}>\ldots>\la_{j'_{i-1}-1}\ge \la_{j_{i-1}'} ,~\forall i\in\llbracket 2,r\rrbracket, \mbox{ and }\la_{j_1'}>\ldots>\la_1 \}$$
By sending $(\m,\e)$ such that $\mathcal E(\m,\e)=(\e_1,\ldots,\e_n)$ to the decreasing sequence of its lengths, we get a bijection:
$$\ell_{\e_1,\ldots,\e_n}:\mathcal E^{-1}(\e_1,\ldots,\e_n)\stackrel{\sim}\longrightarrow\Lambda(\e_1,\ldots,\e_n)$$
Now, take $\la=(\la_1,\ldots,\la_n)\in \Lambda(\e_1,\ldots,\e_n)$;
define:

\[\m_{\e_1,\ldots,\e_n}(\la) = \left\{
  \begin{array}{ll}
    \m_{-\delta_{\e_1=-1}}(\la_{j_1'},\ldots,\la_n)+\sum_{i=1}^{r}\m_{\lfloor \alpha_i/2\rfloor}(\la_{j'_i},\ldots,\la_{j'_{i-1}-1}) & \mbox{if } (\la_1,\ldots,\la_n)\in(2\mathbb N)^n \\
    \m_{-1}(\la_{j_1'},\ldots,\la_n)+\sum_{i=1}^{r}\m_{\lfloor (\alpha_i-1)/2\rfloor}(\la_{j'_i},\ldots,\la_{j'_{i-1}-1}) & \mbox{if } (\la_1,\ldots,\la_n)\in(2\mathbb N+1)^n
  \end{array}
\right.\]

\begin{theorem}\label{thm explicit tempd}
Take $(\m,\e)\in\Temp^d$; write $\m^{red}=s_1+\ldots+s_n$ with $s_1\subseteq \ldots\subseteq s_n$, let $\la_i=\ell(s_n-i+1)$ for any $1\le i\le n$, and $\e_i=\e^{red}(s_i)$ for any $1\le i\le n$. Write $\la=(\la_1,\ldots,\la_n)$, which is an element of $\Lambda(\e_1,\ldots,\e_n)$, and $(\alpha_1,\ldots,\alpha_{r+1})=\alpha(\e_1,\ldots,\e_n)$. Then we have:

$$\AD(\m,\e)=\sigma_x(\e_1,\ldots,\e_n)+\m_{\e_1,\ldots,\e_n}(\la)+\sum_{s\in\m-\m^{red}}\sum_{x\in s}[x,x]$$
with the first term of the sum being the tempered part, and the rest of it the non-tempered part.

\end{theorem}

\begin{proof}
  According to Theorem \ref{thm red} and Proposition \ref{prop temp part tempd}, it suffices to assume that $(\m,\e)$ is reduced and show that:
  $$\AD(\m,\e)_{non-temp}=\m_{\e_1,\ldots,\e_n}(\la)$$
  with $\la=(\la_1,\ldots,\la_n)\in \Lambda(\e_1,\ldots,\e_n)$. We do this by induction. Remove from $(\m,\e)$ its largest segment and write $(\m',\e')$ for the result. Now, assume that the proposition is true for $(\m',\e')$, that is $\AD(\m',\e')_{non-temp}=\m_{\e_1,\ldots,\e_{n-1}}(\la')$ with $\la'=(\la_1',\ldots,\la_{n-1}')=(\la_1,\ldots,\la_n)$. We have: $$\AD(\m,\e)=\CT_{s_n,\e_n}(\AD(\m',\e'))=\CT_{s_n,\e_n}((\AD(\m',\e'))_{temp}+\m_{\e_1,\ldots,\e_{n-1}}(\la'))$$

  First, assume that $\e_n=\e_{n-1}$, so $\alpha_{r+1}=0$ and $\alpha(\m',\e')=(\alpha_1,\ldots,\alpha_r)$; then we compute first that: $$\m_{\e_1,\ldots,\e_{n}}(\la)=\m_{\e_1,\ldots,\e_{n-1}}(\la')+\m_{\e_n}(\la_1)=\m_{\e_1,\ldots,\e_{n-1}}(\la')+\m_{0} (\la_1)=\m_{\e_1,\ldots,\e_{n-1}}(\la')+\m(0,\la_1)$$

  where $\m(0,\la_1)$ is defined in Subsection \ref{subsec some def}.
  But by definition of $\CT_{s_n,\e_n}$, since $\e_n=\e_{n-1}$, we have $\CT_{s_n,\e_n}(\AD(\m',\e'))=\AD(\m',\e')+\m(0,\ell(s_n))$,
  and thus, since $\ell(s_n)=\la_1$, we get:
\begin{align*}
  \AD(\m,\e)_{non-temp}=&(\CT_{s_n,\e_n}((\AD(\m',\e'))_{temp}+\m_{\e_1,\ldots,\e_{n-1}}(\la')))_{non-temp}\\
  =&\m_{\e_1,\ldots,\e_n}(\la')+\m(0,\la_1)\\
  =&\m_{\e_1,\ldots,\e_n}(\la)
\end{align*}
Now, we assume that $\e_n\neq \e_{n-1}$. So $\alpha(\m',\e')=(\alpha_1,\ldots,\alpha_r,\alpha_{r+1}-1)$. Since $(\m,\e)\in\Temp^d_{red}$, we have: $\alpha_r\ge \alpha_{r+1}>\alpha_{r+1}-1$, so we can apply Proposition \ref{prop initial sequence}. Write $d=\alpha_{r+1}$.

First, assume that $\rho$ is of a different type from $G$. Then write $j=\lceil d/2\rceil$. Notice that $d+1=j'_r$.

We have:
\begin{align*}
  \m_{\e_1,\ldots,\e_n}(\la)=&\m_{-\delta_{\e_1=-1}}(\la_{j_1'},\ldots,\la_n)+\sum_{i=1}^{r-1}\m_{\lfloor \alpha_i/2\rfloor}(\la_{j'_i},\ldots,\la_{j'_{i-1}-1})+\m_{j}(\la_1,\ldots,\la_{d+2})\\
  \m_{\e_1,\ldots,\e_{n-1}}(\la')=&\m_{-\delta_{\e_1=-1}}(\la_{j_1'},\ldots,\la_n)+\sum_{i=1}^{r-1}\m_{\lfloor \alpha_i/2\rfloor}(\la_{j'_i},\ldots,\la_{j'_{i-1}-1})+\m_{\lfloor d/2\rfloor}(\la')
\end{align*}
and
$$\m_{j}^{\ge 0}(\la_1,\ldots,\la_{d+2})=\delta_{d\notin2\mathbb Z}\cdot \sigma_j^++\sum_{s\in\m_{\lfloor d/2\rfloor}(\la')}s^++\m^{\ge 0}(\la_2,\la_1)$$
But according to Proposition \ref{prop initial sequence}, we have
$$\Theta(\m',\e')^{\ge 0}=\delta_{d\notin2\mathbb Z}\cdot\sigma_j^{\ge 0}+\m_{j}^{\ge 0}(\la_2,\ldots,\la_{d+1})$$

Consequently, by definition of $\CT_{s_n,\e_n}$ when $\e_n\neq \e_{n-1}$, we get
\begin{align*}
  &\AD(\m,\e)_{non-temp}^{\ge 0}\\
  &=(\CT_{s_n,\e_n}((\AD(\m',\e'))_{temp}+\m_{\e_1,\ldots,\e_{n-1}}(\la')))_{non-temp}^{\ge 0}\\
  &=\m^{\ge 0}_{-\delta_{\e_1=-1}}(\la_{j_1'},\ldots,\la_n)+\sum_{i=1}^{r-1}\m^{\ge 0}_{\lfloor \alpha_i/2\rfloor}(\la_{j'_i},\ldots,\la_{j'_{i-1}-1})+\delta_{d\notin2\mathbb Z}\cdot \sigma_j^++\sum_{s\in\m_{\lfloor d/2\rfloor}(\la')}s^++\m^{\ge 0}(\la_2,\la_1)\\
  &=\m^{\ge 0}_{-\delta_{\e_1=-1}}(\la_{j_1'},\ldots,\la_n)+\sum_{i=1}^{r-1}\m^{\ge 0}_{\lfloor \alpha_i/2\rfloor}(\la_{j'_i},\ldots,\la_{j'_{i-1}-1})+\m^{\ge 0}_{j}(\la_1,\ldots,\la_{d+2})\\
  &=\m^{\ge 0}_{\e_1,\ldots,\e_n}(\la)
\end{align*}

So, by symmetry,
$$\AD(\m,\e)_{non-temp}=\m_{\e_1,\ldots,\e_n}(\la)$$

If $\rho$ is of the same type as $G$, the computation is completely analogous.
\end{proof}

Note that this result provides a direct and explicit construction of a large class of co-tempered data which exhausts all the possible tempered parts.

\printbibliography

\end{document}